\documentclass[english,pdftex,12pt]{article}

\usepackage{setspace}       %
\usepackage{amsfonts}       %
\usepackage{dsfont}
\usepackage{amsmath}        %
\usepackage{amssymb}        %
\usepackage{amsthm}         %
\usepackage{mathrsfs}       %
\usepackage[normalem]{ulem} %
\usepackage{bm}             %
\usepackage{bbm}            %
\usepackage{color}          %
\usepackage{xcolor}         %
\usepackage{mathtools} %
\usepackage[T1]{fontenc}
\usepackage{lmodern}   %
\usepackage{babel}
\usepackage{hyperref}       %
\normalfont
\DeclareFontShape{T1}{lmr}{bx}{sc} { <-> ssub * cmr/bx/sc }{}
\definecolor{MyDarkBlue}{rgb}{0.1,0.2,.65}
\hypersetup{linkcolor=black, citecolor=MyDarkBlue,
            urlcolor=black,colorlinks=true}   %

\usepackage[titletoc]{appendix}

\usepackage[shortlabels]{enumitem}		    %

\usepackage{graphicx}                   %
\usepackage{subfig}                     %
\usepackage{float}                      %
\usepackage{tikz}                       %
\usetikzlibrary{decorations.markings}
\usetikzlibrary{decorations.pathreplacing}
\pgfdeclarelayer{edgelayer}             %
\pgfdeclarelayer{nodelayer}             %
\pgfsetlayers{edgelayer,nodelayer,main} %
\tikzstyle{none}=[inner sep=0pt]        %

\numberwithin{equation}{section}

\usepackage[longnamesfirst]{natbib} %
\usepackage{todonotes}              %

\newtheoremstyle{mytheoremstyle} %
    {\topsep}                    %
    {\topsep}                    %
    {\itshape}                   %
    {}                           %
    {\bf}                %
    {.}                          %
    {.5em}                       %
    {}  %

\newtheoremstyle{mydefstyle} %
    {\topsep}                    %
    {\topsep}                    %
    {}                           %
    {}                           %
    {\bf}                %
    {.}                          %
    {.5em}                       %
    {}  %

\theoremstyle{mytheoremstyle}

\newtheorem*{lemma*}{Lemma}

\newtheorem{proposition}{Proposition}

\newtheorem*{theorem*}{Theorem}

\theoremstyle{mydefstyle}

\newtheorem*{axiom*}{Axiom}
\newtheorem*{comment*} {Comment}

\newcommand{\bC}{\mathbb{C}}

\newcommand{\bE}{\mathbb{E}}

\newcommand{\bR}{\mathbb{R}}

\newcommand{\bV}{\mathbb{V}}
\newcommand{\bW}{\mathbb{W}}

\newcommand{\cA}{\mathcal{A}}

\newcommand{\cC}{\mathcal{C}}
\newcommand{\cD}{\mathcal{D}}
\newcommand{\cE}{\mathcal{E}}
\newcommand{\cF}{\mathcal{F}}
\newcommand{\cG}{\mathcal{G}}

\newcommand{\cL}{\mathcal{L}}

\newcommand{\cQ}{\mathcal{Q}}

\newcommand{\cX}{\mathcal{X}}

\newcommand{\cZ}{\mathcal{Z}}

\newcommand{\beq}{ \begin{equation} }
\newcommand{\eeq}{ \end{equation} }
\newcommand{\bal}{ \begin{align*} }
\newcommand{\eal}{ \end{align*} }

\newcommand{\tr}{ \text{tr } }

\newcommand{\diverge}{\text{div }}

\makeatletter
\newcommand*\dotp{\mathpalette\dotp@{.5}}
\newcommand*\dotp@[2]{\mathbin{\vcenter{\hbox{\scalebox{#2}{$\m@th#1\bullet$}}}}}
\makeatother

\usepackage{afterpage}
\usepackage{mathrsfs}
\usepackage[margin=0.95in]{geometry}
\usepackage[ruled]{algorithm2e}
\usepackage{algorithmic}
\usepackage{tabularx}
\usepackage{makecell}
\usepackage{tablefootnote}

\usepackage{mathtools}        %
\usepackage{lineno}           %
\mathtoolsset{showonlyrefs}   %

\usepackage[normalem]{ulem}

\definecolor{ao(english)}{rgb}{0.0, 0.5, 0.0}

\let\OLDthebibliography\thebibliography
\renewcommand\thebibliography[1]{
  \OLDthebibliography{#1}
  \setlength{\parskip}{0.0pt}
  \setlength{\itemsep}{6.0pt plus 0.3ex}
}

\begin{document}

\title{Global Solutions to Master Equations for Continuous Time Heterogeneous Agent Macroeconomic Models
}
\author{
Zhouzhou Gu\footnote{Princeton, Department of Economics. Email: zg3990@princeton.edu},
Mathieu Lauri\`{e}re\footnote{Shanghai Frontiers Science Center of Artificial Intelligence and Deep Learning; NYU-ECNU Institute of Mathematical Sciences at NYU Shanghai. Email: mathieu.lauriere@nyu.edu},
Sebastian Merkel\footnote{University of Exeter, Department of Economics. Email: s.merkel@exeter.ac.uk},
Jonathan Payne\footnote{Princeton, Department of Economics. Email: jepayne@princeton.edu}
\footnote{We would like to thank Adam Rebei for outstanding research assistance throughout this project. We are also very grateful to the comments and discussion from Fernando Alvarez, Mark Aguiar, Yacine Ait-Sahalia, Marlon Azinovic, Johannes Brumm, Markus Brunnermeier, Rene Carmona, Wouter Den Haan, Mikhail Golosov, Goutham Gopalakrishna, Lars Hansen, Ji Huang, Felix Kubler, Greg Kaplan, Moritz Lenel, Ben Moll, Galo Nu\~{n}o, Thomas J. Sargent, Simon Scheidegger, Jes\'{u}s Fern\'{a}ndez-Villaverde, Yucheng Yang, Jan \v{Z}emli\v{c}ka, and Shenghao Zhu. 
We are also thankful for the comments from participants in the DSE 2023 Conference on Deep Learning for Solving and Estimating Dynamic Models, the CEF 2023 Conference, the IMSI 2023 Conference, the 11th Western Conference on Mathematical Finance, the Macro-finance Reading Group at USC Marshall, the BFI Workshop at Chicago, and the Princeton Civitas Seminar.
We thank Princeton Research Computing Cluster for their resources.}
}

\date{\today}
\maketitle{}

\vspace{-0.75cm}

\begin{abstract}
\noindent We propose and compare new global solution algorithms for continuous time heterogeneous agent economies with aggregate shocks.
First, we approximate the agent distribution so that equilibrium in the economy can be characterized by a high, but finite, dimensional non-linear partial differential equation.
We consider different approximations: discretizing the number of agents, discretizing the agent state variables, and projecting the distribution onto a finite set of basis functions.
Second, we represent the value function using a neural network and train it to solve the differential equation using deep learning tools.
We refer to the solution as an Economic Model Informed Neural Network (EMINN).
The main advantage of this technique is that it allows us to find global solutions to high dimensional, non-linear problems.
We demonstrate our algorithm by solving important models in the macroeconomics and spatial literatures
(e.g. \cite{Krusell1998}, \cite{khan2007inventories}, \cite{bilal2021solving}). \\

\noindent\emph{JEL:} C63, C68, E60, E27. \\
\noindent\emph{Keywords:} Heterogeneous agents, computational methods, deep learning, inequality, mean field games, continuous time methods, aggregate shocks, global solution.

\end{abstract}

\thispagestyle{empty}
\clearpage
\pagenumbering{arabic}

\section{Introduction}

Macroeconomists have great interest in studying models with heterogeneous agents and aggregate shocks.
Recent work has characterized equilibria in these models recursively in continuous time using ``master equations'' from the mean field game theory literature.
A major challenge with these recursive formulations is that the agent distribution of agents states becomes an aggregate state variable and so the state space becomes infinite dimensional.
There have been attempts to resolve this issue by using perturbations in the distribution space (e.g. \cite{bilal2021solving}, \cite{alvarez2023price}, \cite{bhandari2023perturbational}).
In this paper we use tools from the deep learning literature to characterize global solutions to the master equation. 
This necessitates constructing a finite dimensional approximation to the cross-sectional distribution of agent states.
Most existing deep learning approaches are in discrete time and replace the agent continuum by a finite collection of agents.
We develop and compare the three main approaches for approximating the distribution in continuous time: imposing a finite number of agents, discretizing the state space, and projecting onto a collection of finite basis functions.
For each approximation, we show how to characterize general equilibrium as a high but finite dimensional differential master equation and how to customize deep learning techniques to compute global numerical solutions to the differential equation.

We develop solution techniques for a class of continuous time dynamic, stochastic, general equilibrium economic models with the following features.
There is a large collection of price-taking agents who face uninsurable idiosyncratic and aggregate shocks.
Given their belief about the evolution of aggregate state variables, agents choose control processes to solve dynamic optimization problems.
When making their decisions, agents face financial frictions that constrain their behaviour, which potentially breaks ``aggregation'' results and makes the distribution of agent states an aggregate state variable.
Solving for the rational expectations equilibrium reduces to solving a ``master'' partial differential equation (PDE) that summarizes both the agent optimization behaviour (from the Hamilton-Jacobi-Bellman equation (HJBE)) and the evolution of the distribution (from the Kolmogorov Forward Equation (KFE)).
A canonical example of this type of environment is the continuous time version of the \cite{Krusell1998} model, which we (and many others) use to our solution.
 
Our solution approach approximates the infinite dimensional master equation by a finite, but high, dimensional PDE and then uses deep learning to solve the high dimensional equation.
We consider three different approaches for reducing the dimension of the master equation: the \emph{finite-agent}, \emph{discrete-state}, and \emph{projection} methods.
The finite agent method approximates the distribution by a large, finite number of agents.
The discrete state method approximates the distribution by discretizing the agent state space so the density becomes a collection of masses at grid points.
The projection method approximates the distribution by a linear combination of finitely many basis functions. 
Most deep learning macroeconomic papers are in discrete time and have focused on the finite agent approximation. One exception is \cite{huang2023probabilistic}, which uses a discrete-state approximation.
All of our distribution approximations preserve the full non-linearity of the model, which is a key difference to perturbation based approaches.

We solve the finite dimensional approximation to the master equation by adapting the Deep Galerkin Method (DGM) and the Physics Informed Neural Networks (PINN) method developed in the applied mathematics and physics literatures.
This involves approximating the value function by a neural network and then using stochastic gradient descent to train the neural network to minimise a loss function that summarizes the average error in the master equation. %
We calculate average errors by randomly sampling over points in the state space.
We refer to our approach as training Economic Model Informed Neural Networks (EMINNs). 

Although our deep learning approach is relatively simple to describe at an abstract level, there are many intricate implementation details.
A particularly important detail is choosing how to sample the states on which we evaluate the master differential equation error.
This choice is less relevant for discrete time approaches where the economy is simulated to calculate expectations and also for other continuous time deep learning papers that don't have high dimensional distributions as states.
We consider three approaches for sampling the distribution.
The first is \emph{moment sampling} that draws selected moments of the distribution and then samples agent distributions that satisfy the drawn moments.
The second is \emph{mixed steady state sampling} that samples random mixtures of steady state distributions at different fixed aggregate states.
The third is \emph{ergodic sampling} that samples by regularly simulating the model economy based on the current estimate of the model solution.
We find that moment sampling is most effective for the finite-agent approximation, a combination of mixed steady state sampling and ergodic sampling is most effective for the discrete state space approximation, and a combination of moment sampling and ergodic sampling is most effective for the projection approach.
Deep learning techniques are sometimes referred to as ``breaking the curse of dimensionality'' but this is not really true in the sense that we cannot train a neural network by showing it a dense set of distributions.
Instead, a key ``art'' to training a neural network is to sample from an intelligently chosen subspace that helps the neural network to learn the equilibrium functional relationships in an economically interesting part of the state space.

In Section \ref{sec:krusell_smith}, we illustrate our techniques by solving a continuous time version of \cite{Krusell1998} using all three of our methods.
We take this model as a ``test case'' because there are well established approximate solutions using traditional techniques to which we can compare.
We show that our methods generate solutions with a low master equation error that produce very similar simulations to contemporary approaches such as \cite{fernandez2023financial}.
In addition, for the version of the model without aggregate shocks (the \cite{Aiyagari1994} model), we show that we can match the finite difference solution to high accuracy and, for the finite-agent approximation, solve for transition dynamics without a shooting algorithm.

Although all our methods are ultimately effective for solving the \cite{Krusell1998} model, we find they have different strengths and weaknesses.
One point of difference is how much information about the model solution is required for sampling.
We find that only the finite agent method can be solved without ergodic sampling, which makes its solution more robust away from the ergodic mean.
A second point of difference is how large and complicated the neural network needs to be.
The discrete state space requires a fine grid (approximately 200 points) while the finite agent method requires a moderate number of agents (approximately 40) and the projection method only requires 5 basis functions.
A third point of difference is the complication of the law motion for the parameters approximating the distribution.
For the finite agent method, this is simply the law of motion for agent idiosyncratic states.
However, for the discrete state method it is the law of motion for the mass between grid points and for the projection method it is the law of motion for projection coefficients, both of which are hard to compute.
How to balance these trade-offs depends upon how the distribution interacts with agent decision making and the complexity of the KFE.
Ultimately, our conclusions are:
\begin{enumerate}
    \item When only the mean of the distribution enters into the pricing equations, then the finite agent method with moment sampling is robust and powerful.
    \item When many aspects of the distribution are important for prices and the KFE does not involve complicated derivatives, then the discrete state method with mixed steady state and ergodic sampling is most successful. However, with complicated KFEs it is a difficult method.
    \item When it is clear how to customize the projection approach to the economic problem, then it can effectively solve the model with low dimensions.
\end{enumerate}

In Section \ref{sec:add_examples}, we solve two additional models that have more complicated master equations and illustrate the power of different methods: a model with heterogeneous firms (similar to a continuous time version of \cite{Khan2008}) and a spatial model (taken from \cite{bilal2021solving}).
We solve the firm model using a finite agent method and show that this technique copes well with the additional non-linearity that comes with pricing firm equity.
We solve the spatial model using a discrete set of locations (since finite agent and projection approximations infeasible) and show this technique is effective when different parts of the distribution affect different prices. \\

\noindent\emph{Literature Review}:
Our paper is part of the literature attempting to solve  heterogeneous agent continuous time (HACT) general equilibrium economic models with aggregate shocks.
Many researchers have shown that it can be useful to cast economic models in continuous time (e.g. \cite{achdou2022income}, \cite{Ahn2018}, \cite{kaplan2018monetary}).
Recent papers (e.g. \cite{bilal2021solving}) have made progress by characterizing equilibrium using the ``master equation'' approach developed in the mathematics literature (see \cite{LionsCDF}, \cite{Cardaliaguet2015}).
However, the continuous time literature has lacked a global solution technique for these models, particularly when written recursively in master equation form.
Instead, many of the existing techniques in the literature focus on perturbation in the distribution and other aggregate state variables.
Our main contribution is to offer a global solution to the master equation for HACT models with aggregate shocks.

Technically, our approach is part of a literature attempting to use deep learning to numerically characterize solutions to dynamic general equilibrium models with heterogeneous agents and aggregate shocks, which in the mathematics literature are mean-field games with common noise. 
There have been many papers attempting to solve discrete time versions of these models (e.g. \citet{azinovic2022deep}, \citet{han2021}, \citet{Maliar2021}, \citet{kahou2021}, \citet{bretscher2022ricardian}, \cite{germain2022deepsets}, \cite{azinovic2023economics}).
We are focused on deep learning for continuous time models, which has been the focus for the mathematics literature but less so for the economics deep learning literature.
For continuous time systems, there are two broad approaches for using deep learning: (i) a ``probabilistic'' approach that attempts to solve a system of stochastic differential equations and (ii) an ``analytic'' approach that attempts solve a system of partial differential equations.
Papers using the former approach follow \cite{han2018solving} and train the neural network to learn optimal controls across simulations of a discrete time approximation to the system of stochastic differential equations (e.g \cite{fouque2020deep}, \cite{min2021signatured}, 
\cite{Carmona2022},
\cite{germain2022numerical},
\cite{huang2023probabilistic}, \cite{huang2023breaking}, \cite{huang2024applications}).
This approach has many similarities  to the discrete time deep learning methods in economics, which also use simulations to train the optimal control problem.
Our paper follows the latter (i.e., analytic) approach, which involves minimizing the loss in the differential equation across randomly sampled points.
We view this as the ``true'' continuous time approach in the sense that it does not involve discretizing the time dimension in order to train the neural network.
To implement the analytic solution approach, we build on the DGM and PINN methods developed by \cite{Sirignano2018}, \citet{raissi2017physics}, and \cite{li2022} to solve problems in the physics and applied mathematics literatures.
The key difficulty we resolve is that none of the existing deep learning papers are directly applicable to solving our master equations.
Relative to the discrete-time and probabilistic approaches, we need to minimize loss in a differential equation, which requires developing new sampling approaches and resolving how to evaluate derivatives (rather than simulating to approximate expectations).
Relative to the DGM and PINN literatures, we have to work out how handle models with forward looking optimizing agents and market clearing conditions.
Some papers in the mean-field literature have attempted to address some these issues (e.g. \cite{al2022extensions,Carmona2021}) but they do not handle aggregate shocks (see e.g.~\cite{hu2022recent} for a recent survey).

The major challenge with using the analytic approach to solve master equations with a general state space is that we need to develop and sample from a finite dimensional representation of the distribution.
There are some existing papers in the mean field literature that work with low dimensional discrete state spaces (e.g. \cite{perrin2022generalization}, \cite{cohen2024deep}) and some papers in the economics literature that solve continuous time models without complicated distributions or with one-dimensional distribution approximations (e.g. \cite{duarte2024machine}, \cite{Gopalakrishna2021}, \cite{fernandez2023financial}, \citet{sauzet2021projection}, \cite{achdou2022simulating}, \cite{barnett2023deep}).
By contrast, we develop and compare a range of finite dimensional approximations: reducing to finite agents, discretizing the state space, and projection onto a finite set of basis functions.
An important contribution of our paper is to systematically explore the full set of distributional approximations and develop sampling approaches that are customized to those distributional approximations.

More generally, our technique (and other deep learning techniques) expand the traditional techniques in economic for solving heterogeneous agent models with aggregate shocks.
One traditional approach has been to fit a statistical approximation to the law of motion for the key aggregate state variables (e.g. \cite{Krusell1998}, \cite{DenHaan1997}, \cite{fernandez2023financial}).
As has been extensively discussed in the literature, this approach works well when the law of motion for the key state variables can be efficiently approximated as a function of key moments of the distribution (and so the economy is very close to permitting ``aggregation'').
By contrast, our approach can handle economies without near-aggregation results.
A second approach is to take a type of perturbation in the aggregate state and then solve the resulting problem with matrix algebra (e.g. \cite{Reiter2002}, \cite{Reiter2008}, \cite{Reiter2010}, \cite{Winberry2018}, \cite{Ahn2018}, \cite{auclert2021using}, \cite{bilal2021solving}, \citet{bhandari2023perturbational}).
By contrast, we solve the model globally and so can handle partial differential equations with global non-linearity.
A final approach is to take a low dimensional projection of the distribution (e.g. \cite{Prohl2017}, \cite{Schaab2020}). 
Our work is complementary to these papers in that it allows for more general, higher dimensional projections.\\

This paper is organised as follows. 
Section 2 describes the general economic environment that we will be studying and derives the master equation.
Section 3 describes the different finite dimensional approximations to the master equation.
Section 4 describes the solution approach.
Section 5 applies our algorithm to a continuous time version of \cite{Krusell1998}. Section 6 provides additional examples.
Section 7 concludes with practical lessons.

\section{Economic Model}\label{sec:genericModel}

In this section, we outline the class of economic models for which our techniques are appropriate. 
At a high level, in economics terminology, we are solving continuous time, general equilibrium models with a distribution of optimizing agents who face idiosyncratic and aggregate shocks.
In maths terminology, we are solving mean field games with common noise.

\subsection{Environment}\label{subsec:env}

\noindent\emph{Setting:} The model is in continuous time with infinite horizon. There is an exogenous one-dimensional\footnote{
    For ease of exposition, we restrict attention to models with one-dimensional aggregate shocks.
    The method is no different for the case with multi-dimensional aggregate shocks.
    }
aggregate state variable, $z_t$, which evolves according to:
\begin{align}
    dz_t ={}& \mu_z(z_t) dt + \sigma_z(z_t) dB_t^0, \quad z_0 \hbox{ given}, \label{eq:dstate_agg}
\end{align}
where $B_t^0$ denotes a common Brownian motion process. We let $\cF_t^0$ denote the filtration generated by $B_t^0$. \\

\noindent\emph{Agent Problem:} The economy is populated by a continuum of agents, indexed by $i \in I = [0,1]$. Each agent $i$ has an idiosyncratic state vector, $x^i \in \cX \subset \bR^{N_x}$, that follows:
\begin{align}
    dx_t^i ={}& \mu_x(c_t^i, x_t^i, z_t, q_t) dt + \sigma_{x}(c_t^i, x_t^i, z_t, q_t) dB_t^i + \varsigma_x(c_t^i, x_t^i, z_t, q_t) dJ_t^i, \quad x^i_0 \hbox{ given},  \label{eq:dstate_i}
\end{align}
where $c_t^i$ is a one-dimensional control variable chosen by the agent, $q_t \in \cQ$ is a collection of aggregate prices in the economy that will be determined endogenously in equilibrium, $B_t^i$ denotes an $N_B$-dimensional idiosyncratic Brownian motion process, $J_t^i$ denotes a 1-dimensional idiosyncratic Poisson jump process, $\mu_x(\cdot)$ is $N_x$--dimensional, $\sigma_x(\cdot)$ is an $N_x \times N_B$--dimensional matrix, and $\varsigma_x(\cdot)$ is $N_x$-dimensional.\footnote{
    We consider a one-dimensional control and $J_t^i$ process for notational simplicity.
    We do not consider the case where $dB_t^0$ directly impacts the evolution of idiosyncratic states. 
}
We let $\lambda(x,z)$ denote the rate at which Poisson jump shocks arrive given idiosyncratic state $x$ and aggregate state $z$.
We let $\cF_t^i$ denote the filtration generated by $B_t^i$, $J_t^i$, and $B_t^0$.

Each agent $i$ has a belief about the stochastic price process $\tilde{\mathsf{q}} = \{\tilde{q}_t : t \ge 0\}$ adapted to $\cF_t^0$. Given their belief, agent $i$ chooses their control process, $\mathsf{c}^i = \{ c_t^i : t \ge 0 \}$ adapted to $\cF_t^i$, to solve: 
\begin{align}
    V(x_0^i, z_0) ={}& \max_{\mathsf{c}^i \in \cC}\mathbb{E}_0 \left[ \int_0^{\infty} e^{-\rho t} u(c_t^i, z_t, q_t) dt \right]
    \ s.t. \ \eqref{eq:dstate_agg}, \ \eqref{eq:dstate_i}, \label{prob:agent_opt} 
\end{align}
where $\rho>0$ is a discount parameter, $u(c_t^i,z_t,g_t)$ is the flow benefit the agent gets and $\cC = \{ c_t^i \in \cC(x^i_t, z_t, q_t) : t \ge 0 \}$ is the set of admissible controls, where $\cC(x, z, q)$ denotes the set of possible actions for a player whose current state is $x$, when the aggregate state is $z$ and the prices are $q$.
This constraint set incorporates any ``financial frictions'' that restrict agent choices.
A classic example in economics is that $x_t^i$ represents agent wealth and the control must keep $x_t^i \ge \underline{x}$.
We assume that $u$ is increasing and concave in its first argument.
\\

\noindent\emph{Distributions and Markets:} Let $\cD(x_t^i|\cF_t^0)$ be the population distribution across $x_t^i$ at time $t$, for a given history $\cF_t^0$. 
We assume that $\cD(x_t^i|\cF_t^0)$ has a density $g_t$, where for continuous components of $x$ the density is with respect to the Lebesgue measure and for discrete components of $x$ the density is with respect to the counting measure.
This means we restrict attention to distributions without mass points in the continuous variables, which potentially requires us to ``smooth'' constraints in the model, as we discuss in Section \ref{subsec:master_eqn}. 
We further assume that $g_t \in \cG$, where $\cG$ is a subspace of the set of square integrable functions defined on $\cX$. 

We assume that the economy contains a collection of markets with price vector $q_t$ and market clearing conditions that allow us to solve for $q_t$ explicitly in terms of $z_t$ and $g_t$:%
\begin{align}
    q_t = Q(z_t, g_t), \quad \forall t \ge 0, \label{eq:market_clearing}
\end{align}
We assume that markets are incomplete in the sense that agents cannot trade claims directly (or indirectly) on their idiosyncratic shocks $dB_t^i$ and $dJ_t^i$. This means that the idiosyncratic shocks generate a non-degenerate cross sectional distribution of agent states. \\

\noindent\emph{Equilibrium:} 
Given an initial density $g_0$, an equilibrium for this economy consists of a collection of $\cF_t^0$--adapted stochastic processes, $\{\mathsf{g}, \mathsf{q}, \mathsf{z}\}$, and $\cF_t^i$--adapted stochastic processes, $\mathsf{c}^i$, for each $i \in I$, that satisfy the following conditions: (i) each agent's control process $\mathsf{c}^i$ solves problem \eqref{prob:agent_opt} given their belief that the price process is $\tilde{\mathsf{q}}$, (ii) the equilibrium prices $\mathsf{q}$ satisfy market clearing condition \eqref{eq:market_clearing}, and (iii) agent beliefs about the price process are consistent with the optimal behaviour of other agents in the sense that $\tilde{\mathsf{q}} = \mathsf{q}$.

\subsection{Recursive Representation of Equilibrium}\label{subsec:master_eqn}

We assume that there exists an equilibrium that is recursive in the aggregate state variables: $(z,g) \in \cZ \times \cG$.
In this case, beliefs about the price process can be characterized by beliefs about the evolution of the distribution $dg_t(x) = \tilde{\mu}_g(x, z_t, g_t) dt$ since the price vector $q$ can be expressed explicitly in terms of $(z,g)$. 
We assume that in this equilibrium, the value function of the agent takes the form $V: \cX\times\cZ\times\cG \rightarrow \bR$, $(x,z,g) \mapsto V(x,z,g)$, where $V$ is twice differentiable with respect to $x$ and $z$ and Frechet differentiable with respect to $g$.
\\ 

\noindent\emph{Hamilton Jacobi Bellman Equation (HJBE):}
In principle, the constraint $c_t^i \in \cC(x^i_t, z_t, q_t)$ restricts $x$ to a subspace $\cX^c \subset \cX$ and imposes the constraint on the recursive formulation:
\begin{align}
    \Psi(c,x,Q(z,g); \ V, D_x V, D_x^2 V) \ge 0, \qquad x \in \cX^c. \label{eq:Psi}
\end{align}
However, in order to help the neural network train the value function, throughout we replace the hard constraint on the set of admissible controls by a flow utility penalty $\psi(c,x,Q(z,g))$ that is larger the ``more'' that \eqref{eq:Psi} is violated (and where we have left the dependence on the value function implicit).
We provide explicit examples of $\psi$ in our applications in Section \ref{sec:krusell_smith} and Section \ref{sec:add_examples}.
Given a belief about the evolution of the distribution, $\tilde{\mu}_g$, the agent's value function, $V$, and optimal choice of control, $c$, jointly solve the HJBE: 
\begin{align}
    0 =\max_{c \in \cC(x,z,g)} \Big\{ {}&-\rho V(x,z,g) + u(c,z,Q(z,g)) + \psi(c,x,Q(z,g)) \\
    {}& + (\cL_x^{c;Q} + \cL_z + \cL_g^{\tilde{\mu}_g}) V(x,z,g) \Big\} \label{eq:HJBE-generic}
\end{align}
where the operators are defined as:
\begin{align}
    \cL_x^{c;Q} V(x,z,g) ={}& \sum_{j=1}^{N_x} \partial_{x_j} V(x,z,g) \mu_{x,j}(c,x,z,Q(z,g)) \\
    {}& + \frac{1}{2} \sum_{j,k=1}^{N_x} \partial_{x_j,x_k} V(x,z,g) \Sigma_{x,jk}(c,x,z,Q(z,g))  \\
    {}& + \lambda(x,z) \left( V(x + \varsigma_{x}(c,x,z,Q(z,g)),z,g) - V(x,z,g) \right) \\
    \cL_z V(x,z,g) ={}& \partial_z V (x,z,g) \mu_z(z) + \frac{1}{2}\partial_{zz} V(x,z,g) \left(\sigma_z(z)\right)^2  \\
    \cL_g^{\tilde{\mu}_g} V(x,z,g) ={}& \int_{\cX} D_g V(x,z,g)(y) \tilde{\mu}_g(y, z, g) dy %
\end{align}
where $D_g V(x,z,g): \cX \rightarrow \bR$, $y \mapsto D_g V(x,z,g)(y)$ is the kernel of the Riesz representation of the Frechet derivative of $V$ with respect to the distribution at $(x,z,g)$, $\Sigma_x(\cdot) := \sigma_{x}(\cdot) \left(\sigma_{x}(\cdot)\right)^\top $, $a_j$ denotes element $j$ of vector $a$ and $A_{jk}$ denotes element $(j,k)$ of matrix $A$.
The recursive optimal control, $c^\ast(x,z,g)$, is characterised by the first order condition:
$$0 = \partial_c \left( u(c,z,Q(z,g)) +  \psi (c,x,Q(z,g)) + \cL^{c;Q}_xV(x,z,g)\right)\big|_{c=c^*(x,z,g)}$$

\noindent\emph{Kolmogorov Forward Equation (KFE):}
Denote the recursive equilibrium optimal control of the individual agents by $c^*(x,z,g; \tilde{\mu}_g)$. Compared with the above notation, we add $\tilde{\mu}_g$ to stress the fact that the belief of the agent may differ from the true $\mu_g$.
In Appendix \ref{asec:genericModel}, we show that, for a given $z$ path, the evolution of the distribution under the optimal control can be characterized by the Kolmogorov Forward Equation (KFE):\footnote{
    Observe that there is no noise term in the KFE because $dB_t^0$ does not directly impact the evolution of idiosyncratic states.
}${}^{,}$\footnote{
    For notational convenience, we assume in equation~\eqref{eq:generic-KFE} that all variables are continuous. The same formula extends naturally to discrete variables if the integrals in the discrete dimensions are interpreted as sums and all derivatives in that dimension are set to zero.
}
\begin{align}
\label{eq:generic-KFE}
    dg_t(x) ={}& \mu_g(x, z_t, g_t; c^*, \tilde{\mu}_g)dt, \quad \text{where} \\
    \mu_g(x, z, g; c^*, \tilde{\mu}_g) 
    :={}& -\sum_{j=1}^{N_x} \partial_{x_j}\left[\mu_{x,j}(c^*(x,z,g; \tilde{\mu}_g), x, z, Q(z, g)) g(x)\right] \\
    {}& +  \frac{1}{2} \sum_{j,k=1}^{N_x} \partial_{x_j,x_k} \left[\Sigma_{x,jk}(c^*(x,z,g; \tilde{\mu}^g),x,z,Q(z,g)) g(x)\right] \\
    {}& + \lambda (g(x-\breve{\varsigma}(x,z,g; \tilde{\mu}_g))|I-D_x \breve{\varsigma}(x,z,g; \tilde{\mu}_g)| - g(x))
\end{align}
where $\breve{\varsigma}$ is defined so that $X = x - \breve{\varsigma}(x,z,g; \tilde{\mu}_g)$ is equivalent to $X + \varsigma(c^*(X,z,g; \tilde{\mu}_g),X,z,Q(z,g)) = x$.
Under this recursive characterization, the belief consistency condition becomes that $\mu_g = \tilde{\mu}_g$. \\

\noindent\emph{Master Equation:} 
We follow the approach of \cite{LionsCDF} and characterize the equilibrium in one PDE, which is often referred to as the ``master equation'' of the ``mean field game''.
This formulation is particularly convenient when the evolution of the economy is subject to aggregate shocks and the evolution of the aggregate state variables cannot be determined deterministically.
Conceptually, the master equation is related to the HJBE~\eqref{eq:HJBE-generic} but it imposes the belief consistency by putting the equilibrium KFE into the HJBE.
Formally, the equilibrium value function $V(x,z,g)$ is the solution to the following ``master equation'': 
\begin{align}
    0 = \cL V(x,z,g) ={}& -\rho V(x,z,g) + u(c^*(x,z,g),z,Q(z,g)) + \psi(c^*(x,z,g),x,Q(z,g))  \\
    {}&+ (\cL_x^{c^*(x,z,g);Q} + \cL_z + \cL_g^{\mu_g}) V(x,z,g) \label{eq:master}
\end{align}
where all operators are as defined above.
The mathematical challenge of working with the master equation is that it contains an infinite dimensional variable, $g$, and a derivative with respect to this variable.
This poses a collection of mathematical difficulties for the mean field game theory literature, which attempts to find conditions under which the infinite dimensional master equation is well defined and has a solution.
We refer to~\cite{Cardaliaguet2015,bensoussan2015master} for more details.
By contrast, we are focused on numerical approximation, which means that we need to find a finite dimensional approximation to the distribution, convert the master equation into a high, but finite, dimensional PDE, and develop techniques for solving the PDE.
The goal of our paper is use deep learning techniques to characterize such a finite dimensional numerical approximation without using a perturbation approach.

\subsection{Model Generality}

Our model set up nests many canonical models in macroeconomics such as continuous time versions of \cite{Krusell1998} and \cite{Khan2008}.
However, we have also made strong assumptions, which we discuss below:
\begin{enumerate}[label=(\roman*).]
    \item We have assumed that the prices, $q$, can be expressed explicitly in closed form as functions of the aggregate state variables $(z,g)$, in equation ~\eqref{eq:market_clearing}.
    In models with complicated asset pricing and market clearing conditions we need to solve an auxiliary fixed point problem in order to solve for the recursive representation of the price.
    In \cite{GopalakrishnaGuPayne2024}, some of the authors to this paper extend the methodology to resolve the challenges of introducing portfolio choice and long-term assets with complicated pricing.
    \item We have assumed that the distribution only enters the master equation through the pricing function.
    If the model had search and matching frictions, then the distribution would impact which agents are likely to be matched and so enter the master equation in a more complicated way.
    In \cite{payne2024DeepSAM}, some of the authors to this paper extend the methodology to resolve these challenges.
    \item We have assumed that the aggregate Brownian motion $B^0$ does not directly shock the evolution of idiosyncratic states.
    Instead it changes prices and so indirectly changes agent controls.
    This means that the KFE~\eqref{eq:generic-KFE} does not have an aggregate noise term $dB_t^0$.
    The solution approach easily generalizes to handle KFEs with aggregate noise if it does not lead to significantly more complicated market clearing conditions (a common extension studied is \cite{fernandez2023financial}).
    In \cite{GopalakrishnaGuPayne2024}, some of the authors to this paper consider such cases (and more complicated extensions with long-term asset pricing).
\end{enumerate}

Ultimately, the goal of this paper is not to offer a toolbox to solve every macroeconomic model because deep learning techniques need to be customized to particular master equations.
Instead, we explain and compare how to use different types of distribution approximations to solve continuous time models with heterogeneous agents.
This offers a foundational set of techniques that are being extended in other papers.

\section{Finite Dimensional Master Equation} \label{sec:finite_master}

In this section, we study finite dimensional approximations to the population distribution $g$.
Let $\hat{\varphi} \in \hat{\Phi} \subseteq \bR^{N}$ be a finite dimensional parameter vector, which could represent the collection of agents, the bins in a histogram, or the coefficients in a projection.
Let $\hat{G}$ be a mapping from parameters to distributions:
\begin{align}
    \hat{G}: \hat{\Phi} \to \cG, \quad \hat{\varphi} \mapsto  \hat{G}(\hat{\varphi}) = \hat{g}
\end{align}
We look for an approximation to the value function 
of the form $\hat{V} : \cX \times \cZ \times \hat{\Phi} \rightarrow \bR$, $(x,z,\hat{\varphi}) \mapsto \hat{V}(x,z,\hat{\varphi})$
that satisfies an approximate master equation:
\begin{align}
    0 = \hat{\cL} \hat{V}(x,z,\hat{\varphi}) :={}& -\rho \hat{V}(x,z,\hat{\varphi}) + u(\hat{c}^*(x,z,\hat{\varphi}),z,\hat{Q}(z,\hat{\varphi}))   + \psi(\hat{c}^*(x,z,\hat{\varphi}),x,\hat{Q}(z,\hat{\varphi}))\\ {}&+ (\cL_x^{\hat{c}^*(x,z,\hat{\varphi});\hat{Q}} + \cL_z + \hat{\cL}_g) \hat{V}(x,z,\hat{\varphi}), %
\end{align}
where $\hat{Q}(z,\hat{\varphi}) := Q(z,\hat{G}(\hat{\varphi}))$ and $\hat{c}^{*}$ is defined to satisfy:
\begin{align}
    0 = \partial_c \left( u(c,z,\hat{Q}(z,\hat{\varphi})) + \psi (c,x,\hat{Q}(z,\hat{\varphi})) + \cL^{c;\hat{Q}}_x V(x,z,\hat{\varphi}) \right)\big|_{c=\hat{c}^*(x,z,\hat{\varphi})} \label{eq:chat}
\end{align}
The operators $\cL_x^{c;Q}$ and $\cL_z$ are the same as before.
However, the operator $\hat{\cL}_g$, is different because it depends upon the way that the distribution is approximated.

We characterize three distribution approximation approaches.
The first approach approximates the distribution with a finite collection of agents.
The second approach approximates the continuous state space by a finite collection of grid points.
The third approach projects the distribution onto a finite set of basis functions.
In each case, we show how $\hat{G}$ and $\hat{\cL}_g$ are defined.
We then discuss the relationship to perturbation approaches.

\subsection{Finite Agent Approximation}

\noindent \emph{Distribution approximation:}
In this approach, we restrict the model so that the economy contains a large but finite number of agents $N < \infty$.
In this case, the finite dimensional parameter vector is the vector of agent positions: $\hat{\varphi}_t := \left( x_t^i \right)_{i \le N}$.
The mapping $\hat{G}$ takes the agents positions and computes the empirical measure: $\hat{G}(\hat{\varphi}) = \frac{1}{N} \sum_{i=1}^N \delta_{x_t^i}$
where $\delta_{x_t^i}$ denotes a Dirac mass at $x_t^i$.
The evolution of $\hat{\varphi}_t$ is simply the law of motion for each agents $x_t^i$, as described by equation \eqref{eq:dstate_i}. \\

\noindent\emph{The Operator $\hat{\cL}_g$:}
The market clearing condition now becomes $q_t = Q(z_t, \hat{\varphi}_t) = Q(z,\hat{G}(\hat{\varphi}_t))$, as described in the introduction.\footnote{
    In the examples we consider, it is straightforward to extend $Q$ to an empirical measure because $Q$ only depends upon moments of the distribution.
    For more complicated cases, we could fit a kernel to the empirical measure to approximate smooth density.
}
However, to maintain the price taking assumption in the finite agent model, we impose that agent $i$ behaves as if their individual actions do not influence prices.
Formally, this means that agent $i$ perceives the pricing function to be $q_t = Q(z_t, \hat{\varphi}_t^{-i})$,
where $\hat{\varphi}_t^{-i} = \{x_t^j \in N^{-i}\}$ is the position of the other agents $N^{-i} := \{j \le N: j \ne i\}$.
Ultimately, this will ensure that the neural network trains the policy rules as if the agents believe that their assets do not influence the market prices.
The approximate distribution impact operator $\hat{\cL}_g$ is given by:
\begin{align}
\begin{aligned} \label{eq:Lg_finite_agent}
    (\hat{\cL}_g \hat{V})(x^i,z,\hat{\varphi}) :={}& \sum_{j \ne i} D_{\hat{\varphi}^j} \hat{V}(x^i,z,\hat{\varphi}) \mu_{x}(\hat{c}^*(x^j,z,\hat{\varphi}),x^j,z,\hat{Q}(z,\hat{\varphi}^{-j})) \\
    {}& + \sum_{j \ne i} \frac{1}{2} \text{tr} \left\{ \Sigma_x(x^j,z,\hat{Q}(z,\hat{\varphi}^{-j})) D^2_{\hat{\varphi}^j} \hat{V}(x^i,z,\hat{\varphi}) \right\} \\
    {}& + \sum_{j \ne i} \lambda(x^j,z) \left( \hat{V}(x^i,z, \hat{\varphi}^{-i} + \Delta_j) - \hat{V}(x^i,z,\hat{\varphi}^{-i}) \right),
\end{aligned}
\end{align}
where for simplicity we have used the vector and trace notation (rather than the summation notation in Section \ref{sec:genericModel}) and where $D_{\hat{\varphi}^j} \hat{V}(x^i,z,\hat{\varphi})$ is the partial gradient of $\hat{V}$ with respect to the $j$-th point in $\hat{\varphi}$ (i.e., $x^j_t$), and $\Delta_k = 0$ for $k\neq j$ and $\Delta_k = \varsigma_{x}(x^k,z,\hat{Q}(z,\hat{\varphi}^{-k}))$ for $k=j$.\\

\noindent\emph{Convergence properties:}
The solution $V(x^i,z,\hat{\varphi})$ is expected to converge to the solution of the master equation~\eqref{eq:master} as the number of agents, $N$, grows to infinity so long $V$ is smooth. 
Intuitively, this relies on the idiosyncratic noise in the population distribution averaging out as the population becomes large, see e.g.~\cite{sznitman1991topics}. 
Such results have been extended to systems with equilibrium conditions by the mean field games literature; see e.g.~\cite{Cardaliaguet2015,lacker2020convergence,delarue2020master}.

\subsection{Discrete State Space Approximation}
\label{subsec:discrete-state-approx-master}

We consider $N$ points in the state space, denoted by $\xi_1, \dots, \xi_{N} \in \cX$. 
We approximate $g$ by a vector $\hat{\varphi} = (\hat{\varphi}_1, \dots, \hat{\varphi}_N) \in \mathbb{R}^{N}$, whose values represent the masses at $\xi_1, \dots, \xi_{N}$. 
The mapping $\hat{G}$ then takes the form: $\hat{G}(\hat{\varphi}) = \sum_{n=1}^{N}{\hat{\varphi}_{n}\delta_{\xi_n}}$,
where again $\delta_{\xi_n}$ denotes a Dirac mass at $\xi_n$.
Conceptually, the finite agent approximation fixes the mass associated to each $x_t^i$ and allows the $x_t^i$ values to move whereas the discrete state space approximation fixes the grid points $\xi_n$ and allows the masses at each grid point to change.

We denote by $\hat{\varphi}_t$ the vector at time $t$. Its evolution is given by an ordinary differential equation in dimension $N$ of the form:
\begin{equation}
    \label{eq:generic-KFE-finite-g}
    d\hat{\varphi}_t = \mu_{\hat{\varphi}}(z_t, \hat{\varphi}_t)dt
\end{equation}
describing the evolution of mass at $\xi_1, \dots, \xi_{N}$. 
The right-hand side needs to be obtained using information from the KFE~\eqref{eq:generic-KFE}. 
In our numerical examples we use a finite difference approximation to the KFE to derive $\mu_{\hat{\varphi}}$, analogous to the approximation described in \cite{achdou2022income}. 
However, the technique can be applied to other types of approximations, like the finite volume method used by \cite{huang2023probabilistic}.
In Appendix \ref{asec:master_equations}, we provide an explicit example for the \cite{Krusell1998} model using a finite difference scheme.
Here, we describe a generic approximation of the KFE equation~\eqref{eq:generic-KFE} using a finite difference scheme: 
\begin{align*}
 \mu_{\hat\varphi}(z,\hat{\varphi}) :={}& -\sum_{j=1}^{N_x} \hat{\partial}_{x_j}\left[\boldsymbol{\mu}_{x,j}(z,\hat{\varphi}) \odot \hat{\varphi}\right]  +  \frac{1}{2} \sum_{j,k=1}^{N_x} \hat{\partial}_{x_j,x_k} \left[\boldsymbol{\Sigma}_{x,jk}(z,\hat{\varphi}) \odot \hat{\varphi}\right] \\
 {}&+ \lambda (\Delta_{\breve{\boldsymbol{\varsigma}}(z,\hat{\varphi})}[\hat{\varphi}]\odot \boldsymbol{a}(z,\hat{\varphi}) - \hat{\varphi})
\end{align*}
where $\odot$ denotes the element-wise product of vectors (Hadamard product), and $\boldsymbol{\mu}_{x,j}(z,\hat{\varphi})$, $\boldsymbol{\Sigma}_{x,jk}(z,\hat{\varphi})$, $\breve{\boldsymbol{\varsigma}}_j(z,\hat{\varphi})$, $\boldsymbol{a}(z,\hat{\varphi}) \in \mathbb{R}^N$ and $\breve{\boldsymbol{\varsigma}}(z,\hat{\varphi})\in(\mathbb{R}^{N_x})^N$ are the vectors representing the values on the grid points:
\begin{align*}
    \boldsymbol{\mu}_{x,j}(z,\hat{\varphi}) &= \left(\mu_{x,j}(\hat{c}^*(\xi_n,z,\hat{\varphi}), \xi_n, z, \hat{Q}(z, \hat{\varphi})): n=1,...,N\right)\\ 
    \boldsymbol{\Sigma}_{x,jk}(z,\hat{\varphi}) &= \left(\Sigma_{x,jk}(\hat{c}^*(\xi_n,z,\hat{\varphi}), \xi_n, z, \hat{Q}(z, \hat{\varphi})): n=1,...,N\right)\\
    \breve{\boldsymbol{\varsigma}}_j(z,\hat{\varphi}) &= (\breve{\varsigma}_{j}(\xi_n,z,\hat{\varphi}):n=1,...,N)\\
    \breve{\boldsymbol{\varsigma}}(z,\hat{\varphi}) &= (\breve{\varsigma}(\xi_n,z,\hat{\varphi})):n=1,...,N)\\
    \boldsymbol{a}(z,\hat\varphi) &=(|I-((\hat{\partial}_{x_{k}}[\breve{\boldsymbol{\varsigma}}_{j}(z,\hat{\varphi})])_n)_{j,k=1,...,N_x}|:n=1,...,N).
\end{align*}
Informally, $\hat{\partial}_{x_j}$ and $\hat{\partial}_{x_j,x_k}$ takes the function values on the grid points and return finite difference approximations of the first derivative w.r.t. $x_j$ and the second derivative w.r.t. $(x_j, x_k)$.
Informally, $\Delta_{\breve{\boldsymbol{\varsigma}}}$ takes the function values on the grid points and returns interpolated values at shifted inputs $\xi_n - \breve{\boldsymbol{\varsigma}}_n$, which may not be on the grid.
More formally, 
$\hat{\partial}_{x_j}$, $\hat{\partial}_{x_j,x_k}$, and $\Delta_{\breve{\boldsymbol{\varsigma}}}$ are operators $\mathbb{R}^N \rightarrow \mathbb{R}^N$ that map vectors $\boldsymbol{f}\in\mathbb{R}^N$ of function values of a function $f$ on the grid (i.e., $\boldsymbol{f}_n=f(\xi_n)$) into vectors of the same size, such that, $(\hat{\partial}_{x_j}[\boldsymbol{f}])_n\approx \partial_{x_j}f(\xi_n)$, $(\hat{\partial}_{x_j,x_k}[\boldsymbol{f}])_n\approx \partial_{x_j,x_k}f(\xi_n)$, and $(\Delta_{\breve{\boldsymbol{\varsigma}}}[\boldsymbol{f}])_n \approx f(\xi_n - \breve{\boldsymbol{\varsigma}}_n)$.\footnote{
    The precise forms of $\hat{\partial}_{x_j}$, $\hat{\partial}_{x_j,x_k}$, and $\Delta_{\breve{\boldsymbol{\varsigma}}}$ are implementation-specific. In the most common arrangement, grid points are positioned on a rectangular grid, the operators $\hat{\partial}_{x_j}$ and $\hat{\partial}_{x_j,x_k}$ compute approximate derivatives at each point based on differences in function values to neighbor grid points, and the operator $\Delta_{\breve{\boldsymbol{\varsigma}}}$ implements linear interpolation on the regular grid.
    }\\

\noindent\emph{The Operator $\hat{\cL}_g$:} The approximate distribution impact operator $\hat{\cL}_g$ is defined by:
\begin{equation}
    \label{eq:def-hatcL-finite-state-generic}
        (\hat{\cL}_g \hat{V})(x,z,\hat{\varphi})= \sum_{n=1}^{N} \mu_{\hat{\varphi},n}(z,\hat{\varphi}) \, \partial_{\hat{\varphi}_n} \hat{V}(x,z,\hat{\varphi}),
\end{equation}
where $\mu_{\hat{\varphi},n}(z,\hat{\varphi})$ denotes the $n$-th coordinate of the $N$-dimensional vector $\mu_{\hat{\varphi}}(z,\hat{\varphi})$. \\

\noindent\emph{Convergence properties:}
Once again, we expect the solution of the approximate master equation $V(x,z,\hat{g})$ to converge towards the solution of the true master equation~\eqref{eq:master} so long as $V$ is smooth. Convergence of mean field games with discrete state spaces to mean field games with continuous state spaces has been proved by~\cite{bayraktar2018numerical,hadikhanloo2019finite} without noise or with idiosyncratic noise, and
by~\cite{bertucci2022mean} with common noise.

\subsection{Projection onto Basis\label{sec:finite_master:projection}}

\noindent\emph{Distribution approximation:}
In this approach, we represent the distribution $g_t$ by functions of the form:
$$\hat{g}_t(x) = \hat{G}(\hat{\varphi}_t)(x) := b_0(x) + \sum_{n=1}^{N}{\hat{\varphi}_{n,t}b_n(x)}, $$
where $\hat{\varphi}_t := (\hat{\varphi}_{1,t},...,\hat{\varphi}_{N,t})\in\hat{\Phi} := \mathbb{R}^N$ is a vector of real-valued coefficients and $b_0,b_1,...,b_N$ is a collection of linearly independent real-valued functions on $\mathcal{X}$ that satisfy
\begin{equation}\label{eq:basisProperties}
\int_{\mathcal{X}} b_{n}(x)dx=\begin{cases}
1, & n=0\\
0, & n\geq1
\end{cases}
\end{equation}
and we refer to as the basis of the projection. %
To complete this approximation description we need to specify the the law of motion for $\hat{\varphi}_t$ and the choice of basis.

As in the discrete state space approach, the evolution of $\hat{\varphi}_t$ is given by an ODE of the form~\eqref{eq:generic-KFE-finite-g} with drift $\mu_{\hat{\varphi}}$.
To derive $\mu_{\hat{\varphi}}$, we would like to use the KFE~\eqref{eq:generic-KFE} directly.
However, if we substitute our projection $\hat{g}_t$ into the KFE, then we a get a linear equation system:
\begin{align}
\sum_{n=1}^{N}{\mu_{\hat{\varphi},n,t}b_n(x)} =
\mu^g(\hat{c}_t^*,x,z_t,\hat{g}_t) \label{eq:kfe_projection}
\end{align}
for $\mu_{\hat{\varphi}}$ that has generically no solution for $N < \infty$ because we typically have an infinite number of equations (for each $x \in \cX$) but only $N$ degrees of freedom to solve these equations.
To make progress, we need to generate a finite set of equations to solve.
One option would be to discretize the $x$-dimension, as in the previous technique.
However, we instead find it helpful to work with a more general discretization to a set of $M \ge N$ ``test functions'' $\phi:\mathcal{X}\rightarrow \mathbb{R}$.
This is useful because it ultimately allows us to choose test functions that focus the approximation accuracy on pre-selected statistics of the distribution that are likely to be economically important. For example, if prices $q_t = Q(z_t, g_t)$ depend only on certain moments of the distribution such as the mean, it makes sense to choose test functions that represent these moments.
To make this idea precise, we start with the integral form of the KFE~\eqref{eq:generic-KFE} before integration by parts:
\begin{align*}
\frac{d\int_{\mathcal{X}}{\phi(x)g_t(x)dx}}{dt} &= \int_{\mathcal{X}}{\cL_x^{c^\ast(x,z_t,g_t);Q} \phi(x,z_t,g_t)g_t(x)dx} =: \mu_\phi(z_t,g_t;c^{\ast})
\end{align*}
that has to hold for all the $M$ test functions.
This variant of the KFE describes the time evolution of the statistics $\int_{\mathcal{X}}{\phi(x)g_t(x)dx}$ of the distribution.\footnote{We recover the original KFE~\eqref{eq:generic-KFE} for the Dirac delta distributions as test functions, i.e. $\phi=\delta_x$ for $x\in\cX$.} 
We define the coefficient drifts $\mu_{\hat{\varphi},n}(z,\hat{\varphi})$ (conditional on the aggregate state $(z,\hat{\varphi})$) as the regression coefficients that minimize, in the least-squares sense, the $M$ linear regression residuals\footnote{In practice, the integral terms appearing in the residual formula have to be computed either analytically (if possible) or by numerical quadrature.}
$$\varepsilon_m(z,\hat{\varphi}) := \mu_{\phi_m}(z,\hat{G}(\hat{\varphi});\hat{c}^{\ast}) - \sum_{n=1}^{N}{\mu_{\hat{\varphi},n}(z,\hat{\varphi})\int_X{\phi_m(x)b_n(x)dx}},\qquad m=1,...,M.$$
Ultimately, this means that our approximation is most accurate the on chosen statistics encoded in the test functions.

The approximation presented so far works, in principle, for any choice of basis. 
Here we propose a basis that approximately tracks the persistent dimensions of $g_t$ while neglecting those dimensions that mean-revert fast. 
These persistent dimensions of the distribution are related to certain eigenfunctions of the differential operator characterizing the KFE~\eqref{eq:generic-KFE}. Because this differential operator is generally time-dependent and stochastic, we first replace it by a time-invariant operator $\overline{\cL}^{KF}$. For example, this could be the steady-state operator $\cL^{KF,ss}$ that is defined as the KFE operator at the steady state, $g_t = g^{ss}$ in a simplified model without common noise.\footnote{
    Our chosen basis approximately tracks the persistent dimensions of $g_t$ if $\tilde{\cL}^{KF}$ is similar to full stochastic operator $\mathcal{L}^{KF}_t$. This is plausible for $\tilde{\cL}^{KF}=\cL^{KF,ss}$ because both operators share many similar features.
}
Let $\{b_i\,:\, i \ge 0\}$ be the set of eigenfunctions of $\overline{\cL}^{KF}$ with corresponding eigenvalues $\{\lambda_i \in \mathbb{C} \,:\, i \ge 0\}$. If the dynamics prescribed by the KFE are locally stable around the steady state $g^{ss}$, there is one eigenvalue $\lambda_0=0$ with eigenfunction $b_0 = g^{ss}$ and all remaining eigenvalues have negative real part, $\Re \lambda_i<0$. We pick as our basis $b_0,b_1,...,b_N$ the $N+1$ eigenfunctions corresponding to eigenvalues with real parts closest to zero as these represent the most persistent components. We provide further details on this basis choice in Appendix~\ref{sec:eigenfunction-basis}.\\

\noindent\emph{The Operator $\hat{\cL}_g$:}
The approximate distribution impact operator $\hat{\cL}_g$ is defined by:
$$(\hat{\cL}_g \hat{V})(x,z,\hat{\varphi})= \sum_{n=1}^{N} \mu_{\hat{\varphi},n}(z,\hat{\varphi}) \, \partial_{\hat{\varphi}_n} \hat{V}(x,z,\hat{\varphi}).$$

\noindent\emph{Convergence:} There are fewer results about convergence for projections in the continuous time mean-field-game literature.
However, in discrete time, \cite{Prohl2017} has proven convergence results for particular projections.

\subsection{Relationship to Continuous Time Perturbation Approaches}

All approaches offer a global characterization of the aggregate dynamics. This opens up the possibility of studying the full nonlinear dynamics but requires non-trivial approximations to the KFE to preserve the nonlinearity. 
Here, we briefly contrast our approaches to continuous time perturbation approaches that replace the nonlinear evolution of aggregate states by a simpler linear or quadratic one.
We compare to state space perturbations rather than sequence space perturbations since they are closer to our work.

\cite{Ahn2018} adapts and extends the perturbation approach of \cite{Reiter2002, reiter2009solving} to continuous-time settings. 
They first discretize the state space and then linearize the model with respect to aggregate dynamics.
This allows the authors to solve for equilibrium using conventional matrix algebra techniques.
In our setup, this is analogous to linearizing the finite-dimensional master equation that results from our discrete state space approximation and imposing their dimension reduction.

The perturbation method of \cite{bilal2021solving} works directly with the master equation formulation.
Relative to traditional techniques, they take linear or quadratic perturbations of the master equation with respect to the aggregate states $(z,g)$. 
The approximation by a linear or quadratic functional form reduces the complexity of the problem because one only needs to solve for the perturbation coefficients not for the value function on the infinite-dimensional state space of the full nonlinear problem.\footnote{
        Note that, when solving the remaining equation numerically, \cite{bilal2021solving} chooses a finite-difference method, which still requires a discretization of the state space akin to our discrete state space approach.
} 
By contrast we obtain a finite-dimensional master equation by means of approximating the distribution, which allows us to preserve the full nonlinearity of the problem.

\cite{alvarez2022analytic} use Fourier methods to analytically characterize impulse responses functions to shocks to the cross-sectional distribution $g_t$ in a model that features a linear KFE and a constant $z_t$-state.\footnote{\cite{alvarez2023price} employ similar methods in a nonlinear model but linearize dynamics first.} That approach is closely related to our projection method based on an eigenfunction basis. Specifically, if the KFE is truly linear, then the following proposition holds:
\begin{proposition}\label{prop:linear_kfe_eigenfunction_basis}
 If the KFE~\eqref{eq:generic-KFE} is of the form $dg_t = \mathcal{L}^{KF}g_t dt$ with a (constant) linear operator $\mathcal{L}^{KF}$, $b_0=g^{ss}$, $b_1$, ..., $b_N$ are eigenfunctions of $\cL^{KF}$ with eigenvalues $\lambda_0=0$, $\lambda_1$, ..., $\lambda_N$, and $g_0 - g^{ss} = \sum_{n=1}^{N}\varphi_{i,0} b_i \in \operatorname{span}\{b_1,...,b_N\}$, then $g_t$ satisfies equation~\eqref{eq:generic-KFE} for all $t\geq 0$ if and only if, for all  $t\geq 0$,
 \begin{equation}\label{eq:linear_kfe_eigenfunction_representation}
 g_t = g_t - g^{ss} = \sum_{i=1}^{N}{\varphi_{i,0} e^{\lambda_i t}b_i}.
 \end{equation}
\end{proposition}
\begin{proof}
See Appendix \ref{sec:eigenfunction-basis}.
\end{proof}

This says that, if the KFE is linear and we use $\overline{\cL}^{KF} = \cL^{KF}$ to form the eigenfunction basis, then we are in the non-generic case in which we can approximate the KFE perfectly because the equation system~\eqref{eq:kfe_projection} has a solution even though there are only finite degrees of freedom.
The connection to the impulse response formula provided by \cite{alvarez2022analytic} is as follows. If we consider an initial shock that moves $g_0$ away from $g^{ss}$ in a way such that $g_0-g^{ss}$ is spanned by $b_1,...,b_N$, then equation~\eqref{eq:linear_kfe_eigenfunction_representation} tells us exactly how the distribution evolves over time. We can further integrate this equation with respect to any function $f:\cX\rightarrow \mathbb{R}$ to obtain a formula for the dynamics of the statistic $\int{f(x)g_t(x)dx}$ over time that resembles the impulse response representation of \cite{alvarez2022analytic}.\footnote{A difference to \cite{alvarez2022analytic} is that, in the context of their model, they know closed-form solutions for the eigenfunctions and they consider $N=\infty$, so that they can consider any initial distribution displacement $g_0-g^{ss}$.}

Our computational approach is different from Proposition~\ref{prop:linear_kfe_eigenfunction_basis} in that we do not presume a linear KFE and so allow for non-linear distributional dynamics. 
In this case, the KFE can no longer be approximated perfectly on a finite basis and so there are no closed-form expressions in the spirit of equation~\eqref{eq:linear_kfe_eigenfunction_representation}. 
Instead, we need to make additional choices about how to approximate the KFE, which leads to the more complicated procedure outlined in Section~\ref{sec:finite_master:projection}.%

\subsection{Comparison Between Our Approximation Approaches}

Table \ref{tab:comparion_distribution_approx} summarizes the key differences between the distribution approximations: a finite population, a discrete state space, and a projection onto a finite set of basis functions.

\begin{table}[ht]
\centering
\renewcommand{\arraystretch}{1.5}
\begin{tabularx}{\textwidth}{lXXX}
\hline
 & Finite Population & Discrete State & Projection \\
\hline\hline
Distribution approx. & $\frac{1}{N} \sum_{i=1}^N \delta_{x_t^i}$ 
& $\sum_{n=1}^{N}{\hat{\varphi}_{n,t}\delta_{\xi_n}}$  & $\sum_{n=0}^{N}{\hat{\varphi}_{n,t}b_n(x)}$ \\
\hline
KFE approx. & Evolution of other agents' states & Evolution of mass between discrete states & Evolution of projection coefficients \\
\hline
\end{tabularx}
\caption{Comparison of Distribution Approximations} \label{tab:comparion_distribution_approx}
\end{table}

\noindent We discuss how these approximations compare with regard to computational difficulties:
\begin{enumerate}
    \item Dimensionality ($N$): The approximation dimension needs to be large enough to capture sufficient shape in the distribution.
    The projection method can potentially have the lowest dimension if the choice of basis is efficient (e.g. $N=5$ in our \citet{Krusell1998} example).
    The finite population needs to be large enough to average out idiosyncratic noise (e.g. $N=40$).
    The discrete state space needs to be sufficiently fine to approximate the derivatives in the KFE, which means it needs a high dimension (e.g. $N=200$).
    \item Customization: For the finite agent approach, we just choose the number of agents, $N$.
    For the discrete state space approach, we choose a grid and a method for approximating the KFE on the grid points.
    For the projection method, we choose a set of basis function and a set of statistics on which to minimize the error.
    In this sense, the projection is potentially lower dimensional because we need to make more intricate choices in the setup.
    \item Computational ``bottlenecks'':
    Each method has its own computational bottlenecks.
    For the finite agent method, we need to switch agent positions in the neural network approximation to $V$ when we calculate the derivatives of $V$ with respect to the other agent positions in equation \eqref{eq:Lg_finite_agent}.
    For the discrete state space method, the dimensionality of the approximation is the main computational problem.
    For the projection method, determining the drift $\mu_{\hat{\varphi}}$ of the distribution approximation is computationally involved as we need to solve a linear regression problem and compute several integrals by quadrature for every single evaluation of the master equation.
\end{enumerate}

\section{Solution Method\label{sec:solution}}

All approaches in Section~\ref{sec:finite_master} lead to finite approximations to the density, $\hat{g}$, and the master equation operator $\hat{\cL}$.
However, the resulting master equations are high dimensional and so cannot be solved by traditional numerical techniques.
Instead, we represent the solution to the approximate master equation by a neural network and deploy tools from the ``deep learning'' literature to ``train'' the neural network to solve the approximate master equation.

\subsection{Neural Network Approximations\label{sec:solution:network}}

A neural network is a type of parametric functional approximation that is built by composing affine and non-linear functions in a chain or ``network'' structure (see \cite{goodfellow2016deep}).
We let $\hat{X}:=\{x, z, \hat{\varphi}\}$ denote the collection of inputs into the neural network representation of the value function.
We denote the neural network approximation to the value function by $\hat{V}(\hat{X}) \approx \bV(\hat{X}; \Theta)$, where $\Theta$ are the neural network parameters that depend upon the architecture.
There are many types of neural networks.
The simplest form is a ``feedforward'' neural network which is defined by: 
\begin{align}
\begin{aligned} \label{nn:structure}
    h^{(1)} ={}& \phi^{(1)}( W^{(1)} \hat{X} + b^{(1)}) &&& \ldots \text{Hidden layer 1}\\
    \vdots \\
    h^{(H)} ={}& \phi^{(H)}( W^{(H)} h^{(H-1)} + b^{(H)}) &&& \ldots \text{Hidden layer $H$} \\
    o ={}& W^{(H+1)} h^{(H)} + b^{(H+1)} &&& \ldots \text{Output layer} \\
    \bV ={}& \phi^{(H+1)}(o) &&& \ldots \text{Output}
\end{aligned}
\end{align}
where the $\{h^{(i)}\}_{i\le H}$ are vectors referred to as ``hidden layers'' in the neural network, $\{W^{(i)}\}_{i\le (H+1)}$ are matrices referred to as the ``weights'' in each layer, $\{b^{(i)}\}_{i\le (H+1)}$ are vectors referred to as the ``biases'' in each layer, $\{\phi^{(i)}\}_{i\le (H+1)}$ are non-linear functions applied element-wise to each affine transformation and referred to as ``activation functions'' for each layer.
The length of hidden layer, $h^{(i)}$, is defined as the number of neurons in the layer, which we refer to as $\#h^{(i)}$.
The total collection of parameters is denoted by $\Theta = \{ W^{(i)}, b^{(i)} \}_{i \le (H+1)}$.
The goal of deep learning is to train the parameters, $\Theta$, to make $\bV(\cdot; \Theta)$ a close approximation to $\hat{V}$.

The neural network defined in \eqref{nn:structure} is called a ``feedforward'' network because hidden layer $i$ cannot depend on hidden layers $j > i$.
This is in contrast to a ``recurrent'' neural network where any hidden layer can be a function of any other hidden layer.
It is called ``fully connected'' if all the entries in the weight matrices can be non-zero so each layer can use all the entries in the previous layer.
In this paper, we will consider a fully connected ``feedforward'' network to be the default network.
This is because these networks are the quickest to train and so we typically start by trying out this approach.
However, there are applications where we find that more complicated neural network architectures are useful.
In particular, we find that the type of recurrent neural network suggested by the Deep Galerkin Method in \cite{Sirignano2018} is helpful for discrete state and projection approximations.

\subsection{Solution Algorithm\label{sec:solution:algorithm}}

We train the neural network to learn parameters $\Theta$ that minimize the error in the master equation and boundary conditions.
We describe the key steps in Algorithm~\ref{alg:generic}.\footnote{
    The generic pseudo-code given in Algorithm~\ref{alg:generic} can be modified in practice. For example, instead of fixing a precision threshold, one can fix a number of iterations, and instead of using a fixed sequence of learning rates, one can use an adaptive method, such as Adam.}
Essentially, the algorithm samples random points in the discretized state space $\{x, z, \hat{\varphi}\}$, calculates the master equation error on those points under the current neural network approximation, and then updates the neural network parameters to decrease the master equation error.
In practice, our loss consists of two terms:  $\cE^e$ for the average mean squared master equation error and $\cE^s$, which is used to incorporate information about the shape of the solution (e.g., monotonicity or concavity). 
Specific examples of $\cE^s$ will be discussed in the following sections.
In the deep learning literature, this approach is sometimes referred to as ``unsupervised'' learning (e.g. \cite{azinovic2022deep}) because we do not have direct observations of the value function, $V(x,z,\hat{g})$, and instead have to learn it indirectly via the master equation.

\begin{algorithm}[ht] %
    
    \caption{Pseudo Code for Generic Solution Algorithm} \label{alg:generic}
    \SetKwInOut{Input}{Input}
    \SetKwInOut{Output}{Output}
    \Input{Initial neural network parameters $\Theta^0$, number of sample points $M$, positive weights $\kappa^e$ and $\kappa^s$ on the master equation errors; sequence of learning rates $\{\alpha_n : n \ge 0\}$, precision threshold $\epsilon$}
    \Output{A neural network approximation $(x, z, \hat{\varphi}) \mapsto \bV(x, z, \hat{\varphi}; \Theta)$ of the value function.}
    \setstretch{1.1}
    \vspace{0.2cm}
    \begin{algorithmic}[1]
    \STATE Initialize neural network object $\bV(x, z, \hat{\varphi}; \Theta^0)$ with parameters $\Theta^0$.
    \WHILE{Loss $>$ $\epsilon$, in iteration $n$}
        \STATE Generate $M$ new sample points, $S^n = \{(x_m, z_m, \hat{\varphi}_m)\}_{m \le M}$.
        \STATE Calculate the weighted average error:
            \begin{align}
                \cE(\Theta^n, S^n) = \kappa^e \cE^e(\Theta^n, S^n) 
                + \kappa^s \cE^s(\Theta^n, S^n)
            \end{align}
            where $\cE^e$ is the average mean square master equation error:
            \begin{align}
                \cE^e(\Theta^n, S^{n}) :={}& \frac{1}{|S^n|} \sum_{(x, z, \hat{\varphi})\in S^n} |(\hat{\cL}\bV(x, z, \hat{\varphi};\Theta^n)) |^2
            \end{align}
            in which the derivatives in the operator $\hat{\cL}$ are calculated using automatic differentiation
            and $\cE^s$ is a penalty for a ``wrong'' shape that depends upon the specific problem.
        \STATE Update the parameters using stochastic gradient descent:
            \begin{align*}
                \Theta^{n+1} = \Theta^n - \alpha_n D_{\Theta} \cE(\Theta^n, S^n)
            \end{align*}
            where $D_{\Theta} \cE$ is the gradient (i.e., vector differential) operator.
    \ENDWHILE
    \end{algorithmic}
\end{algorithm}

Although the algorithm is straightforward to describe at the high level, implementing the deep learning training scheme successfully involves a lot of complicated decisions.
We discuss these decisions in detail for solving the \cite{Krusell1998} in Sections \ref{subsec:ks:implementation} and \ref{asec:implementation_details}.
Here we discuss sampling, which an implementation detail that is particular challenging, and shape constraints.

\subsubsection{Sampling\label{sec:solution:algorithm:sampling}}

A very important aspect of our solution algorithm is the approach used to sample the set of training points $S^n$ in each iteration. 
Sampling ultimately has to be tailored to the specific application at hand.
This is a difference between deep learning for continuous time and discrete time techniques.
Discrete time models need to calculate expectations and so typically need to use simulation to approximate the expectation operator.
Continuous time models replace the expectation term in the Bellman equation by the derivative terms in the HJBE and then sample points on which to evaluate the HJBE.
This gives continuous time techniques more flexibility in how to sample but can also make the sampling task harder.
The following general considerations are relevant when deciding on how to sample.

We can sample the idiosyncratic state $x$, the aggregate exogenous state $z$, and the distribution state $\hat{\varphi}$ independently with different approaches. The separation between $x$ and $\hat{\varphi}$ deserves particular emphasis in the context of the finite agent approximation, for which also $\hat{\varphi}$ contains a sample of $x$-points (for the other agents). 
This is because we can focus the sampling on regions of the idiosyncratic state space with high curvature without having to increase the number of agents so that simulations have many agents in those regions.

Sampling $x$ and $z$ is less complicated because their dimension is usually relatively low.
For example, in macro models a typical dimension of $x$ is 2-3 and a typical dimension of $z$ is 1-5. For these variables, we typically sample from a pre-specified statistical distribution such as the uniform or normal. The sampling can be refined by using ``active sampling'' (see, e.g., \cite{Gopalakrishna2021} and \cite{lu2021deepxde}), which adapts the sampling during training to actively learn in regions where the algorithm is having trouble minimizing the loss function. This is achieved by regularly inspecting the losses during training and adding training points to regions with the largest losses.

Sampling the distribution approximation $\hat{\varphi}$ is significantly more complicated.
This is because it is typically high-dimensional and so we can only train the neural network on a very small subset of the total possible distributions.
In this sense, deep learning does not break the ``curse of dimensionality''.
Instead, it gives flexibility to train on a useful subspace that gives enough information to the value function for economically relevant distributions.
This means that choosing the right subspace on which to sample is very important for the algorithm to converge in reasonable time (or at all).
Ultimately, this requires us to use some information about the model solution in the sampling.
We have focused on the following three sampling schemes:
\begin{enumerate}[label=(\roman*)]
\item \emph{Moment sampling:} 
We first draw samples for selected moments of the distribution that are important for calculating prices $\hat{Q}(z,\hat{\varphi})$.
We then sample $\hat{\varphi}$ from a distribution that satisfies the moments drawn in the first step. 
E.g., in many macroeconomic models, the distribution mean is particularly important for calculating prices.
In this case, we would sample from the mean and then draw $\hat{\varphi}$ from a distribution with that mean.\footnote{
This final step could involve sampling from a pre-specified distribution (e.g. uniform) or from an ergodic distribution, as described in (iii).}
\item \emph{Mixed steady state sampling:} 
We first solve for the steady state for a collection of fixed aggregate states $z$. This needs to be done only once before training begins.
We then draw in each training step random mixtures of this collection of steady state distributions and then perturb with additional noise.\footnote{Without these perturbations, the random mixtures remain strictly confined to a subspace whose dimension (the number of steady states in the collection) is typically much smaller than $\dim{\hat{\Phi}}$.}
\item \emph{Ergodic sampling:} We adapt the training sample dynamically by regularly simulating the model economy based on the candidate solution for the value function from a previous iteration.
\end{enumerate}

Two additional issues arise in sampling schemes that adapt the training sample dynamically, such as active sampling (for $x$ and $z$) and ergodic sampling (for $\hat{\varphi}$). First, these schemes only adapt the sampling distribution in a meaningful way if the current guess for the value function is sufficiently good. It is therefore advisable to start with a pre-specified sampling distribution in early training and switch to a dynamic sampling scheme later on. Second, dynamically adapting the sample might lead to instability of training due to feedback effects between the training sample and the trained solution. We have found this issue to be particularly relevant for ergodic sampling. To mitigate the issue, we have found the following work well: (i) to use ergodic sampling only for a fraction of the training sample and (ii) to update training points by simulating over a small time interval starting from the end the last simulation rather than repeatedly taking very long simulations to approximate the ergodic distribution.\footnote{
    In this sense, despite the name ``ergodic sampling'', we do not insist on drawing training points from the ergodic distribution implied by the current value function candidate. Instead, the ergodic distribution is reached only gradually once the value function is close to convergence.}

At a high level, we find the following sampling strategies are useful for the different types of distribution approximations.
For the finite agent approximation, we find moment sampling to be simple and effective because the $\hat{\varphi}$ variables have a natural interpretation as the idiosyncratic ($x$) states of the other agents in the population and so it is straightforward to determine a region of ``typical'' values to sample from.
For the discrete state space approximation, we find the most stable approach is to start with mixed steady state sampling and move to ergodic sampling once the neural network starts to converge.
For the projection approximation, we find that a combination of moment sampling and ergodic sampling is effective. To put moment sampling to work with projections requires a rotation of the basis functions so as to isolate the components that correspond to the selected moments, see Appendix~\ref{asec:master_equations} for details.
We discuss all three sampling strategies in detail for our \cite{Krusell1998} example in Appendix~\ref{asec:implementation_details:sampling}.

\subsubsection{Constraining the Value Function Shape \label{sec:solution:algorithm:shape_constraints}}

We find that the neural network can converge to ``cheat solutions'' that approximately solve the differential equation by setting derivatives to zero.
One way of helping with this is to include terms in the loss function the penalize undesirable curvature of the value function to enforce the correct shape (e.g. penalizing non-monotonicity or non-concavity).
There are other, complementary potential ways of addressing this issue: (i) initializing the neural network to match an initial guess satisfying known properties of the value function so that when the neural network is trained to minimize the PDE loss, it converges to a (local) minimizer which has the same desired properties, (ii) sampling from a sufficiently large part of the state space that the neural network realizes there is curvature in all dimensions, and (iii) choosing an architecture which satisfies the constraints (or at least makes it easier for the neural network to satisfy these constraints). 
We find a combination of these approaches to be helpful.

\subsubsection{Extension: Boundary Conditions\label{sec:solution:algorithm:extensions}}

Algorithm~\ref{alg:generic} is suitable for solving a problem that does not include boundary conditions. 
This is sufficient for our purposes because we ``soften'' any hard constraints in the problem by replacing them with a utility flow penalty (see Section~\ref{subsec:master_eqn}).
In principle, the solution algorithm could be extended to a problem that does require boundary conditions. In this case, we would need to sample a second sample $S^b$ of training points on the boundary in step 3 and, in step 4, add an additional error term $\mathcal{E}^b$ (with corresponding weight $\kappa^b$) for the mean-squared error of boundary condition residuals evaluated at the training set $S^b$, as described by \cite{Sirignano2018}. All other aspects of the algorithm would remain unchanged.
That being said, while conceptually straightforward, we have found in the context of our examples that the inequality boundary conditions arising from hard constraints represent a significant difficulty for the robustness of our solution algorithm. 
Specifically, we observe that the neural network only learns an accurate solution if the weights on equation residuals ($\kappa^e$) and the boundary condition ($\kappa^b$) in the loss function are well-calibrated, which we find difficult to achieve for inequality constraints. 
Replacing the hard constraints by a soft penalty makes the method much more robust.

\subsubsection{Extension: Staggered updating\label{sec:solution:algorithm:extensions}}

In Algorithm~\ref{alg:generic}, when evaluating the master equation residuals to determine $\mathcal{E}^e$, we compute the policy $\hat{c}^\ast$ as a function of $\bV(\cdot;\Theta)$ according to equation~\eqref{eq:chat}. 
In principle, there is no need to parameterize the policy $\hat{c}^\ast$ with a separate neural network. However, in practice there are reasons for doing so, as suggested by \cite{duarte2024machine} and \cite{al2022extensions}.
First, if equation~\eqref{eq:chat} does not have a closed form solution, the computational cost of solving this equation numerically at each training point may be prohibitively high. In this case, we could include a separate neural network for $\hat{c}^\ast$ with separate parameters $\Theta^c$ and use this network in place of the true solution to equation~\eqref{eq:chat} to evaluate the master equation errors. 
This would necessitate adding steps into Algorithm~\ref{alg:generic} that involve updating the $\hat{c}^\ast$-network for a given $\bV$-network.
Second, even if equation~\eqref{eq:chat} can be solved, adding a separate neural network for $\hat{c}^\ast$ allows us to slow down the updating of the policy function akin to a ``Howard improvement algorithm''. To do so, we can make the updating of $\Theta^c$ infrequent, effectively fixing the same policy rule $\hat{c}^\ast$ for several iterations in the training of $\Theta$. In some implementations for our numerical example we use this variant of our baseline algorithm. We have found that this can help with stability, particularly when starting from a poor (e.g. random) initial guess.

\section{Example: Uninsurable Income Risk and TFP Shocks} \label{sec:krusell_smith}

\noindent A canonical macroeconomic model with heterogeneous agents and aggregate risk is \cite{Krusell1998}, which we refer to as the KS model.
In this section, we illustrate how our three solution approaches can be deployed to solve the continuous time version of the model described in \cite{achdou2022income} and \cite{Ahn2018}.

\subsection{Model Specification} \label{subsec:ks:model}

\noindent\emph{Setting:} 
There is a perishable consumption good and a durable capital stock with depreciation rate $\delta$.
The economy consists of a unit continuum $I=[0,1]$ of households and a representative firm.
The representative firm controls the production technology, which produces consumption goods according to the production function $Y_t = e^{z_t} K_t^{\alpha} L_t^{1-\alpha}$,
where $K_t$ is the capital rented at time $t$, $L_t$ is the labor hired at time $t$, and $z_t$ is the aggregate productivity, which follows an Ornstein-Uhlenbeck (OU) process $dz_t = \eta (\overline{z} - z_t) dt + \sigma dB_t^0$. \\

\noindent\emph{Heterogeneous households:}
Each household $i \in [0,1]$ has discount rate $\rho$ and gets flow utility $u(c_{t}^i) = (c_t^i)^{1-\gamma}/(1-\gamma)$ from consuming $c_{t}^i$ consumption goods at time $t$.
Each household has two idiosyncratic states $x_t^i = [a_t^i, l_t^i]$, where $a_t^i$ is the household's net wealth and $l_{t}^i \in \{l_1, l_2\}$ is the household's labor endowment, where $l_1 < l_2$ so $l_1$ is interpreted as unemployment and $l_2$ is interpreted as employment.
Labor endowments switch idiosyncratically between $l_1$ and $l_2$ at Poisson rate $\lambda(l_{t}^i)$.
Households choose consumption $c_t^i$ and their idiosyncratic state evolves according to:
\begin{align}
    dx_t^i ={}& d \begin{bmatrix} a_t^i \\ l_t^i \end{bmatrix} = \begin{bmatrix} s(a_t^i,l_t^i,c_t^i,r_t,w_t) \\ 0 \end{bmatrix} dt + \begin{bmatrix} 0 \\ \check{l}_t^i - l_t^i \end{bmatrix} dJ_t^i
\end{align}
where $r_t$ is the return on household wealth, $w_t$ is the wage rate, $\check{l}_t^i$ is the complement of $l_t^i$, $J_t^i$ is an idiosyncratic Poisson process with arrival rate $\lambda(l_t^i)$, and the agent's saving function is given by $s(a,l,c,r,w) = w l + r a - c$.
So, $g_t$ denotes the population density across $\{a_t^i, l_t^i\}$ at time $t$, given a filtration $\cF_t^0$ generated by the sequence of aggregate productivity shocks.\\

\noindent\emph{Assets, markets, and financial frictions:}
Each period, there are competitive markets for goods, capital rental, and labor.
We use goods as the numeraire.
We let $r_t$ denote the rental rate on capital, $w_t$ denote the wage rate on labor, and $q_t = [r_t, w_t]$ denote the price vector.
Given $g_t$ and $z_t$, firm optimization and market clearing imply that $r_t$ and $w_t$ solve: 
\begin{align}
    r_t {}&= \alpha e^{z_t}  K_t^{\alpha-1}L^{1-\alpha} - \delta, & w_t {}&= (1-\alpha)e^{z_t} K_t^\alpha L^{-\alpha}, \label{eq:market_clearing_agg} \\
    K_t {}&= \sum_{j \in \{1,2\}} \int_{\bR} a g_t(a, l_j) da & L {}&= \sum_{j \in \{1,2\}} \int_{\bR} l_j g_t(a, l_j) da.
\end{align}
where we assume that economy starts with the steady state labor distribution.
So, in the terminology of Section~\ref{sec:genericModel}, we can write the prices explicitly as functions of $(g_t, z_t)$:
\begin{align}
    q_t 
    = \begin{bmatrix} r_t \\ w_t \end{bmatrix} 
    = \begin{bmatrix} \alpha e^{z_t}  \left( \sum_{j \in \{1,2\}} \int_{\bR} a g_t(a, l_j) da\right)^{\alpha-1}L^{1-\alpha} - \delta \\ (1-\alpha)e^{z_t}\left( \sum_{j \in \{1,2\}} \int_{\bR} a g_t(a, l_j) da \right)^\alpha L^{-\alpha} \end{bmatrix}
    =: Q(g_t,z_t) \label{eq:Q}
\end{align}

Asset markets are incomplete so households cannot insure their idiosyncratic labor shocks.
Instead, households can trade claims to the aggregate capital stock in a competitive asset market.
The original \cite{Krusell1998} model imposes the ``borrowing constraint'' that each agent's net asset position, $a_{t}^i$, must satisfy $a_{t}^i \ge \underline{a}$, where $\underline{a}$ is an exogenous ``borrowing limit''.
This generates an inequality boundary constraint and a mass point, as discussed in \cite{achdou2022income}.
However, as discussed in Section \ref{subsec:master_eqn}, to help the neural network train more smoothly, we instead follow \cite{brzoza2015penalty} and introduce a penalty function $\psi$ at the left boundary, replacing the agent flow utility by $U(a_t,c_t) = u(c_t)+\textbf{1}_{a_t \leq a_{lb}}\psi(a_t)$.
The penalty function we use here is the quadratic function: $\psi(a)=-\frac{1}{2}\kappa(a-a_{lb})^2$ where $\kappa$ is a positive constant and $a_{lb} > \underline{a}$.\\

\noindent\emph{Master equation:}
Let $c^*((a, l), z, g)$ denote the equilibrium optimal household control.
Then, for the ABH model, the master equation \eqref{eq:master} becomes:
\begin{align}
    0 = \cL V((a,l),z,g) ={}& -\rho V((a,l),z,g) + u(c^*((a,l),z,g)) + \textbf{1}_{a \leq a_{lb}}\psi(a) \\
    {}& + (\cL_x + \cL_z + \cL_g) V((a,l),z,g) \label{eq:Master-KS}
\end{align}
where the operators become:
\begin{align}
    \cL_x V((a,l),z,g) ={}& \partial_a V((a, l), z, g)s((a,l),c^\ast((a,l),z,g),r(z,g),w(z,g))  \\
    {}& + \lambda(l)(V((a,\check{l}),z,g) - V((a,l),z,g)) \\
    \cL_z V((a,l),z,g) ={}& \partial_z V((a,l),z,g)\eta(\overline{z} - z) + \frac{1}{2} \sigma^2 \partial_{zz} V((a,l),z,g)  \\
    \cL_g V((a,l),z,g) ={}& \sum_{j \in \{1,2\}} \int_{\bR} D_{g_j} V((a,l),z,g)(b) \mu_g((b,l_j) z, g) db
\end{align}
where $D_{g_j}$ is the Frechet derivative with respect to the marginal density $g(\cdot,l_j)$ and the KFE is:
\begin{align}
    \mu_g((a,l),z,g) ={}& -\partial_a\left[ s\left( (a,l),c^\ast((a,l),z,g),r(z,g),w(z,g) \right) g(a,l)\right]\\
    {}& + \lambda(\check{l})g(a,\check{l}) - \lambda(l)g(a,l),
\end{align}
where $\check{l}$ denotes the complement of $l$,
where $r(z,g)$ and $w(z,g)$ solve the system of equations \eqref{eq:Q}, and the optimal consumption policy satisfies the first order optimality condition:
\begin{align}
    \partial_a V((a,l),z,g) = u'(c^\ast((a,l),z,g)).
\end{align}

In the next sections, we solve this master equation numerically using Algorithm~\ref{alg:generic}.
Because the optimal control is a function of the $\partial_a V((a,l),z,g)$, it will turn out to be more convenient to solve the master equation for the partial derivative, which we denote by $W((a,l),z,g) := \partial_a V((a,l),z,g)$.
The parameters that we use in numerical experiments are in Appendix~\ref{asec:implementation_details:parameters}.

\subsection{Implementation Details} \label{subsec:ks:implementation}

The implementation details for all methods are summarized in Table \ref{tab:comparion_implementation}.
Here we discuss key features about: the neural network structure, the sampling, the loss function, and the training.
We provide a more detailed description of the implementation in Appendix \ref{asec:implementation_details}.

\begin{table}[htbp]
\centering
\begin{small}
\renewcommand{\arraystretch}{1.0}
\begin{tabularx}{\textwidth}{lXXX}
\hline
 & Finite Population & Discrete State & Projection \\
\hline\hline
Neural Network & &  &  \\
\hline
(i) Structure & 
Fully connected feed-forward 
& Recurrent with embedding & Recurrent with embedding\\
(ii) Activation, $(\phi^{(i)})_{i\le H}$ & tanh & tanh & tanh \\
(ii) Output, $\phi^{(H+1)}$ & soft-plus & elu activation and factor $(a_0 + a)^{-\tilde{\eta}}$  & elu activation and factor $(a_0 + a)^{-\tilde{\eta}}$ \\
(iii) Layers, $H$ & 5 & recurrent: 3 \newline embedding: 2  & recurrent: 3 \newline embedding: 2 \\
(iv) Neurons, $\#|h|$ & 64 & recurrent: 100 \newline embedding: 128 & recurrent: 100 \newline embedding: 64 \\
(v) Initialization & $W(a,\cdot) = e^{-a}$ & random & random \\
(vi) Auxiliary networks & none & consumption & consumption\\
\hline
Sampling & &  &  \\
\hline
(i) $(a,l)$ & Active sampling $[a_{min}, a_{max}] \times \{ y_1, y_2\}$ & Uniform sampling $[a_{min}, a_{max}] \times \{ y_1, y_2\}$ & Uniform sampling $[a_{min}, a_{max}] \times \{ y_1, y_2\}$ \\
(ii) $(\hat{\varphi}_i)_{i \le N}$ & $\bullet$ Moment sampling: sample $r$ then random distribution of agents to generate $r$ & $\bullet$ Mixed steady-state sampling\newline $\bullet$ Ergodic sampling & $\bullet$ Moment sampling: sample $K$ then other coefficients uniformly\newline $\bullet$ Ergodic sampling \\
(iii) $z$ & $U[z_{min}, z_{max}]$ & $U[z_{min}, z_{max}]$ & $U[z_{min}, z_{max}]$ \\
\hline
Loss Function & &  &  \\
\hline
(i) Master equation & $\cE^e$ & $\cE^e$ & $\cE^e$ \\
(ii) Constraints & $\partial_a W(a,\cdot) < 0$ and $\partial_z W(a,\cdot) < 0$ & $\partial_a W(a,\cdot) < 0$ and $\partial_z W(a,\cdot) < 0$ & $\partial_a W(a,\cdot) < 0$ and $\partial_z W(a,\cdot) < 0$ \\
(iii) Weights & $\kappa^e = 100$, $\kappa^s = 1$ & $\kappa^e = \kappa^s =1$ & $\kappa^e = \kappa^s =1$ \\
\hline
Training & & & \\
\hline
(i) Learning rate & $10^{-4}$ & Decaying from \newline $3 \times 10^{-4}$ to $10^{-6}$ & Decaying from \newline $3 \times 10^{-4}$ to $10^{-6}$ \\
\hline
(ii) Optimizer & ADAM & ADAM & ADAM \\
\hline \hline
\end{tabularx}
\caption{Key Implementation Details} \label{tab:comparion_implementation}
\end{small}
\end{table}

First, consider the neural network structure. 
For the finite agent approximation, we use a ``plain vanilla'' fully connected feed-forward neural network with 5 layers and 64 neurons per layer. 
We choose a $\tanh$ activation function in all hidden layers and a softplus activation in the output layer to ensure that the output is always positive. 
For the discrete state and projection methods, we use an architecture that
feeds a fully connected feed-forward network (referred to as ``embedding'') into  the one proposed by~\cite{Sirignano2018} (referred to as ``recurrent'').
The former network is used to preprocess the distribution state $\hat{\varphi}$ before the result is passed to the main (recurrent) network.
The embedding portion has an output dimension of 10, 2 layers, and 128 neurons per layer for the discrete state space method and 64 neurons per layer for the projection method.
The recurrent portion of the network has 3 layers and 100 neurons per layer. 
We use a tanh activation function in all hidden layers and apply an elu activation in the last layer of the recurrent network to ensure positivity of the output. 
In addition, we multiply the neural network output by the factor $(a_{0}+a)^{-\tilde{\eta}}$, where $a_0,\tilde{\eta}\geq 0$ are non-trainable shape parameters. 
The additional factor aids training by reducing the amount of curvature in the marginal value function that must be captured by the neural network. 
Finally, in both the discrete state and the projection method, we use the variant of our solution algorithm with staggered updating of the consumption function and parameterize the latter by a separate auxiliary neural network that has the same structure as the one for the marginal value function with the exception that we do not multiply the network output by $(a_{0}+a)^{-\tilde{\eta}}$ in the output layer.

Second, consider the sampling.
We find moment sampling is sufficient for the finite agent method but not for the other techniques.
For the discrete state method, our main approach is ergodic sampling. 
As discussed in Section~\ref{sec:solution:algorithm:sampling}, this sampling scheme is only meaningful once a sufficiently good guess of the value function exists. 
We therefore start with mixed steady state sampling in early training and gradually increase the fraction of the training set taken from the ergodic sample.
For the projection method, we use a variant of moment sampling to sample capital stock $K$ from a uniform distribution. 
To be able to do this, we rotate the basis, such that the first basis vector points in the direction of increasing first moments (in the $a$-dimension) and all other basis vectors are orthogonal to the first. 
This rotation leaves the space of approximate distributions unaffected but leads to a one-to-one relationship between the aggregate capital stock and the first component of the distribution state $\hat{\varphi}$. 
We explain this in detail in Appendix~\ref{asec:master_equations}.
For the remaining components of $\hat{\varphi}$ we use a combination of uniform and ergodic sampling and gradually increase the fraction of the training set taken from the ergodic sample.

Third, consider the loss function.
In all cases, we impose concavity constraints on $V$, which are equivalent to monotonicity constraints on $W$.
In addition, we find it is very helpful to impose a sign constraint on $\partial_{z} W$.
This is because the KS model has limited curvature with respect to $z$ and so the neural network tends to find approximate solutions where the $z$ component is ignored.

Finally, consider the neural network training. In all cases, we choose an ADAM optimizer to perform the gradient descent steps. We fix the learning rate at $10^{-4}$ throughout for the finite agent approach, whereas we use a learning rate schedule with decaying learning rate (from $3\times 10^{-4}$ to $10^{-6}$) for the remaining approaches.

\subsection{Results\label{subsec:ks:stochastic_z}}

We solve the Krusell-Smith model using all three methods.
We first discuss the accuracy of the results for the full model.
We then discuss the model without aggregate shocks and possible ways to use neural network training for calibration.

\subsubsection{Mean Reverting TFP Process\label{subsec:ks:stochastic_z}}

For each approach, the error in the neural network approximation to the master equation is shown in Table \ref{tab:ks_loss} below.
Evidently all approaches have master equation losses to the approximate order of $10^{-5}$, which we interpret as convergence. In Appendix~\ref{sec:robustness} we show that all three methods are robust across different training runs.

\begin{table}[H]
    \renewcommand{\arraystretch}{1.0}
    \centering
    \begin{tabular}{lc}
     & Master equation training loss \\
     \hline
     Finite Agent NN & $3.037\times 10^{-5}$ \\
     Discrete State Space  NN & $9.639\times 10^{-5}$ \\
     Projection  NN & $8.506\times 10^{-6}$ \\
    \end{tabular}
    \caption{
    \small
    Neural Networks' final losses for solving the KS master equation. 
    }
    \label{tab:ks_loss}
\end{table}

In order to sense check our solution, we compare output from our neural networks to output from traditional approaches.
Unfortunately, we do not have a clear benchmark solution because there is no traditional technique that provides an arbitrarily precise solution to the model with aggregate shocks.
We choose to compare to the recent approach suggested by \cite{fernandez2023financial}, which uses a neural network to approximate a statistical law of motion but solves the Master equation using a finite difference scheme.
It is understood that this technique is good approximation for the KS model.
We make our comparison by computing sample paths from all of our solution approaches.
For the discrete state space and projection methods, we can generate sample paths by iterating the approximate equilibrium KFE.
For the finite agent method, computing the sample path is more complicated because we need to average over the transition matrices generated by many draws from the finite population.
We describe this in detail in Algorithm \ref{alg:transition_KS} in the Appendix \ref{subsec:KS_simulation}.

Figure \ref{fig:ks:fa:transition} offers a visual inspection of the difference between our neural network solutions and \cite{fernandez2023financial} for a random path of productivity shocks.
The upper-left panel shows the draw from the Ornstein-Uhlenbeck process: $dz_t = \eta(\overline{z}-z)dt + \sigma d B^0_t$.
The upper-right compares the evolution of capital stock.
The second row compares the evolution of prices.
The third and fourth rows compare the population distribution at various times in the simulation.
Evidently we get a similar path for all variables across all the methods.

In Figure~\ref{fig:ks:fan_chart_main}, we generate multiple random TFP paths, $z_t$, and show the evolution of our neural network solutions and the \cite{fernandez2023financial} solution in a ``fan chart'' that displays percentiles for the evolution of the population.
In particular, we generate 1000 TFP paths starting from $z_0 = 0$ and calculate the corresponding aggregate capital evolution paths. 
At each time $t$, for each solution, we compute the pth-percentile of the capital stock across simulation paths.
We then plot the time series of the percentiles for p equal to 10\%, 30\%, 50\%, 70\%, and 90\%.
Evidently, the finite agent and projection methods are a close match to the \cite{fernandez2023financial}.
The discrete state space method does a good job at central percentiles but has more difficulty at the extremes.
This reflects our general experience that the discrete state space approximation is the most difficult to work with for the KS model and relies most heavily on ergodic sampling, which decrease it accuracy away from the ergodic mean.

\begin{figure}[hbtp]
    \centering
    \includegraphics[width=1.\textwidth]{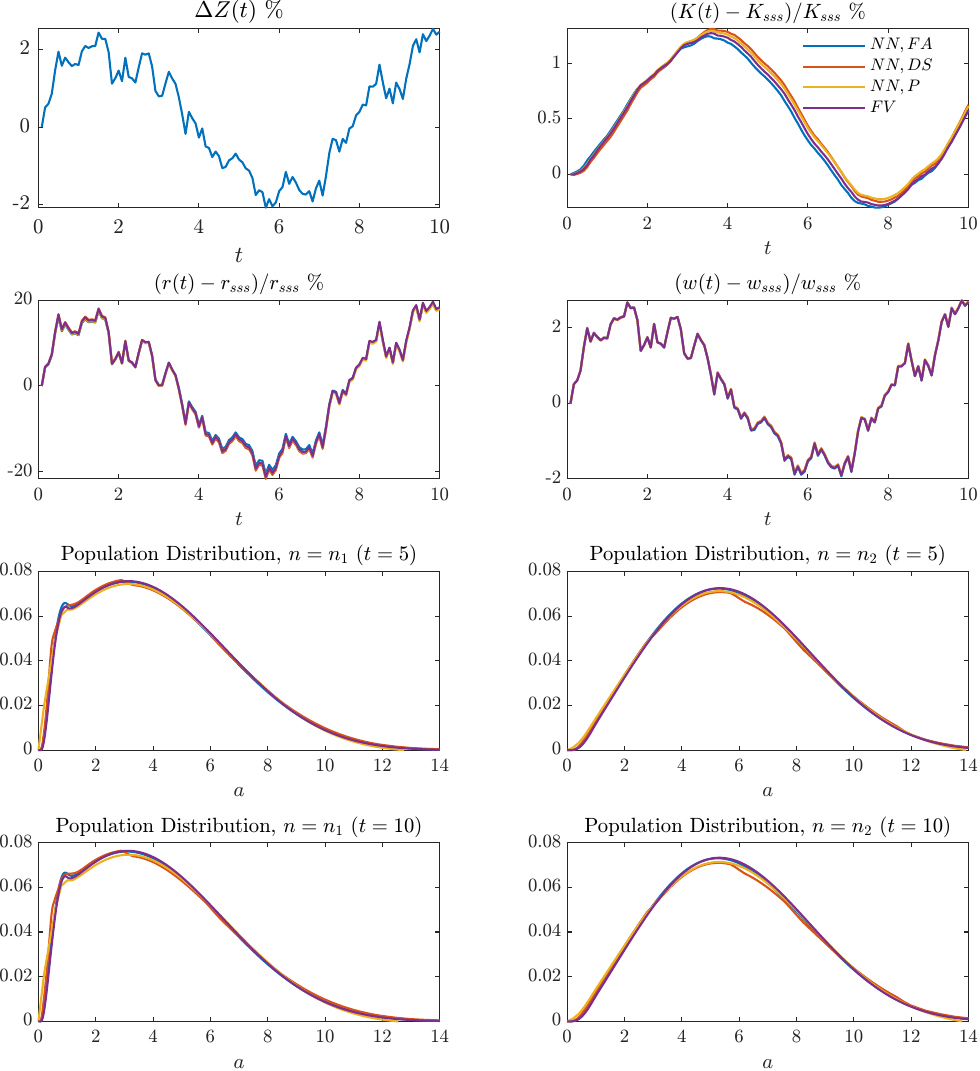}
    \caption{\small Simulations for the Krusell-Smith Model. The top left plot is the TFP shock path, the top right panel is the aggregate relative capital change.
    The second row left plot shows the relative change in the capital return and the second row right plot shows the relative change in the wage rate.
    The plots on rows three and four show the distribution at different times in the simulation.
    The labels \textit{``NN, FA''}, \textit{``NN, DS''}, and \textit{``NN, P''} refer to solutions from the finite agent, discrete state, and projection neural networks respectively. \emph{``FV''} refers to the solution from \cite{fernandez2023financial}. Subscript \textit{sss} refers to the stochastic steady state at $z=0$.}
    \label{fig:ks:fa:transition}
\end{figure}

\begin{figure}[ht]
    \centering
    \includegraphics[width = 1.\textwidth]{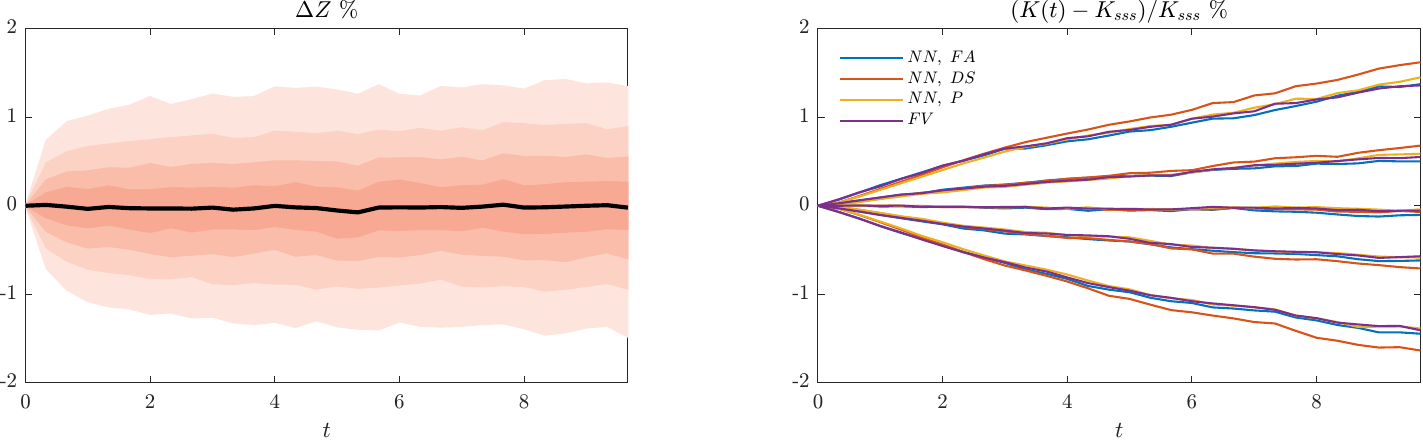}
    \caption{\small Forecasted aggregate capital dynamics starting from the stochastic steady state (\textit{sss}) for the Krusell-Smith Model. The left plot is the fan chart for the TFP shock path, generated from the OU process with initial condition $z_0 = 0$. The right panel is the time series plot for relative change in aggregate capital at percentiles 10\%, 30\%, 50\%, 70\%, 90\% (from the lowest to the highest). The labels \textit{``NN, FA''}, \textit{``NN, DS''}, and \textit{``NN, P''} refer to solutions from the finite agent, discrete state, and projection neural networks respectively. \emph{``FV''} refers to the solution from \cite{fernandez2023financial}.}
    \label{fig:ks:fan_chart_main}
\end{figure}

\subsubsection{Calibration}
As has been suggested by a number of papers (e.g. \cite{duarte2024machine}), a potential benefit of deep learning algorithms is that we can include the parameters, $\zeta$, as additional inputs into the neural network, $\hat{V}(\hat{X}, \zeta) \approx \bV(\hat{X}, \zeta; \Theta)$,
and then train the neural network using sampling from both $\hat{X}$ and $\zeta$.
In principle, this means that we can train the model once to get a solution across the parameter space and then use $\bV$ to calculate the moments for different parameters to calibrate the model.
In practice, this requires that the distribution approximation and sampling do not have high dependence on the solution.
This means that, in our example, the finite agent method is the natural candidate for testing this approach because it only requires moment sampling in the training.
In Appendix \ref{subsec:calibration}, we provide a simple example that calibrates the borrowing constraint to match a particular capital-to-labor ratio.
For the other methods this would be much more difficult because their training requires ergodic sampling.

\subsubsection{Fixed Aggregate Productivity} \label{subsec:ks:fixed_z}

For fixed aggregate productivity, $z_t = z$, our example model is the continuous time version of the Aiyagari-Bewley-Huggett (ABH) model discussed in \cite{achdou2022income}.
This model has a precise finite difference solution and so acts as a more detailed ``check'' for the accuracy of our solution technique.
Table \ref{tab:aiy_loss} plots the mean squared error (MSE) between our steady state solution and the finite difference solution.
Figure \ref{fig:aiy_fd_comp} plots the steady state consumption policy rule, value function derivative, probability density function (pdf), and cumulative distribution function (cdf) for the solutions from the finite agent, discrete state space, and finite difference methods.
It does not make sense to compare to the steady state for the projection method because the basis has been constructed from steady state solution.
Evidently, the neural network solutions align very closely to the finite difference solution.

\begin{table}[H]
    \centering
    \begin{tabular}{lcc}
     & Master equation loss & MSE(NN, FD)\\
     \hline
     Finite Agent NN & $3.135\times 10^{-5}$ & $4.758\times 10^{-5}$\\
     Discrete State Space NN &$9.303\times 10^{-6}$ &$6.591\times 10^{-5}$
    \end{tabular}
    \caption{\small Neural Networks' final losses for solving the ABH master equation. 
    Master equation loss is the mean squared error of residuals. 
    MSE(NN,FD) is the mean squared difference between steady state consumption from the neural network and finite difference solutions on a wealth grid.
    }
    \label{tab:aiy_loss}
\end{table}

\begin{figure}[h]
    \centering
    \includegraphics[width = 1.0\textwidth]{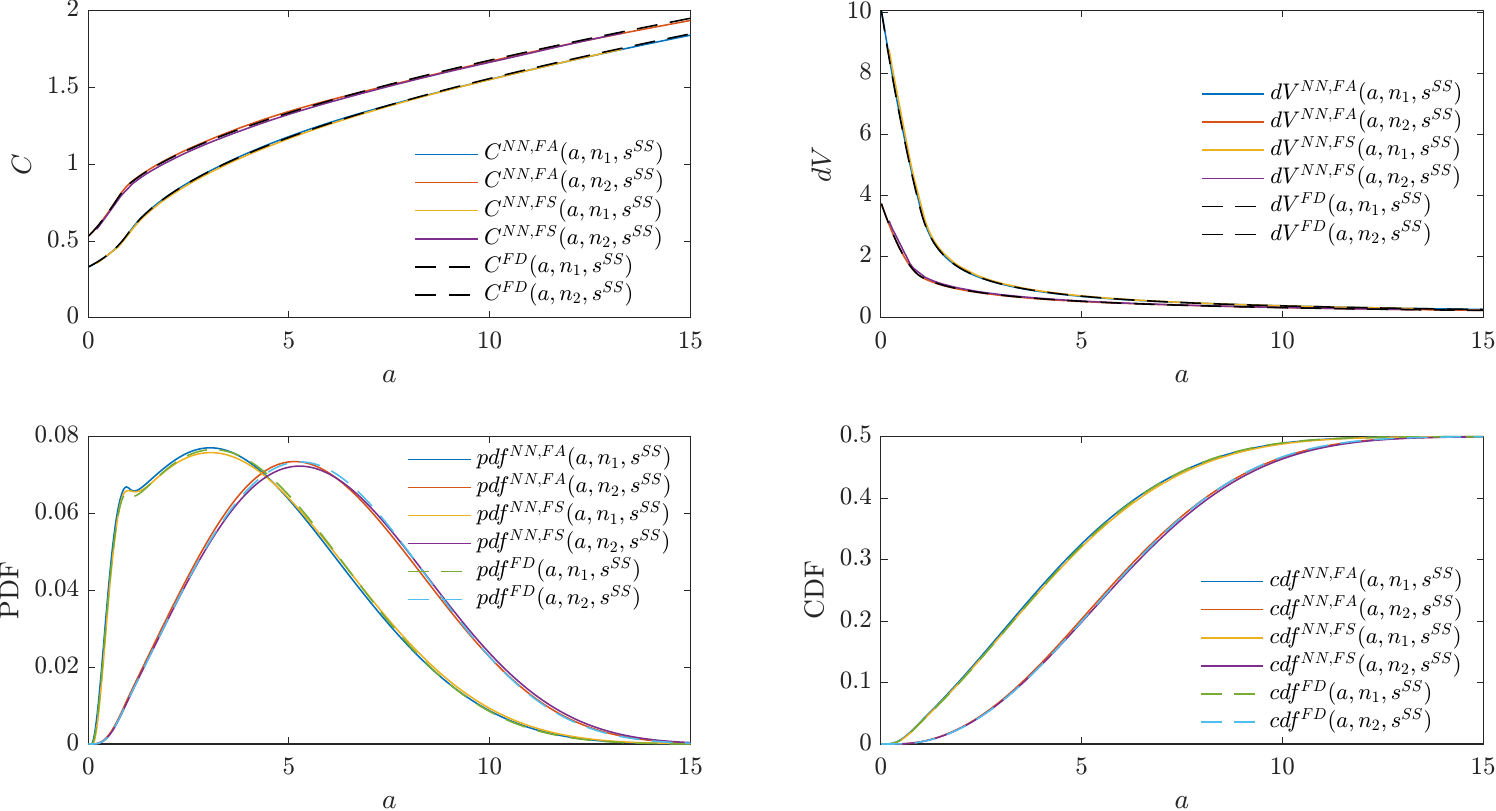}
    \caption{\small Comparison between the neural network and finite difference solutions for the Aiyagari model. The top left plot shows the consumption policy. The top right shows the derivative of the value function, the bottom left shows the pdf, and the bottom right shows the cdf. The labels \emph{``NN,FA''}, \emph{``NN,FS''}, and \emph{``FD''} refers to solutions from the finite agent neural network, the discrete state space neural network, and finite difference respectively.}
    \label{fig:aiy_fd_comp}
\end{figure}

We can also consider the transition path following an unexpected shock to aggregate
productivity (a so-called “MIT” shock). This makes little sense for the discrete state space approximation and the projection method because both rely on types of ergodic sampling and so have difficulty with unanticipated shocks.
However, we were able to train our finite agent approximation using moment sampling and so it makes sense to consider
how successfully the neural network approximation can handle transition paths.
We do the comparison in Appendix \ref{subsubsec:transition_dynamics_ABH} and find it is remarkably successful even for large shocks.

\subsection{Comparison of Techniques}

Although the different distribution approximation approaches can all be effective, we find they have different strengths and weaknesses for solving the KS model.
The finite-agent method is very robust in a number of ways: the neural network can be trained with the moment sampling procedure, the algorithm only requires a moderate number of agents (approximately 40), and we can successfully add parameters as auxiliary states so the model can be solved across the state and parameter space at the same time.

By contrast, we find the discrete-state method to be difficult to work with.
We believe many of the issues come from having to approximate the derivatives in the KFE on the discrete state space.
One challenge is that this requires a fine grid for agent wealth (approximately 200 grid points in our example).
A related challenge is that the training samples need to come from relatively smooth densities and so ergodic sampling is very important.
Ultimately, this makes training the KS model using the discrete state method slow and complicated.
In Section \ref{subsec:spatial}, we consider a spatial model with locations that are ex-ante different and agents that choose where to locate instead of having a consumption saving decision.
This means the model is effectively impossible to solve using the finite agent method but ends up being straightforward to solve when there are a discrete number of locations.
Follow up work in \cite{payne2024DeepSAM} extends our approach to search and matching models and also finds that the discrete state approximation is very effective for that class of models.
This offers suggestive evidence that the discrete state method is most helpful when the KFE does not contain complicated derivatives.

Finally, the projection method brings a different set of trade-offs.
Both the finite agent and discrete state methods feed relatively large state spaces into the neural network and then let the neural network work out how to structure the approximate value function.
By contrast, the projection method requires more choices ex-ante.
This allows us to work with a much lower dimensional approximation (approximately 5 basis functions) and to choose which statistics of the distribution we want to match most closely in our approximation.
For these reasons, we believe the projection method offers new advantages for the macroeconomics deep learning literature.

\section{Additional Examples} \label{sec:add_examples}

We close the paper by considering two additional examples: a model of firms with capital adjustment costs (similar to a continuous time version \cite{khan2007inventories}) and a dynamic spatial model (similar to the model solved using perturbation in \cite{bilal2021solving}).
We have chosen these examples to help understand the power of the different techniques.
The first model is well suited to a finite agent approximation while the second model can only be solved by using a discrete set of locations.

\subsection{Heterogeneous Firms With Adjustment Costs} \label{subsec:kahn_thomas}

\noindent\emph{Setting.} There is a perishable consumption good and durable capital stock.
The economy consists of a representative household and a continuum of firms $i \in [0,1]$ that invest in capital and hire labor.
There are competitive markets for goods, labor, firm equity and risk free bonds but capital is not tradable.
We use goods as the numeraire and denote the labor wage by $w_t$, the bond interest rate by $r_t$, and the price of equity in firm $i$ by $p_t^i$.
The representative household chooses consumption, $C_t$, labour supply, $L_t$, and the wealth invested in firm $i$ equity, $E_t^i$, to solve:
\begin{align}
    \max_{C_t, L_t, \{e^i\}} \mathbb E \int_0^\infty e^{-\rho t} U(C_t, L_t) dt, \quad \text{where} \quad U(C_t, L_t) = \frac{1}{1-\gamma} \left(C_t- \chi \frac{L_t^{1+\varphi}}{1+\varphi}\right)^{1-\gamma}
\end{align}
and where $\rho$ is the continuous time discount factor, $\gamma$ is the coefficient of relative risk aversion, $\chi$ determines the disutility of labor, $\varphi$ is the Frisch elasticity of labor supply, and the household is subject to the budget constraint:
\begin{align}
    {dA_t} = - C_t dt + w_tL_t dt + \left[\int_{i} E_{t}^i \left(\frac{\pi_{t}^idt + dp_{t}^i}{p_{t}^i}\right)di\right], \quad \text{where }\int_i E_t^i di =A_t.
\end{align}
where $\pi_t^i$ is firm profit.
So, following standard analysis, their stochastic discount factor (or ``state-price'') and labor supply are given by:
\begin{align}
    \Lambda_t ={}& \partial_C U(C_t, L_t) = \left( C_t - \chi \frac{L_t^{1+\varphi}}{1+\varphi} \right)^{-\gamma}, & L_t ={}& \left( \frac{w}{\chi} \right)^{1/\varphi}
\end{align}

\noindent\emph{Heterogeneous Firms:} Each firm $i\in[0,1]$ has a production function $e^{z_t} \epsilon_t^{i}(k_t^i)^\theta (l_t^i)^\nu$,
where $z_t$ is aggregate TFP following a mean reverting process $dz_t = \eta (\overline{z} - z_t) dt + \sigma dB_t^0$ and $\epsilon_t^i$ is the idiosyncratic shock taking values from $\{\epsilon_L,\epsilon_H\}$, with switching rate $\lambda_L,\lambda_H$ respectively. 
The firm can pay dividends or invest to accumulate capital but faces the adjustment cost $\psi(n,k) = \frac{\chi_1}{2}\frac{n^2}{k}$, where $n$ is investment.
So, each firm has idiosyncratic state $x_t^i = [k_t^i, \epsilon_t^i]$, which evolves according to:
\begin{align}
    dx_t^i ={}& d \begin{bmatrix} k_t^i \\ \epsilon_t^i \end{bmatrix} = \begin{bmatrix} n_t^i - \delta k_t^i \\ 0 \end{bmatrix} dt + \begin{bmatrix} 0 \\ \check{\epsilon}_t^i - \epsilon_t^i \end{bmatrix} dJ_t^i. \label{eq:firm:lomx}
\end{align}
where $J_t^i$ is the idiosyncratic Poisson process.
Taking the wage rate and the households SDF, $\Lambda_t$, as given, firm $i$ chooses labor, $l_t^i$, and investment, $n_t^i$, to maximize their price of equity:
\begin{align}
    \max_{l^i, n^i} \left\{ p_0^i = \bE_0 \int_0^{\infty} \frac{\Lambda_t}{\Lambda_0} e^{-\rho t} 
    \pi_t^i
    dt \right\} \quad s.t. \quad \eqref{eq:firm:lomx}
\end{align}
where $\pi_t^i := e^{z_t} \epsilon_t^{i}(k_t^i)^\theta (l_t^i)^\nu - w_t l_t^i - \left(n_t^i + \psi(n_t^i,k_t^i) \right)$ is firm profit.
So, for this model $g_t$ denotes the population density across $[k_t^i, \epsilon_t^i]$ at time t, given a filtration $\cF_t^0$ generated by the sequence of aggregate productivity shocks. \\

\noindent \emph{Market Clearing.} 
Market clearing for goods, labor, and firm equity implies: %
\begin{align}
    C_t ={}& \sum_{\epsilon=\epsilon_L,\epsilon_H}\int\left[
    e^{z_t} \epsilon_t(k_t)^\theta (l_t)^\nu - \left(n_t + \psi(n_t,k_t) \right)
    \right]g_t(k,\epsilon)dk \label{eq:firms:mc_c}  \\
    L_t ={}& \sum_{\epsilon=\epsilon_L,\epsilon_H}\int l_t(k,\epsilon)g_t(k,\epsilon)dk, 
    \quad E_t^i = p_t^i \label{eq:firms:mc_l}
\end{align}
So, referring back to our general notation, the price vector that can be expressed explicitly in terms of $(z,g)$ is $q_t = [w_t, \Lambda_t]$, which satisfies:
\begin{align}
    w_t ={}& \left( \sum_{\epsilon=\epsilon_L,\epsilon_H}\int\left(\frac{1}{\nu z_t\epsilon k^\theta}\right)^{\frac{\varphi}{\nu -1}}g_t(k,\epsilon)dk \right)^{\frac{1-\nu}{1+\varphi-\nu}}, \label{eq:firm:wage} \\
    \Lambda_t ={}& \left( \sum_{\epsilon=\epsilon_L,\epsilon_H}\int\left[
    e^{z} \epsilon(k)^\theta (l_t(k,\epsilon))^\nu - n_t(k,\epsilon) - \psi(n_t(k,\epsilon),k)
    \right]g_t(k,\epsilon)dk - \frac{ \chi\left(\frac{w_t}{\chi}\right)^{\frac{1+\varphi}{\varphi}}}{1+\varphi} \right)^{-\gamma} \label{eq:firm:lambda} 
\end{align}
The price of equity in firm $i$ is the value function of firm $i$, $p_t^i = V_t^i(k_t,\epsilon_t)$, which we need to solve for numerically. \\

\noindent \emph{Master equation.}
The aggregate states are $(z_t,g_t)$.
We let $V_t(k,\epsilon) = V(k,\epsilon, z_t,g_t)$ denote the value function in recursive form.

\begin{proposition}[Master Equation]\label{prop:master_eqn:firms}
Let $l^*((k,\epsilon),z,g)$ and $n^*((k,\epsilon),z,g)$ denote the optimal firm policy functions.
Then the master equation \eqref{eq:master} is given by:
\begin{align}
    0 = {}& -\rho V((k,\epsilon),z,g) + \Lambda(z,g) \pi_i(l^*((k,\epsilon),z,g),n^*((k,\epsilon),z,g),k,\epsilon,w(z,g)) \\
    {}& + (\cL_x + \cL_z + \cL_g) V((k,\epsilon),z,g) \label{eq:ME:firm_dynamics}
\end{align}
where the operators are defined as:
\begin{align}
    \cL_x V((k,\epsilon),z,g) ={}& \partial_k V((k,\epsilon),z,g)(n^*((k,\epsilon),z,g)-\delta k) \\
    {}& + \lambda(\epsilon) (V((k,\check{\epsilon}),z,g)-V((k,\epsilon),z,g)) \\
    \cL_z V((k,\epsilon),z,g) ={}& \partial_zV((k,\epsilon),z,g)\eta(\bar{z}-z)+\frac{1}{2} \partial_{zz} V ((k,\epsilon),z,g) \sigma_Z^2 \\
    \cL_g V((k,\epsilon),z,g) ={}& \sum_{j \in \{1,2\}} \int_{\bR} D_{g_j} V((k,\epsilon),z,g)(b) \mu_g((b,\epsilon_j) z, g) db
\end{align}
where $D_{g_j}$ is the Frechet derivative with respect to the marginal $g(\cdot,l_j)$
and the KFE is:
\begin{align}
    \mu_g((k,\epsilon),z,g) ={}& -\partial_k\left[ \left( n^*((k,\epsilon),z,g) - \delta k \right) g(k,\epsilon)\right] + \lambda(\check{\epsilon})g(k,\check{\epsilon}) - \lambda(\epsilon)g(k,\epsilon),
\end{align}
The optimal firm policies satisfy:
\begin{align}
    n^*((k,\epsilon),z,g) ={}&\frac{k}{\chi_1}(\partial_kV((k,\epsilon),z,g)-1), &
    l^*((k,\epsilon),z,g) ={}&\left(\frac{w}{\nu z \epsilon k^\theta}\right)^{\frac{1}{\nu -1}}
\end{align}

\end{proposition}
\begin{proof}
    See Appendix \ref{asubsec:firms}.
\end{proof}

For computation, it is convenient to define the scaled value function $\breve{V} := V/\Lambda$.
Then, the new Master equation that we take to the computer is:
\begin{align}
    \breve{\cL}^h \breve{V} (k,\epsilon,z,g) = \cL^h \breve{V} (k,\epsilon,z,g) + \frac{\mu_{\Lambda}(z,g)}{\Lambda(z,g)} \breve{V}(k,\epsilon,z,g) + \partial_{z}\breve{V}(k,\epsilon,z,g)\partial_{z}\Lambda(z,g) \sigma_z^2
\end{align}
The main technical difference compared to the KS master equation is that the drift of the household SDF appears in the effective discount rate.
This introduces additional feedback which makes the master equation harder train. As a result, we start the training without $\mu^{\Lambda}$ and then introduce the term once the value function has started to converge.
We discuss the details in Appendix \ref{asubsec:firms}.
We train the model and get master equation loss of $7.408\times10^{-6}$.
In Figure \ref{fig:Kahn-Thomas-results}, we show two plots from the solution: the investment policy rule and the ergodic distribution for firms in the low and high state.

\begin{figure}[H]
    \centering
    \includegraphics[width = \textwidth]{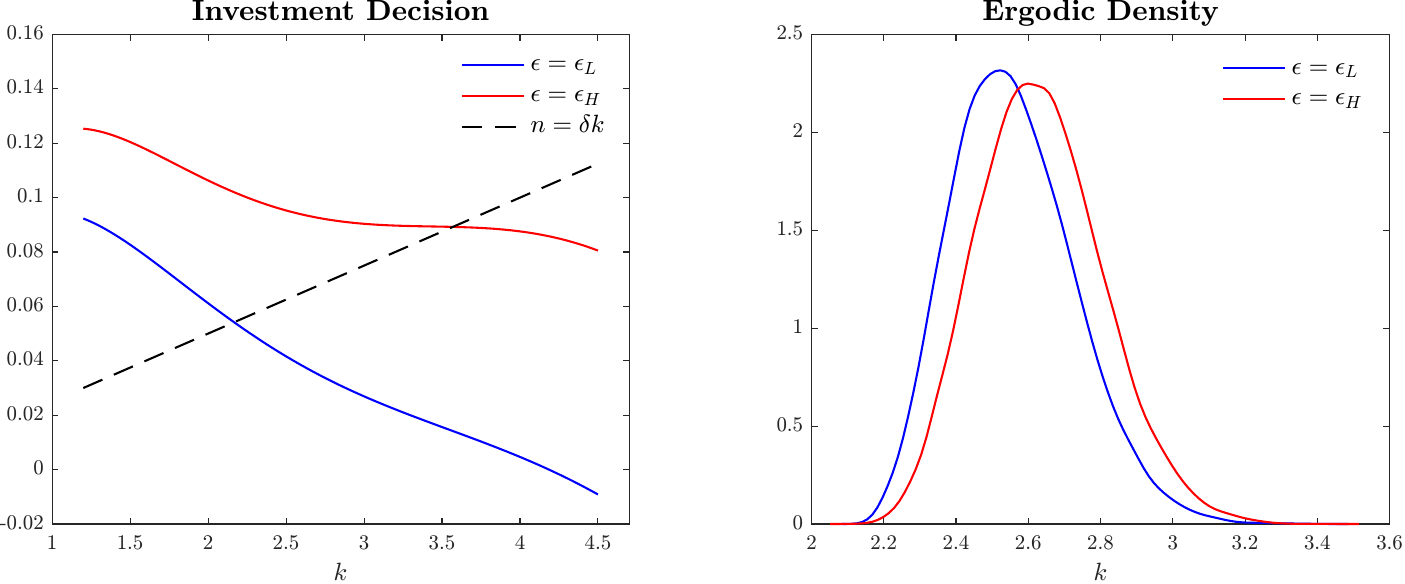}
    \caption{\small Illustration of model solution for the firm dynamics model. The left panel depicts the investment policy in the stochastic steady state and the depreciation line (dashed) as a function of firm's own capital level. The right panel depicts the (marginal) ergodic density for high and low productivity firms.%
    }
    \label{fig:Kahn-Thomas-results}
\end{figure}

\subsection{Dynamic Spatial Model} \label{subsec:spatial}

\noindent\emph{Setting.}\footnote{
    We describe here a streamlined version of the model specified in \cite{bilal2021solving} that does not include housing.
    These changes are without loss of generality in the sense that, up to a change of parameters, our resulting master equation is isomorphic to \cite{bilal2021solving}.} 
The economy consists of a finite set of locations $j\in \{1,...,J\}$, a continuum of workers $i \in [0,1]$, and a representative competitive firm at each location $j$ with an owner who consumes their profits.
At any given date $t$, each worker $i$ resides at a specific location $j$. 
Workers receive infrequent moving opportunities. There is a perishable consumption good, tradable across locations. 
The representative firm at location $j$ produces consumption goods according to the production function $Y_{j,t} = \exp(\beta_j  + \chi_j z_t) L_{j,t}^{1-\alpha}$, where $L_{j,t}$ is the labor hired from workers residing at location $j$. As in the KS model, $z_t$ is the aggregate productivity, which follows the same process as in Section~\ref{subsec:ks:model}. $\beta_j>0$ is a time-invariant location-specific productivity shifter and $\chi_j>0$ is a parameter that governs the sensitivity of production in location $j$ to variation in aggregate productivity $z_t$.\\

\noindent\emph{Heterogeneous workers:}
Each worker $i\in [0,1]$ has discount rate $\rho$ and gets flow utility $u(c_{t}^i) = (c_t^i)^{1-\gamma}/(1-\gamma)$ from consuming $c_{t}^i$ consumption goods at time $t$. 
The worker is endowed with one unit of labor, which they supply inelastically in the labor market at location $j_t^i$ where they currently reside. 
The worker has no access to financial markets and therefore consumes wage income each period, $c_t^i = w_{j_t^i,t}$, where $w_{j,t}$ denotes the wage at location $j$ and time $t$.
The worker's idiosycratic state $x_t^i = j_t^i$ consists solely of the location of residence. 
With arrival rate $\mu>0$, the worker receives an idiosyncratic opportunity to move location. 
Upon receiving such an opportunity, the worker draws idiosyncratic i.i.d.\ additive preference shocks for each potential destination $j^\prime$ according to a Gumbel distribution with mean $0$ and inverse scale parameter $\nu$ and then chooses one location $j^\prime$ as their new residence.
When moving from $j$ to $j^\prime$, the worker also incurs a moving disutility of $\tau_{j,j^\prime}\geq0$. The structure of this location choice problem implies that, after optimization, $j_t^i$ follows a, possibly time-inhomogeneous, continuous-time Markov chain with transition rate $\mu \pi_{j,j^\prime,t}$ from $j$ to $j^\prime$ where
$$\pi_{j,j^\prime,t} = \pi_{j,j^\prime}(V_t) := \frac{e^{\nu (V_{t}(j^\prime)-\tau_{j,j^\prime})}}{\sum_{k=1}^{J}e^{\nu (V_{t}(k)-\tau_{j,k})}}.$$
Here, $V_t(j)$ denotes the value function of a worker with idiosyncratic state $j$. The distribution $g_t$ is the population density across $\{j_t^i\}$ at time $t$, given a filtration $\cF_t^0$ generated by the sequence of aggregate productivity shocks.\\

\noindent\emph{Assets, markets, and financial frictions:}
Each period, there are competitive markets for goods and, in each location, for labor. We use goods as the numeraire. The price vector $q_t= [w_{j,t}: j=1,...,J]$ is given by the collection of all local wages. Given $g_t$ and $z_t$, firm optimization and market clearing imply that the wages $w_{j,t}$ solve: 
\begin{align}
     w_{j,t} {}&= w_j(z_t,g_t) := (1-\alpha)\exp(\beta_j  + \chi_j z_t)(g_t(j))^{-\alpha}.
\end{align}

\noindent \emph{Master equation.}
The aggregate states are $(z_t,g_t)$.
We let $V_t(j) = V(j,z_t,g_t)$ denote the value function in recursive form.

\begin{proposition}[Master Equation]\label{prop:spatial_model:master_equation}
The master equation \eqref{eq:master} is given by:\begin{align}
    0 = - \rho V(j,z,g) + u(w_j(z,g)) + (\cL_x + \cL_z + \cL_g) V(j,z,g) \label{eq:ME:spatial_model}
\end{align}
where the operators are defined as:
\begin{align*}
    \cL_x V(j,z,g) &= \mu\left(\frac{1}{\nu}\log \left(\sum_{j^\prime=1}^{J}e^{\nu(V(j^\prime,z,g)-\tau_{j,j^\prime})}\right)-V(j,z,g)\right)\\
    \cL_z V(j,z,g) &= \partial_z V(j,z,g)\eta(\overline{z} - z) + \frac{1}{2} \sigma^2 \partial_{zz} V(j,z,g)\\
    \cL_g V(j,z,g) &= \sum_{j^{\prime}=1}^J\partial_{g(j^{\prime})}V(j,z,g)\mu_{g}(j^\prime,z,g),
\end{align*}
the KFE is:
\begin{align}
   \mu_g(j, z, g) = \mu\left(\sum_{k=1}^J\pi_{k,j}(V(\cdot,z,g))g(k)-g(j)\right),
\end{align} 
and the conditional moving probabilities $\pi_{j^\prime\prime,j^\prime}$ are as defined previously. Up to a transformation of model parameters, this master equation is equivalent to the one in \citet{bilal2021solving}.
\end{proposition}
\begin{proof}
    See Appendix~\ref{asubsec:spatial:model_solution}.
\end{proof}

\noindent \emph{Model solution.} Because the model's idiosyncratic state space is discrete and all individual density values $g_t(j)$ matter directly for prices, the discrete state space method is the most appropriate for this model. We solve the model with $J=50$ locations.
We provide further details on the parameterization and numerical implementation in Appendices~\ref{asubsec:spatial:parameters} and \ref{asubsec:spatial:implementation}, respectively.
Here, we only note that, with regard to moving costs $\tau_{jj^\prime}$, we create a simple example based on a ``cluster structure'' with a central cluster and several periphery clusters that makes is more costly to move from/to some locations (the periphery) than others.  The terminal loss after training is $8.403\times 10^{-05}$, similar in magnitude to the losses for our KS example. Further details on the loss decay are provided in Appendix~\ref{asubsec:spatial:loss_decay}.

\begin{figure}[hbtp]
    \centering
    \includegraphics[width=1.\textwidth]{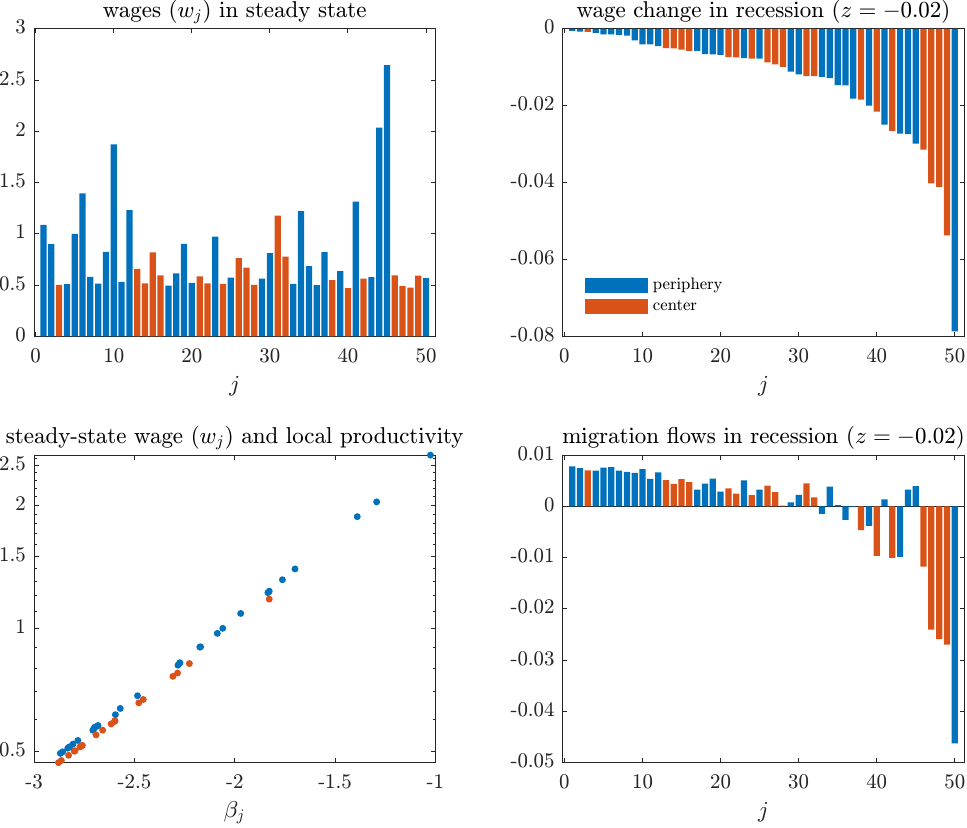}
    \caption{\small Illustration of model solution for the dynamic spatial model. The left panels depict the wage distribution in the stochastic steady state, $g=g^{sss}$, $z=0$, as a function of location $j$ (top) and location-specific productivity $\beta_{j}$ (bottom). The right panels depict the impact effects of a hypothetical ``recession shock'' that moves the aggregate state to $(z,g)=(-0.02,g^{sss})$. The top right panel shows the relative change in wages in a recession and the bottom right panel the resulting net migration flows as a proportion of the population ($\mu_{g,j,t}/g_t(j)$). In all panels but the bottom left, locations $j$ are sorted by their sensitivity to aggregate productivity ($\chi_j$) in ascending order. Blue bars/dots depict periphery locations and red bars/dots depict central locations.}
    \label{fig:spatial:result_illustration}
\end{figure}

Figure~\ref{fig:spatial:result_illustration} depicts some aspects of the computed model solution.
The figure illustrates wages by location in the stochastic steady state (top left panel) and how wages depend on the local productivity (bottom left panel) as well as how a recession (reduction in $z_t$) affects wages and population flows.
Except for the scatter plot, locations in the figure are sorted by their sensitivity to the aggregate productivity shock $\chi_j$. Note that there are several dimensions of ex-ante heterogeneity across locations, so that the non-monotonic variation in the plots does not indicate training errors. The overall picture that emerges appears to be economically meaningful. Wages in harder-hit locations fall more in a recession (top right panel). As a result, migration flows in the recession are mostly directed from high-$\chi_j$ to low-$\chi_j$ regions (bottom right). The exception are primarily very productive regions in which wages are high to begin with. The cluster structure of moving costs matters mainly for the baseline level of wages which tends to be lower in high-productivity regions (bottom left): these regions attract more workers due to a higher option value of moving.

\section{Conclusion}

This paper proposes a new algorithm that uses deep learning to globally characterize numerical solutions to continuous time heterogeneous agent economies with aggregate shocks.
We demonstrate our algorithm by solving canonical models in the macroeconomics literature.
Although deep learning algorithms are straightforward to describe, we find that implementing them successfully requires careful consideration of the training details.
We close by collecting some practical lessons from our experiences:
\begin{enumerate}
    \item Working out the correct sampling approach is very important.
    \item Neural networks find it easier to work with smooth constraints.
    \item Enforcing shape constraints helps with speed and stability.
    \item Starting with a simple, solvable model is helpful for tuning hyperparameters.
\end{enumerate}
Ultimately, we believe this is only the beginning what these techniques can achieve.
We have demonstrated that we can now globally solve continuous time heterogeneous agent economies with aggregate shocks, which opens up new and exciting areas of research.

\singlespacing
\spacing{0.97}

\bibliographystyle{apalike}

\bibliography{aux/library.bib}

\begin{thebibliography}{}

\bibitem[Achdou et~al., 2022a]{achdou2022income}
Achdou, Y., Han, J., Lasry, J.-M., Lions, P.-L., and Moll, B. (2022a).
\newblock Income and wealth distribution in macroeconomics: A continuous-time
  approach.
\newblock {\em The Review of Economic Studies}, 89(1):45--86.

\bibitem[Achdou et~al., 2022b]{achdou2022simulating}
Achdou, Y., Lasry, J.-M., and Lions, P.~L. (2022b).
\newblock Simulating numerically the {K}rusell-{S}mith model with neural
  networks.
\newblock {\em arXiv preprint arXiv:2211.07698}.

\bibitem[Ahn et~al., 2018]{Ahn2018}
Ahn, S., Kaplan, G., Moll, B., Winberry, T., and Wolf, C. (2018).
\newblock When inequality matters for macro and macro matters for inequality.
\newblock {\em NBER macroeconomics annual}, 32(1):1--75.

\bibitem[Aiyagari, 1994]{Aiyagari1994}
Aiyagari, S.~R. (1994).
\newblock Uninsured idiosyncratic risk and aggregate saving.
\newblock {\em The Quarterly Journal of Economics}, 109(3):659--684.

\bibitem[Al-Aradi et~al., 2022]{al2022extensions}
Al-Aradi, A., Correia, A., Jardim, G., de~Freitas~Naiff, D., and Saporito, Y.
  (2022).
\newblock Extensions of the deep {G}alerkin method.
\newblock {\em Applied Mathematics and Computation}, 430:127287.

\bibitem[Alvarez and Lippi, 2022]{alvarez2022analytic}
Alvarez, F. and Lippi, F. (2022).
\newblock The analytic theory of a monetary shock.
\newblock {\em Econometrica}, 90(4):1655--1680.

\bibitem[Alvarez et~al., 2023]{alvarez2023price}
Alvarez, F., Lippi, F., and Souganidis, P. (2023).
\newblock Price setting with strategic complementarities as a mean field game.
\newblock {\em Econometrica}, 91(6):2005--2039.

\bibitem[Auclert et~al., 2021]{auclert2021using}
Auclert, A., Bard{\'o}czy, B., Rognlie, M., and Straub, L. (2021).
\newblock Using the sequence-space {J}acobian to solve and estimate
  heterogeneous-agent models.
\newblock {\em Econometrica}, 89(5):2375--2408.

\bibitem[Azinovic et~al., 2022]{azinovic2022deep}
Azinovic, M., Gaegauf, L., and Scheidegger, S. (2022).
\newblock Deep equilibrium nets.
\newblock {\em International Economic Review}, 63(4):1471--1525.

\bibitem[Azinovic and {\v{Z}}emli{\v{c}}ka, 2023]{azinovic2023economics}
Azinovic, M. and {\v{Z}}emli{\v{c}}ka, J. (2023).
\newblock Economics-inspired neural networks with stabilizing homotopies.
\newblock {\em arXiv preprint arXiv:2303.14802}.

\bibitem[Barnett et~al., 2023]{barnett2023deep}
Barnett, M., Brock, W., Hansen, L.~P., Hu, R., and Huang, J. (2023).
\newblock A deep learning analysis of climate change, innovation, and
  uncertainty.
\newblock {\em arXiv preprint arXiv:2310.13200}.

\bibitem[Bayraktar et~al., 2018]{bayraktar2018numerical}
Bayraktar, E., Budhiraja, A., and Cohen, A. (2018).
\newblock A numerical scheme for a mean field game in some queueing systems
  based on {M}arkov chain approximation method.
\newblock {\em SIAM Journal on Control and Optimization}, 56(6):4017--4044.

\bibitem[Bensoussan et~al., 2015]{bensoussan2015master}
Bensoussan, A., Frehse, J., and Yam, S. C.~P. (2015).
\newblock The master equation in mean field theory.
\newblock {\em Journal de Math{\'e}matiques Pures et Appliqu{\'e}es},
  103(6):1441--1474.

\bibitem[Bertucci and Cecchin, 2022]{bertucci2022mean}
Bertucci, C. and Cecchin, A. (2022).
\newblock Mean field games master equations: from discrete to continuous state
  space.
\newblock {\em arXiv preprint arXiv:2207.03191}.

\bibitem[Bhandari et~al., 2023]{bhandari2023perturbational}
Bhandari, A., Bourany, T., Evans, D., and Golosov, M. (2023).
\newblock A perturbational approach for approximating heterogeneous-agent
  models.

\bibitem[Bilal, 2023]{bilal2021solving}
Bilal, A. (2023).
\newblock Solving heterogeneous agent models with the master equation.
\newblock Technical report, National Bureau of Economic Research.

\bibitem[Bilal and Rossi-Hansberg, 2023]{bilal2023anticipating}
Bilal, A. and Rossi-Hansberg, E. (2023).
\newblock Anticipating climate change across the united states.
\newblock Technical report, National Bureau of Economic Research.

\bibitem[Bretscher et~al., 2022]{bretscher2022ricardian}
Bretscher, L., Fern{\'a}ndez-Villaverde, J., and Scheidegger, S. (2022).
\newblock Ricardian business cycles.
\newblock {\em Available at SSRN}.

\bibitem[Brzoza-Brzezina et~al., 2015]{brzoza2015penalty}
Brzoza-Brzezina, M., Kolasa, M., and Makarski, K. (2015).
\newblock A penalty function approach to occasionally binding credit
  constraints.
\newblock {\em Economic Modelling}, 51:315--327.

\bibitem[Cardaliaguet et~al., 2019]{Cardaliaguet2015}
Cardaliaguet, P., Delarue, F., Lasry, J.-M., and Lions, P.-L. (2019).
\newblock {\em The master equation and the convergence problem in mean field
  games:(ams-201)}.
\newblock Princeton University Press.

\bibitem[Carmona and Lauri{\`e}re, 2021]{Carmona2021}
Carmona, R. and Lauri{\`e}re, M. (2021).
\newblock Convergence analysis of machine learning algorithms for the numerical
  solution of mean field control and games {I}: the ergodic case.
\newblock {\em SIAM Journal on Numerical Analysis}, 59(3):1455--1485.

\bibitem[Carmona and Lauri{\`e}re, 2022]{Carmona2022}
Carmona, R. and Lauri{\`e}re, M. (2022).
\newblock Convergence analysis of machine learning algorithms for the numerical
  solution of mean field control and games: {II}—the finite horizon case.
\newblock {\em The Annals of Applied Probability}, 32(6):4065--4105.

\bibitem[Cohen et~al., 2024]{cohen2024deep}
Cohen, A., Lauri{\`e}re, M., and Zell, E. (2024).
\newblock Deep backward and galerkin methods for the finite state master
  equation.
\newblock {\em arXiv preprint arXiv:2403.04975}.

\bibitem[Delarue et~al., 2020]{delarue2020master}
Delarue, F., Lacker, D., and Ramanan, K. (2020).
\newblock From the master equation to mean field game limit theory: {L}arge
  deviations and concentration of measure.
\newblock {\em Annals of Probability}, 48(1):211--263.

\bibitem[{Den Haan}, 1997]{DenHaan1997}
{Den Haan}, W. (1997).
\newblock {Solving Dynamic Models with Aggregrate Shocks and Heterogeneous
  Agents}.
\newblock {\em Macroeconomic Dynamics}, 1(2):355--386.

\bibitem[Duarte et~al., 2024]{duarte2024machine}
Duarte, V., Duarte, D., and Silva, D. (2024).
\newblock Machine learning for continuous-time finance.

\bibitem[Fern{\'{a}}ndez-Villaverde et~al., 2018]{Fernandez-Villaverde2018}
Fern{\'{a}}ndez-Villaverde, J., Hurtado, S., and Nu{\~{n}}o, G. (2018).
\newblock {Financial Frictions and the Wealth Distribution}.
\newblock {\em Working Paper}, pages 1--51.

\bibitem[Fern{\'a}ndez-Villaverde et~al., 2023]{fernandez2023financial}
Fern{\'a}ndez-Villaverde, J., Hurtado, S., and Nuno, G. (2023).
\newblock Financial frictions and the wealth distribution.
\newblock {\em Econometrica}, 91(3):869--901.

\bibitem[Fouque and Zhang, 2020]{fouque2020deep}
Fouque, J.-P. and Zhang, Z. (2020).
\newblock Deep learning methods for mean field control problems with delay.
\newblock {\em Frontiers in Applied Mathematics and Statistics}, 6:11.

\bibitem[Germain et~al., 2022a]{germain2022deepsets}
Germain, M., Lauri{\`e}re, M., Pham, H., and Warin, X. (2022a).
\newblock Deep{S}ets and their derivative networks for solving symmetric
  {PDE}s.
\newblock {\em Journal of Scientific Computing}, 91(2):63.

\bibitem[Germain et~al., 2022b]{germain2022numerical}
Germain, M., Mikael, J., and Warin, X. (2022b).
\newblock Numerical resolution of mckean-vlasov fbsdes using neural networks.
\newblock {\em Methodology and Computing in Applied Probability}, pages 1--30.

\bibitem[Glorot and Bengio, 2010]{GlorotBengio2010}
Glorot, X. and Bengio, Y. (2010).
\newblock Understanding the difficulty of training deep feedforward neural
  networks.
\newblock In {\em Proceedings of the thirteenth international conference on
  artificial intelligence and statistics}, pages 249--256. JMLR Workshop and
  Conference Proceedings.

\bibitem[Goodfellow et~al., 2016]{goodfellow2016deep}
Goodfellow, I., Bengio, Y., and Courville, A. (2016).
\newblock {\em Deep learning}.
\newblock MIT press.

\bibitem[Gopalakrishna, 2021]{Gopalakrishna2021}
Gopalakrishna, G. (2021).
\newblock Aliens and continuous time economies.
\newblock {\em Swiss Finance Institute Research Paper}, (21-34).

\bibitem[Gopalakrishna et~al., 2024]{GopalakrishnaGuPayne2024}
Gopalakrishna, G., Gu, Z., and Payne, J. (2024).
\newblock Asset pricing, participation constraints, and inequality.
\newblock {\em Princeton Working Paper}.

\bibitem[Hadikhanloo and Silva, 2019]{hadikhanloo2019finite}
Hadikhanloo, S. and Silva, F.~J. (2019).
\newblock Finite mean field games: fictitious play and convergence to a first
  order continuous mean field game.
\newblock {\em Journal de Math{\'e}matiques Pures et Appliqu{\'e}es},
  132:369--397.

\bibitem[Han et~al., 2018]{han2018solving}
Han, J., Jentzen, A., and E, W. (2018).
\newblock Solving high-dimensional partial differential equations using deep
  learning.
\newblock {\em Proceedings of the National Academy of Sciences},
  115(34):8505--8510.

\bibitem[Han et~al., 2021]{han2021}
Han, J., Yang, Y., and E, W. (2021).
\newblock {DeepHAM}: A global solution method for heterogeneous agent models
  with aggregate shocks.
\newblock {\em arXiv preprint arXiv:2112.14377}.

\bibitem[Hu and Lauriere, 2022]{hu2022recent}
Hu, R. and Lauriere, M. (2022).
\newblock Recent developments in machine learning methods for stochastic
  control and games.
\newblock {\em ssrn.4096569}.

\bibitem[Huang, 2023a]{huang2023breaking}
Huang, J. (2023a).
\newblock Breaking the curse of dimensionality in heterogeneous-agent models: A
  deep learning-based probabilistic approach.
\newblock {\em Available at SSRN 4649043}.

\bibitem[Huang, 2023b]{huang2023probabilistic}
Huang, J. (2023b).
\newblock A probabilistic solution to high-dimensional continuous-time macro
  and finance models.

\bibitem[Huang and Yu, 2024]{huang2024applications}
Huang, J. and Yu, J. (2024).
\newblock Applications of deep learning-based probabilistic approach to
  “combinatorial” problems in economics.

\bibitem[Kahou et~al., 2021]{kahou2021}
Kahou, M.~E., Fern{\'a}ndez-Villaverde, J., Perla, J., and Sood, A. (2021).
\newblock Exploiting symmetry in high-dimensional dynamic programming.
\newblock Technical report, National Bureau of Economic Research.

\bibitem[Kaplan et~al., 2018]{kaplan2018monetary}
Kaplan, G., Moll, B., and Violante, G.~L. (2018).
\newblock Monetary policy according to hank.
\newblock {\em American Economic Review}, 108(3):697--743.

\bibitem[Khan and Thomas, 2007]{khan2007inventories}
Khan, A. and Thomas, J.~K. (2007).
\newblock Inventories and the business cycle: An equilibrium analysis of (s, s)
  policies.
\newblock {\em American Economic Review}, 97(4):1165--1188.

\bibitem[Khan and Thomas, 2008]{Khan2008}
Khan, A. and Thomas, J.~K. (2008).
\newblock {Idiosyncratic shocks and the role of nonconvexities in plant and
  aggregate investment dynamics}.
\newblock {\em Econometrica}, 76(2):395--436.

\bibitem[Krusell and Smith, 1998]{Krusell1998}
Krusell, P. and Smith, A.~A. (1998).
\newblock {Income and Wealth Heterogeneity in the Macroeconomy}.
\newblock {\em Journal of Political Economy}, 106(5):867--896.

\bibitem[Lacker, 2020]{lacker2020convergence}
Lacker, D. (2020).
\newblock On the convergence of closed-loop {N}ash equilibria to the mean field
  game limit.
\newblock {\em Annals of applied probability: an official journal of the
  Institute of Mathematical Statistics}, 30(4):1693--1761.

\bibitem[Li et~al., 2022]{li2022}
Li, J., Yue, J., Zhang, W., and Duan, W. (2022).
\newblock The deep learning {G}alerkin method for the general stokes equations.
\newblock {\em Journal of Scientific Computing}, 93(1):1--20.

\bibitem[Lions, 2011]{LionsCDF}
Lions, P.-L. (2007-2011).
\newblock Lectures at {C}ollege de {F}rance.

\bibitem[Lu et~al., 2021]{lu2021deepxde}
Lu, L., Meng, X., Mao, Z., and Karniadakis, G.~E. (2021).
\newblock {DeepXDE}: A deep learning library for solving differential
  equations.
\newblock {\em SIAM review}, 63(1):208--228.

\bibitem[Maliar et~al., 2021]{Maliar2021}
Maliar, L., Maliar, S., and Winant, P. (2021).
\newblock Deep learning for solving dynamic economic models.
\newblock {\em Journal of Monetary Economics}, 122:76--101.

\bibitem[Min and Hu, 2021]{min2021signatured}
Min, M. and Hu, R. (2021).
\newblock Signatured deep fictitious play for mean field games with common
  noise.
\newblock In {\em International Conference on Machine Learning}, pages
  7736--7747. PMLR.

\bibitem[Payne et~al., 2024]{payne2024DeepSAM}
Payne, J., Rebei, A., and Yang, Y. (2024).
\newblock Deep learning for search and matching models.
\newblock {\em Princeton Working Paper}.

\bibitem[Perrin et~al., 2022]{perrin2022generalization}
Perrin, S., Lauri{\`e}re, M., P{\'e}rolat, J., {\'E}lie, R., Geist, M., and
  Pietquin, O. (2022).
\newblock Generalization in mean field games by learning master policies.
\newblock In {\em Proceedings of the AAAI Conference on Artificial
  Intelligence}, volume~36, pages 9413--9421.

\bibitem[Prohl, 2017]{Prohl2017}
Prohl, E. (2017).
\newblock {Discetizing the Infinite-Dimensional Space of Distributions to
  Approximate Markov Equilibria with Ex-Post Heterogeneity and Aggregate Risk}.

\bibitem[Raissi et~al., 2017]{raissi2017physics}
Raissi, M., Perdikaris, P., and Karniadakis, G.~E. (2017).
\newblock Physics informed deep learning (part i): Data-driven solutions of
  nonlinear partial differential equations.
\newblock {\em arXiv preprint arXiv:1711.10561}.

\bibitem[Reiter, 2002]{Reiter2002}
Reiter, M. (2002).
\newblock {Recursive computation of heterogeneous agent models}.
\newblock {\em manuscript, Universitat Pompeu Fabra}, (July):25--27.

\bibitem[Reiter, 2008]{Reiter2008}
Reiter, M. (2008).
\newblock {Solving heterogeneous-agent models by projection and perturbation}.
\newblock {\em Journal of Economic Dynamics and Control}, 33:649--665 Contents.

\bibitem[Reiter, 2009]{reiter2009solving}
Reiter, M. (2009).
\newblock Solving heterogeneous-agent models by projection and perturbation.
\newblock {\em Journal of Economic Dynamics and Control}, 33(3):649--665.

\bibitem[Reiter, 2010]{Reiter2010}
Reiter, M. (2010).
\newblock {Approximate and Almost-Exact Aggregation in Dynamic Stochastic
  Heterogeneous-Agent Models}.
\newblock Technical report, Vienna Institute for Advanced Studies.

\bibitem[Sauzet, 2021]{sauzet2021projection}
Sauzet, M. (2021).
\newblock Projection methods via neural networks for continuous-time models.
\newblock {\em Available at SSRN 3981838}.

\bibitem[Schaab, 2020]{Schaab2020}
Schaab, A. (2020).
\newblock Micro and macro uncertainty.
\newblock {\em Available at SSRN 4099000}.

\bibitem[Sirignano and Spiliopoulos, 2018]{Sirignano2018}
Sirignano, J. and Spiliopoulos, K. (2018).
\newblock {DGM}: A deep learning algorithm for solving partial differential
  equations.
\newblock {\em Journal of computational physics}, 375:1339--1364.

\bibitem[Sznitman, 1991]{sznitman1991topics}
Sznitman, A.-S. (1991).
\newblock Topics in propagation of chaos.
\newblock {\em Lecture notes in mathematics}, pages 165--251.

\bibitem[Winberry, 2018]{Winberry2018}
Winberry, T. (2018).
\newblock A method for solving and estimating heterogeneous agent macro models.
\newblock {\em Quantitative Economics}, 9(3):1123--1151.

\end{thebibliography}

\newpage
\appendix

\singlespacing
\onehalfspacing
\spacing{1.1}
\section{Supplementary Proofs For Section \ref{sec:genericModel} (Online Appendix)} \label{asec:genericModel}

\begin{proof}[Proof of the Kolmogorov Forward Equation]
We derive the KFE by studying the dynamics in a finite agent population and then taking the limit as the number of agents goes to infinite (the so called ``propogation of chaos'' technique).
For notational convenience, in our working, we assume that all variables are continuous.
The end formula extends naturally to discrete variables if the integrals in the discrete dimensions are interpreted as sums and all derivatives in that dimension are set to zero.

\underline{Step 1: Set up problem with a finite number of agents:}
Suppose that there are $N < \infty$ agents. Then, we define the following empirical density and inner product:
\begin{align*}
  \hat{g}_t^N :=& \frac{1}{N} \sum_{i=1}^N \delta_{x_t^i} &
  \langle \phi(\cdot), \hat{g}_t^N \rangle :=& \frac{1}{N} \sum_{i=1}^N \phi(x_t^i),
\end{align*}
where test functions $\phi(\cdot)$ are bounded, twice differentiable and $\phi(x),D_x\phi(x)$ vanishes at the boundary of $\cX$, which we denote by $\partial \cX$.
We define the finite agent equilibrium price as $\hat{q}_t^N$.
We define the limiting cases by:
\begin{align*}
  g_t :=& \lim_{N\rightarrow \infty} \frac{1}{N} \sum_{i=1}^N \delta_{x_t^i}, 
  &
  \langle \phi(\cdot), g_t \rangle&{}  = \int_{\cX} \phi(x)  g_t(x) dx.
\end{align*}

\underline{Step 2: Apply It\^o's Lemma:} Taking It\^o's Lemma we get that the following (since the idiosyncratic state $x_t^i$ is not directly exposed to aggregate shocks): 
\begin{align*}
  d\langle \phi(\cdot), \hat{g}_t^N \rangle =  \frac{1}{N} \sum_{i=1}^N d\phi(x_t^i)
  ={}& \frac{1}{N} \sum_{i=1}^N D_x \phi(x_t^i) \dotp \mu_x(c_t^i,x_t^i,z_t,\hat{q}_t^N) dt \\
  {}&+ \frac{1}{2N} \sum_{j=1}^N \tr \Big\{ \Sigma_x (c_t^i,x_t^i,z_t,\hat{q}_t^N) D_x^2 \phi(x_t^i) \Big\} dt \\
  {}& + \frac{1}{N} \sum_{i=1}^N (\phi(x_t^i + \varsigma_x(c_t^i,x_t^i,z_t,\hat{q}_t^N)) - \phi(x_t^i)) dJ_{t}^i\\
  {}&+ \frac{1}{N} \sum_{i=1}^N D_x \phi(x_t^i) \dotp \sigma_x(c_t^i,x_t^i,z_t,\hat{q}_t^N) dB_t^i
\end{align*}
where we have used the vector and trace notation (rather than the summation notation used in the main text) to streamline the expressions.

\underline{Step 3: Take the limit as $N \rightarrow \infty$:} 
We assume sufficient regulatory that the law or large numbers applies and so the limits as we take $N \rightarrow \infty$ becomes the cross sectional averages:
\begin{align}
    \frac{1}{N} \sum_{i=1}^N D_x \phi(x_t^i) \dotp \mu_x(c_t^i,x_t^i,z_t,\hat{q}_t^N)dt \rightarrow{}& \int_{\cX} D_x \phi(x) \dotp \mu_x(c^\ast(x,z_t,g_t),x,z_t,q_t) g_t(x) dx dt \\
    \frac{1}{2N} \sum_{j=1}^N \tr \Big\{ \Sigma_x (c_t^i,x_t^i,z_t,\hat{q}_t^N) D_x^2 \phi(x_t^i) \Big\} dt {}& \\
    {}& \hspace{-4.0cm} \rightarrow \frac{1}{2} \int_{\cX} \tr \Big\{ \Sigma_x(c^\ast(x,z_t,g_t),x,z_t,q_t) D_x^2 \phi(x) \Big\} g_t(x)dx dt \\
    \frac{1}{N} \sum_{i=1}^N (\phi(x_t^i + \varsigma_x(c_t^i,x_t^i,z_t,\hat{q}_t^N)) - \phi(x_t^i)) dJ_{t}^i  {}& \\ {}& \hspace{-4.0cm} \rightarrow \int_{\cX}\lambda(x,z_t)(\phi(x+ \varsigma_x(c^\ast(x,z_t,g_t),x,z_t,q_t)) - \phi(x))g_t(x) dx dt \\
    \frac{1}{N} \sum_{i=1}^N D_x \phi(x_t^i) \dotp \sigma_x(c_t^i,x_t^i,z_t,\hat{q}_t^N) dB_t^i \rightarrow{}& 0
\end{align}
and so we get:
\begin{align}
  d\langle \phi(\cdot), g_t \rangle ={}&  
  \left[\int_{\cX} D_x \phi(x) \dotp \mu_x(c^\ast(x,z_t,g_t),x,z_t,q_t) g_t(x) dx \right] dt \\
    {}&+ \frac{1}{2}\left[ \int_{\cX} \tr \Big\{ \Sigma_x(c^\ast(x,z_t,g_t),x,z_t,q_t) D_x^2 \phi(x) \Big\} g_t(x)dx\right]dt \\
    {}& +\left[\int_{\cX}\lambda(x,z_t)(\phi(x+ \varsigma_x(c^\ast(x,z_t,g_t),x,z_t,q_t)) - \phi(x))g_t(x) dx \right]  dt \label{eq:step3_end}
\end{align}

\underline{Step 4: ``Intregration by Parts'':}
By the multidimensional integration by parts (formally the Stokes-Cartan theorem) and the assumptions on $\phi(\cdot)$, the first term on the RHS of \eqref{eq:step3_end} can be expressed as:
\begin{align*}
  \Bigg[\int_{\cX} {}& D_x \phi(x) \dotp \mu_x(c^\ast(x,z_t,g_t),x,z_t,q_t) g_t(x) dx \Bigg] dt \\
    ={}& \Bigg[ \int_{\partial \cX} \phi(x) \mu_x(c^\ast(x,z_t,g_t),x,z_t,q_t)g_t(x)dx - \int_{\cX} \diverge [ \mu_x(c^\ast(x,z_t,g_t),x,z_t,q_t) g_t(x) ] \phi(x) dx \Bigg]dt \\
    ={}& - \Bigg[ \int_{\cX} \diverge [ \mu_x(c^\ast(x,z_t,g_t),x,z_t,q_t) g_t(x) ] \phi(x) dx \Bigg]dt
\end{align*}
Applying Stokes-Cartan theorem twice, the second term on the RHS of \eqref{eq:step3_end} can be simplified as:
\begin{align}
    \frac{1}{2} \left( \int_{\cX} \tr \Big\{ D_x^2 \Big(\Sigma_x(c^\ast(x,z_t,g_t),x,z_t,q_t) g_t(x) \Big)\Big\} \phi(x) dx\right) dt
\end{align}
Finally, the third term on the RHS of \eqref{eq:step3_end} can be expressed as:
\begin{align*}
  &\left[\int_{\cX}\lambda(x,z_t)(\phi(x+ \varsigma_x(c^\ast(x,z_t,g_t),x,z_t,q_t)) - \phi(x))g_t(x) dx \right]  dt  \\
  =&\left[\int_{\cX}\lambda(x,z_t)\phi(x+ \varsigma_x(c^\ast(x,z_t,g_t),x,z_t,q_t))g_t(x)dx - \int_{\cX}\lambda(x,z_t)\phi(x)g_t(x) dx \right]  dt.
\end{align*}
Let $\breve{\varsigma}$ be defined to be such that (i.e. $\breve{\varsigma}$ is part of the inverse satisfying the following relationship):
\begin{align*}
  x+ \varsigma_x(c^\ast(x,z,g),x,z,Q(z,g)) = y \quad \Leftrightarrow \quad x = y - \breve{\varsigma}(y,z,g),
\end{align*}
where we have plugged in $q=Q(z,g)$ to streamline notation.
Now, make the change of variables. We have that:
\begin{align*}
  {}&\left[\int_{\cX}\lambda(y,z_t)\phi(y)g_t(y - \breve{\varsigma}(y,z_t,g_t))|I-D_y\breve{\varsigma}(y,z_t,g_t)|dy - \int_{\cX}\lambda(x,z)\phi(x)g_t(x) dx \right] dt \\
  {}& \hspace{1cm} = \left[\int_{\cX} \lambda(x,z_t)\Big( g_t(x - \breve{\varsigma}(x,z_t,g_t))|I-D_x\breve{\varsigma}(x,z_t,g_t)| - g_t(x) \Big) \phi(x) dx \right]  dt,
\end{align*}
where $|I-D_x\breve{\varsigma}(x,z_t,g_t)|$ is the determinant of matrix $I-D_x\breve{\varsigma}(x,z_t,g_t)$.
Putting everything back together, we get that:
\begin{align*}
  {}& \int_{\cX} (\partial_t g_t(x)) \phi(x) dx \\
  {}& \hspace{1cm} =\int_{\cX} \Big(- \diverge [ \mu_x(c^\ast(x,z_t,g_t),x,z_t,Q(z_t,g_t)) g_t(x) ] \Big) \phi(x) dx \\
   {}& \hspace{1cm}  + \frac{1}{2}\int_{\cX} \tr \Big\{ D_x^2 \Big(\Sigma_x(c^\ast(x,z_t,g_t),x,z_t,Q(z_t,g_t)) g_t(x) \Big)\Big\} \phi(x) dx \\
   {}& \hspace{1cm} + \int_{\cX} \Big( \lambda(x,z_t) \big( g_t(x - \breve{\varsigma}(x,z_t,g_t))|I-D_x\breve{\varsigma}(x,z_t,g_t)| - g_t(x) \big) \Big) \phi(x) dx.
\end{align*}
Because this equation must hold for all test functions $\phi$, it follows that $g_t$ must satisfy the KFE~\eqref{eq:generic-KFE} stated in the main text for almost all $x\in\cX$.

\end{proof}

\section{Details on Section \ref{sec:krusell_smith} (Online Appendix)} \label{asec:implementation_details}

\subsection{Parameters} \label{asec:implementation_details:parameters}

\renewcommand{\arraystretch}{1.0}
\begin{table}[H]
\centering
\begin{tabular}{lcc}
\hline
Parameter & Symbol & Value  \\
\hline
Capital share & $\alpha$ & $1/3$  \\
Depreciation & $\delta$ & $0.1$ \\
Risk aversion & $\gamma$ & $2.1$ \\
Discount rate & $\rho$ & $0.05$ \\
Mean TFP & $\overline{Z}$ & $0.00$ \\
Reversion rate & $\eta$ & $0.50$ \\
Volatility of TFP & $\sigma$ & $0.01$ \\
Transition rate (1 to 2) & $\lambda_1$ & $0.4$ \\
Transition rate (2 to 1) & $\lambda_2$ & $0.4$ \\
Low labor productivity & $l_1$ & 0.3 \\
High labor productivity & $l_2$ & $1+\lambda_2/\lambda_1 (1-n_1)$ \\
Penalty Function & $\psi(a)$ & $-\frac{1}{2}\kappa(a-a_{lb})^2$ \\
Penalty parameters  & $a_{lb}$ & $1.0$ \\
Penalty parameters  & $\kappa$ & $3.0$ \\
Borrowing constraint & $\underline{a}$ & $0.0$ \\
Minimum of assets (for sampling) & $a_{min}$ & $10^{-6}$ \\
Maximum of assets (for sampling) & $a_{max}$ & $20.0$ \\
Minimum TFP (for sampling)  & $z_{min}$ & $-0.04$ \\
Maximum TFP (for sampling)  & $z_{max}$ & $0.04$ \\
\hline
\end{tabular}
\caption{Parameters for Krusell-Smith Model from Section \ref{subsec:ks:stochastic_z}}
\end{table}

\onehalfspacing

\subsection{Detail on the Master Equations} \label{asec:master_equations}

In this subsection of the Appendix, we describe the precise master equations that we use to train the neural network for each approach.
Let $\hat{W}((a,l),z,\hat{\varphi}) := \partial_a \hat{V}((a,l),z,\hat{\varphi})$ denote the derivative of the value function with respect to $a$.\\

\noindent\emph{Finite agent approximation:}
In this case, we replace the distribution by the positions of the agents $\hat{\varphi} = \{(a^i,l^i)\}_{i\le I}$
where $I = 41$ agents.
We use the envelope theorem to take the derivative of the HJBE.
The resulting finite dimensional master equation is given by:
\begin{align}
    0 = \cL \hat{W}((a^i,l^i), z, \hat{\varphi}) ={}& (r(z,\hat{\varphi}^{-i})-\rho) \hat{W}((a^i,l^i), z, \hat{\varphi}) + \textbf{1}_{a \leq a_{lb}} \partial_a \psi(a^i) \\
    {}& + (\cL_x + \cL_z + \hat{\cL}_g) \hat{W}((a^i,l^i),z,\hat{\varphi}) \label{eq:Master-KS}
\end{align}
where the operators become:
\begin{align}
    \cL_x \hat{W}((a^i,l^i), z, \hat{\varphi}) ={}& 
    s((a^i,l^i),\hat{c}^\ast((a^i,l^i),z, \hat{\varphi}),r(z,\hat{\varphi}^{-i}),w(z,\hat{\varphi}^{-i})) \partial_{a^i} \hat{W}((a^i, l^i), z, \hat{\varphi}) \\
    {}& +\lambda(l^i)\left(\hat{W}((a^i, \check{l}^i), z, \hat{\varphi})-\hat{W}((a^i, l^i), z, \hat{\varphi})\right)  \\
    \cL_z \hat{W}((a^i,l^i), z, \hat{\varphi}) ={}& \partial_z \hat{W}((a^i,l^i),z,\hat{\varphi})\eta(\overline{z} - z) + \frac{1}{2} \sigma^2 \partial_{zz} \hat{W}((a^i,l^i),z,\hat{\varphi})  \\
    \hat{\cL}_g \hat{W}((a^i,l^i), z, \hat{\varphi}) ={}& \sum_{j\neq i} s((a^j,l^j),\hat{c}^\ast((a^j,l^j),z, \hat{\varphi}),r(z,\hat{\varphi}^{-j}),w(z,\hat{\varphi}^{-j})) \partial_{a^j} \hat{W}((a^i,l^i),z,\hat{\varphi}) \\
    {}& +\lambda(l^j)\left( \hat{W}((a^i,l^i), \{(a^j, \check{l}^j), z, \hat{\varphi}^{-j} \}) -\hat{W}((a^i,l^i),z,\hat{\varphi}) \right)
\end{align}

\medskip

\noindent\emph{Discrete state space approximation:}
In this case, we replace the distribution by its values at the points $\xi_1,\dots,\xi_N$ on the grid. More specifically, we take $N = 186$ points, and we consider a discretization $a_1< \dots < a_{93}$ of the $a$-axis. We then denote $\xi_1 = (a_1, l_1), \dots, \xi_{93} = (a_{93}, l_1)$, $\xi_{94} = (a_1, l_2), \dots, \xi_{186} = (a_{93}, l_2)\in \mathbb{R}^2$. We use a uniform grid and denote $\Delta a = a_{2} - a_{1}$. For ease of presentation, we write $\hat\varphi_{m,j}$ with a two-dimensional index $(m,j)$ in place of $\hat\varphi_{n}$ with a linear index $n=93\cdot (j-1) + m$.

Recall that the dynamics of the discrete distribution take the generic form~\eqref{eq:generic-KFE-finite-g}. In our implementation, we use the finite difference scheme proposed by~\cite{achdou2022income}. The KFE is replaced by the following finite difference equation:
\begin{equation}
    \label{eq:KFE-finite-g-achdouetal}
    d\hat{\varphi}_{m,j,t} = \mu_{\hat{\varphi},m,j}(z_t, \hat{\varphi}_t)dt, \qquad m=1,\dots,93, j=1,2,
\end{equation}
where the drift at point $(m,j)$ is defined by
\begin{align}
\label{eq:kfe-finite-diff-achdouetal}
    \mu_{\hat{\varphi},m,j}(z,\hat\varphi ) :={}& -(\hat{\partial}_a[  \boldsymbol{s}(z,\hat{\varphi}))\odot \hat\varphi])_{m,j} 
    + \lambda(l_{\check{j}})\hat\varphi_{m,\check{j}} - \lambda(l_j) \hat\varphi_{m,j}.
\end{align}
Here, $\check{j} = 2$ for $j=1$ and $\check{j}=1$ for $j=2$, the vector of saving flows $\boldsymbol{s}(z,\hat{\varphi})\in \mathbb{R}^{93\times 2}$ is
$$
    \boldsymbol{s}_{m,j}(z,\hat{\varphi}) := s((a_m,l_j),\hat{c}^\ast((a_m,l_j),z,\hat{\varphi}),r(z,\hat{\varphi}),w(z,\hat{\varphi})),
$$
and the first order derivative is approximated using an upwind scheme:
\[
    (\hat{\partial}_a[\boldsymbol{s}\odot \hat{\varphi}])_{m,i} := \frac{\boldsymbol{s}_{m-1,j}^{+}\hat{\varphi}_{m-1,j}-\boldsymbol{s}_{m,j}^{+}\hat{\varphi}_{m,j}}{\Delta a}+\frac{\boldsymbol{s}_{m+1,j}^{-}\hat{\varphi}_{m+1,j}-\boldsymbol{s}_{m,j}^{-}\hat{\varphi}_{m,j}}{\Delta a}
\]
where $(\boldsymbol{s}_{m,j})^+$ and $(\boldsymbol{s}_{m,j})^-$ denote the positive and negative part of $\boldsymbol{s}_{m,j}$, respectively, and we use the convention that the component of a vector is zero if its index is outside the boundaries of the original vector (which can happen for $m=1$ or $m=93$).

Relative to the generic finite difference scheme presented in Section~\ref{subsec:discrete-state-approx-master}, the variant here is simpler in two regards. First, there is no need to define a second-order finite difference operator $\hat{\partial}_{aa}$ because there are no idiosyncratic Brownian shocks. Second, we have added here the jump term directly into equation~\eqref{eq:kfe-finite-diff-achdouetal} without defining an interpolation operator $\Delta_{\breve{\boldsymbol{\varsigma}}(z,\hat{\varphi})}$ for evaluating the density at shifted inputs. This is possible here because jumps only switch the the $l$-value, so that the post-jump state will be on the grid whenever the pre-jump state was (and vice versa). There is therefore no need to interpolate off the grid.

Similarly to the finite agent approximation, the finite dimensional master equation is given by:
\begin{align}
    0 = \cL \hat{W}((a,l), z, \hat{\varphi}) ={}& (r(z,\hat{\varphi})-\rho) \hat{W}((a,l), z, \hat{\varphi}) + \textbf{1}_{a \leq a_{lb}} \partial_a \psi(a) \\
    {}& + (\cL_x + \cL_z + \hat{\cL}_g) \hat{W}((a,l),z,\hat{\varphi}) \label{eq:Master-KS}
\end{align}
where the operators become:
\begin{align}
    \cL_x \hat{W}((a,l), z, \hat{\varphi}) ={}& 
    s((a,l),\hat{c}^\ast((a,l),z, \hat{\varphi}),r(z,\hat{\varphi}),w(z,\hat{\varphi})) \partial_{a} \hat{W}((a, l), z, \hat{\varphi}) \\
    {}& +\lambda(l)\left(\hat{W}((a, \check{l}), z, \hat{\varphi})-\hat{W}((a, l), z, \hat{\varphi})\right)  \\
    \cL_z \hat{W}((a,l), z, \hat{\varphi}) ={}& \partial_z \hat{W}((a,l),z,\hat{\varphi})\eta(\overline{z} - z) + \frac{1}{2} \sigma^2 \partial_{zz} \hat{W}((a,l),z,\hat{\varphi})  \\
    \hat{\cL}_g \hat{W}((a,l),z,\hat{\varphi})
        ={}& \sum_{j=1,2}\sum_{m=1}^{93} \mu_{\hat{\varphi},m,j}(z,\hat{\varphi}) \, \partial_{\hat\varphi_{m,j}} \hat{W}((a,l),z,\hat{\varphi})
\end{align}
where $\partial_{\hat\varphi_{m,j}}\hat{W}$ denotes the partial derivative of $\hat{W}$ with respect to the coordinate $(m,j)$ of $\hat{\varphi}$, using two-dimensional indexing of $\hat{\varphi}$ as above, and where $\mu_{\hat{\varphi},m,j}(z,\hat{\varphi})$ is as defined in equation~\eqref{eq:kfe-finite-diff-achdouetal}. \\

\noindent\emph{Projection approximation:}
In this case, we replace the distribution by projection coefficients $\hat{\varphi}\in \mathbb{R}^5$ onto a basis $b_0,b_1,...,b_5$ of $6$ basis functions. We choose the basis functions as follows:
\begin{enumerate}[label=(\roman*)]
    \item We start out by solving the steady-state model with finite difference methods on a grid of $101$ equally spaced grid points in the $a$-dimension. This finite-difference solution yields a $202\times 202$-matrix that serves as a finite-dimensional approximation to the steady-state KFE operator $\mathcal{L}^{KF,ss}$. In a first step, we construct the basis of eigenfunctions described in the main text and discussed in more detail in Appendix~\ref{sec:eigenfunction-basis} (using $\overline{\cL}^{KF} = \mathcal{L}^{KF,ss}$). Practically, we approximate the eigenfunctions by eigenvectors of the finite difference KFE matrix. We pick a total of $7$ eigenvectors, which results in a preliminary basis $\tilde{b}_0^0,\tilde{b}_1^0,\dots,\tilde{b}_6^0$.

    \item In a second step, we impose the restriction that the marginal distributions in the $l$-dimension are always in line with the ergodic distribution of the $l_t^i$ process. This restriction ensures that effective labor is constant over time. It also reduces the dimension of the basis by $1$. After imposing the restriction, we are thus left with a reduced basis $\tilde{b}_0,\tilde{b}_1,\dots,\tilde{b}_5$, where $\tilde{b}_0=\tilde{b}_0^0$ and $\tilde{b}_1,\dots,\tilde{b}_5$ are each linear combinations of the original eigenfunctions $\tilde{b}_1^0,\dots,\tilde{b}_6^0$.

    \item In a third step, we make a change of variables that rotates the basis $\tilde{b}_0,\tilde{b}_1,\dots,\tilde{b}_5$ but leaves the set of densities that can be approximated unaffected. Specifically, we first find the representing vector in $\operatorname{span}\{\tilde{b}_1,...,\tilde{b}_5\}$ for the linear functional
    $$K(g):=\int{ag(a,l_1)da} + \int{ag(a,l_2)da},$$
    call it $b_1$ (existence is ensured by the Riesz representation theorem). We then select four more vectors out of $\tilde{b}_1,...,\tilde{b}_5$ such that together with $b_1$ we obtain again a basis of the space and then project those four vectors onto the orthogonal complement of $b_1$. Call the resulting vectors $b_2,...,b_5$. Also define $b_0:=\tilde{b}_0$. For our projection of the distribution, we work with the resulting basis $b_0,...,b_5$. This rotation helps us in sampling because the coefficient $\hat{\varphi}_1$ on $b_1$ fully controls the aggregate capital stock implied by a given distribution approximation vector $\hat{\varphi}$,\footnote{By construction, the basis vectors $b_2,...,b_5$ are orthogonal to $b_1$. Because $b_1$ is the representing vector for the $K$-functional, $K(b_2)=\cdots=K(b_5)=0$, so these components do not contribute to the mean of the distribution.} so that it is relatively straightforward to implement a variant of moment sampling.
\end{enumerate}

For the law of motion of the distribution approximation $\hat{\varphi}$, we adapt the generic approach outlined in the main text as follows for the KS model. We choose a single ``informative'' statistic whose evolution we seek to match perfectly, the aggregate capital stock. This is motivated by the well-known fact that this is the most important aspect of the distribution in the KS model. Given our particular basis, matching this dimension of the distribution evolution effectively determines the evolution of $\hat{\varphi}_1$. To determine the evolution of the remaining components $\hat{\varphi}_2,...,\hat{\varphi}_5$, we use an ``uninformative'' uniform discretization of the individual state space by choosing a 101-point grid in the $a$-dimension and require that the KFE is well-approximated (in a least-squares sense) on the resulting 101$\times$2-point grid of $(a,l)$-pairs. In total, this means we choose 203 ``test functions''. The first is the ``informative'' test function
$$\phi_1(a,l)=a,$$
which selects the first moment in the $a$-dimension, i.e. $K(g)$. The remaining test functions are the Dirac $\delta$-distributions
$$\phi_m=\begin{cases}\delta_{(a_{m-2},l_1)},&m\leq 102\\\delta_{(a_{m-103},l_2)},&m\geq 103\end{cases},\qquad m=2,3,\dots,203,$$
where $a_0<a_1<\cdots < a_{100}$ denotes the chosen grid points. Relative to the generic procedure outlined in the main text, we also choose to use a weighted least-squares procedure when minimizing the residuals that puts a very large -- in fact, infinite -- weight on the $\phi_1$-residual so as to match the capital evolution perfectly. The weight on all other test function residuals is chosen to be the same.

While the previous description implies a precise definition of $\mu_{\hat{\varphi},1},...,\mu_{\hat{\varphi},5}$ (up to the approximation of integrals on the discrete grid), the specific nature of our approximation allows us to compute many of the test function integrals analytically. In practice, we therefore compute $\mu_{\hat{\varphi},1},...,\mu_{\hat{\varphi},5}$ in a two-step procedure as follows:

First, we determine the law of motion of the coefficient $\hat{\varphi}_1$ that governs the evolution of aggregate capital $K_t$, so that we match the law of motion of $K_t$ exactly. Specifically, we start by computing the forward evolution
$$\mu_K(z,\hat{\varphi}) = \left.\frac{dK_t}{dt}\right|_{z_t=z,g_t=\hat{G}(\hat{\varphi})} = Y(\hat{\varphi},z) - \sum_{i=1,2}{\int{\hat{c}^\ast(a,l_i)\hat{G}(\hat{\varphi})(a,l_i)da}}-\delta K(\hat{G}(\hat{\varphi})),$$
where the integrals $\int{\hat{c}^\ast(a,l_i)\hat{G}(\hat{\varphi})(a,l_i)da}$ are approximated using quadrature over the 101-point grid in the $a$-dimension on which the basis functions are known. We then determine the law of motion for $\hat{\varphi}_1$ as follows:
$$\mu_{\hat{\varphi},1}(z,\hat{\varphi}) = \frac{\mu_K(z,\hat{\varphi})}{K(b_1)}$$

For the remaining components $\hat{\varphi}_2,...,\hat{\varphi}_5$, the $\delta$-test function choice implies that vector on the left-hand side of the regression that determines $\mu_{\hat{\varphi},2},...,\mu_{\hat{\varphi},5}$ is given by the vector of KFE drifts at the 202 $(a,l)$-grid points. Because we know the basis functions, and hence the density $\hat{g}$, only on the grid, we inevitably have to approximate the derivatives in those KFE drifts somehow. As in the discrete state space method, we choose a finite difference method to do so. Specifically, we first compute the finite difference approximation of the KFE precisely as in the discrete state space method on the 101$\times$2-point grid on which the basis functions are known. Call the resulting 202-dimensional vector $\mu_g^{DS}(z,\hat{\boldsymbol{g}})$, where $\hat{\boldsymbol{g}}=\hat{\boldsymbol{G}}(\hat{\varphi})$ is the density $\hat{g}=\hat{G}(\hat{\varphi})$ on the 101$\times$2-point grid (which corresponds to the distribution approximation in the discrete state method).\footnote{
    Specifically, this means that if $\mu_{\hat{\varphi}}^{DS}(z,\hat{\varphi})$ denotes the drift expression for the distribution approximation in the discrete state space method as defined in equation~\eqref{eq:kfe-finite-diff-achdouetal} (but for the different grid used here), then we define $\mu_g^{DS}(z,\hat{\boldsymbol{g}}):=\mu_{\hat{\varphi}}^{DS}(z,\hat{\boldsymbol{g}})$. Note that, in the context of the discrete state method, the vector $\hat{\varphi}$ describes the values of the density on the 101$\times$2-point grid, so plugging in $\hat{\boldsymbol{g}}$ for it is a well-defined operation. However, in the present context, $\hat{\varphi}$ has a different meaning (coefficients in the projection), so that we use the notation $\mu_g^{DS}(z,\hat{\boldsymbol{g}})$ instead of $\mu_{\hat{\varphi}}^{DS}(z,\hat{\boldsymbol{g}})$ to avoid any confusion.
} We then determine the drifts $\mu_{\hat{\varphi},2}(z,\hat{\varphi}),...,\mu_{\hat{\varphi},5}(z,\hat{\varphi})$ as the coefficients of a linear regression (orthogonal projection) of $\mu_g^{DS}(z,\hat{\boldsymbol{G}}(\hat{\varphi}))$ on the basis vectors $b_2,...,b_5$.

With these choices, the master equation can be written as:
\begin{align}
    0 = \cL \hat{W}((a,l), z, \hat{\varphi}) ={}& (r(z,\hat{\varphi})-\rho) \hat{W}((a,l), z, \hat{\varphi}) + \textbf{1}_{a \leq a_{lb}} \partial_a \psi(a) \\
    {}& + (\cL_x + \cL_z + \hat{\cL}_g) \hat{W}((a,l),z,\hat{\varphi}) \label{eq:Master-KS}
\end{align}
where the operators become:
\begin{align}
    \cL_x \hat{W}((a,l), z, \hat{\varphi}) ={}& 
    s((a,l),\hat{c}^\ast((a,l),z, \hat{\varphi}),r(z,\hat{\varphi}),w(z,\hat{\varphi})) \partial_{a} \hat{W}((a, l), z, \hat{\varphi}) \\
    {}& +\lambda(l)\left(\hat{W}((a, \check{l}), z, \hat{\varphi})-\hat{W}((a, l), z, \hat{\varphi})\right)  \\
    \cL_z \hat{W}((a,l), z, \hat{\varphi}) ={}& \partial_z \hat{W}((a,l),z,\hat{\varphi})\eta(\overline{z} - z) + \frac{1}{2} \sigma^2 \partial_{zz} \hat{W}((a,l),z,\hat{\varphi})  \\
    \hat{\cL}_g \hat{W}((a,l),z,\hat{\varphi})
        ={}& \sum_{n=1}^{5} \mu_{\hat{\varphi},n}(z,\hat{\varphi}) \,\partial_{\hat{\varphi}_n} \hat{W}((a,l),z,\hat{\varphi})
\end{align}

\medskip

\subsection{Implementation Details} \label{asec:implementation_details}

In this section, we go through the implementation details for each of the different approaches.

\subsubsection{Network Structure\label{asec:implementation_details:network_structure}}

\noindent\emph{Finite Agent Agent Approximation:} We use a fully connected feed-forward neural network with $5$ layers and $64$ neurons per layer.
We use a $\tanh$ activation function between layers and a soft-plus activation at the output level. 
We initialize the neural network so that $\bW(a, \cdot)$ has an exponential shape with negative exponent. This is done through a pre-training phase. 
\\

\noindent\emph{Discrete State Space Approximation:} The neural network for approximating $\hat{W}=\partial_a \hat{V}$ combines three steps to map the input data $\hat{X}:=\{x, z, \hat{\varphi}\}$ into an output $\bW(\hat{X};\Theta_W)$. We describe these steps separately:

In a first step, an ``embedding'' network transforms the component $\hat{\varphi}$ into a 10-dimensional output $\hat{\varphi}^\prime$ by feeding it through a fully connected feed-forward network as described in Section~\ref{sec:solution:network}. This embedding network has 2 layers and 128 neurons per layer. We use a tanh activation function in the hidden layers and the identity function in the output layer. Denote by $\Theta_W^e$ the collection of parameters for this network.

In a second step, we apply a recurrent network as proposed by \cite{Sirignano2018} to the modified input data $\hat{X}^\prime := \{x, z, \hat{\varphi}^\prime\}$, which results from $\hat{X}$ by replacing the distribution approximation $\hat{\varphi}$ with the output $\hat{\varphi}^\prime$ of the embedding network. Specifically, the structure of the \cite{Sirignano2018} network is as follows:
\begin{align}
\begin{aligned} \label{nn:structure:dgm}
    h^{(0)} ={}& \phi^{(0)}( W^{(0)} \hat{X}^\prime + b^{(0)}) &&& \ldots \text{Hidden layer 0}\\
    f^{(1)} ={}& \phi^{(1)}\left(W^{f,(1)}h^{(0)}+U^{f,(1)}\hat{X}^\prime+b^{f,(1)}\right) &&& \ldots \text{Hidden layer 1}\\
    g^{(1)} ={}& \phi^{(1)}\left(W^{g,(1)}h^{(0)}+U^{g,(1)}\hat{X}^\prime+b^{g,(1)}\right) &&& \ldots \text{Hidden layer 1}\\
    r^{(1)} ={}& \phi^{(1)}\left(W^{r,(1)}h^{(0)}+U^{r,(1)}\hat{X}^\prime+b^{r,(1)}\right) &&& \ldots \text{Hidden layer 1}\\
    s^{(1)} ={}& \phi^{(1)}\left(W^{s,(1)}(r^{(1)}\odot s^{(0)})+U^{s,(1)}\hat{X}^\prime+b^{s,(1)}\right) &&& \ldots \text{Hidden layer 1}\\
    h^{(1)} ={}& (1-g^{(1)})\odot s^{(1)}+f^{(1)}\cdot h^{(0)} &&& \ldots \text{Hidden layer 1}\\
    \vdots \\
    f^{(H)} ={}& \phi^{(H)}\left(W^{f,(H)}h^{(H-1)}+U^{f,(H)}\hat{X}^\prime+b^{f,(H)}\right) &&& \ldots \text{Hidden layer $H$}\\
    g^{(H)} ={}& \phi^{(H)}\left(W^{g,(H)}h^{(H-1)}+U^{g,(H)}\hat{X}^\prime+b^{g,(H)}\right) &&& \ldots \text{Hidden layer $H$}\\
    r^{(H)} ={}& \phi^{(H)}\left(W^{r,(H)}h^{(H-1)}+U^{r,(H)}\hat{X}^\prime+b^{r,(H)}\right) &&& \ldots \text{Hidden layer $H$}\\
    s^{(H)} ={}& \phi^{(H)}\left(W^{s,(H)}(r^{(H)}\odot s^{(H-1)})+U^{s,(H)}\hat{X}^\prime+b^{s,(H)}\right) &&& \ldots \text{Hidden layer $H$}\\
    h^{(H)} ={}& (1-g^{(H)})\odot s^{(H)}+f^{(H)}\cdot h^{(H-1)} &&& \ldots \text{Hidden layer $H$}\\
    o ={}& W^{(H+1)} h^{(H)} + b^{(H+1)} &&& \ldots \text{Output layer} \\
    \hat{Y} ={}& \phi^{(H+1)}(o) &&& \ldots \text{Output}
\end{aligned}
\end{align}
Here, $\{h^{(i)}\}_{0\leq i\leq H}$ denote the $H+1$ hidden layers and $\{f^{(i)},g^{(i)},r^{(i)},s^{(i)}\}_{1\leq i\leq H}$ are auxiliary variables required to compute the neuron value in each layer. $W^{(0)}$, $\{W^{f,(i)},W^{g,(i)},W^{r,(i)},W^{s,(i)}\}_{1\leq i\leq H}$, and $\{U^{f,(i)},U^{g,(i)},U^{r,(i)},U^{s,(i)}\}_{1\leq i\leq H}$ are the weight matrices of the network, whereas $b^{(0)}$ and $\{b^{f,(i)},b^{g,(i)},b^{r,(i)},b^{s,(i)}\}_{1\leq i \leq H}$ are the biases. The operator $\odot$ denotes the element-wise product (Hadamard product) of vectors. We choose $H=3$ and $100$ neurons per layer, which means that the dimension of each of the vectors $h^{(i)}$, $f^{(i)}$, $g^{(i)}$, $r^{(i)}$, and $s^{(i)}$ is 100. We use a tanh activation function in the hidden layers and an elu activation function in the output layer to ensure positivity of the output. The trainable parameters $\Theta_W^r$ of this network consists of the collection of all weights and biases of the network. 

In a final third step, we transform the output $\hat{Y}$ from the recurrent network into the approximate value for $W$ as follows:
$$\bW = \hat{Y} (a_{0}+a)^{-\tilde{\eta}}.$$
Here $a_0$ and $\tilde{\eta}$ are non-trainable parameters. The additional factor $(a_0+a)^{-\tilde{\eta}}$ is motivated by the hyperbolic shape of the marginal value function. Its inclusion helps improving the overall accuracy of the approximation for a given neural network size. We make the selection $a_0=10$ and $\tilde{\eta}=0.5$.

The trainable parameters of the network $bW$ consists of the collection $\Theta_W=\{\Theta_W^e,\Theta_W^r\}$ of all trainable parameters from the first and second step. We initialize all bias parameters to zero and all weight parameters randomly according to a uniform Xavier initialization \citep{GlorotBengio2010}.

In addition to the neural network for $\hat{W}$, we also introduce an auxiliary network for the consumption function $\hat{c}^\ast$, $\bC(\hat{X};\Theta_c)$. The structure of $\bC$ precisely mirrors that of the network for $\bW$ (including the number of layers and neurons in each sub-network), except that we omit the third step. For initialization of the parameters $\Theta_c$, we follow the same approach as for $\Theta_W$.\\

\noindent\emph{Projection Approximation:}
We use the same network structure and parameter initialization as for the discrete state space approximation. The only difference is that we choose a smaller embedding network, both for approximating $\hat{W}$ and for approximating $\hat{c}^\ast$: instead of 128 neurons per layer, the embedding network for this approximation has only 64 neurons per layer.

\subsubsection{Sampling\label{asec:implementation_details:sampling}}

\noindent\emph{Finite Agent Approximation:} 
We sample points of the form $\{(a^i,n^i), (a^j,n^j)_{j \in (I-1)}\}$ on the interior of the state space.
For the idiosyncratic variable, $a^i$, we sample using an active sampling technique similar to those developed by \cite{Gopalakrishna2021} and \cite{lu2021deepxde}.
We start active sampling after 2,000 epochs to balance speed with accuracy.\footnote{
    The loss has a steeper drop after we start active learning in the 2000th epoch, as shown in Figure \ref{fig:loss_path_FA_Robustness}.}
Once we start active sampling, the interval $[a_{min}, a_{max}]$ is evenly partitioned into $2^{4}$ subintervals.  
We calculate the residual error in each subinterval then add $2^{4}$ points to the subinterval with the largest residual, $2^3$ points to the neighboring subinterval with the largest error, and $2^2$ to the other neighboring subinterval.
We sample the aggregate variable $z$ uniformly from interval $[z_{min},z_{max}]$.
For sampling the population of agents, $(a^j,n^j)_{j \in (I-1)}$, we first generate a random interest rate from a uniform distribution on $[r_{lb},r_{rb}],$ with $r_{lb}=0.01,r_{rb}=0.05$.
We then generate a random distribution of agents from a Latin hypercube on $[a_{min}, a_{max}]$ and scale their individual wealth so the equilibrium interest rate is the randomly drawn interest rate for the drawn aggregate TPF $z$.
\\

\noindent\emph{Discrete State Space Approximation:}
We first sample points of the form $(a,\hat{\varphi},z)$ randomly and then construct the full sample for points of the form $((a, l), \hat{\varphi}, z)$ by using each originally sampled point twice, once in combination with $l=l_1$ and once in combination with $l=l_2$. We sample the three dimensions of $(a,\hat{\varphi},z)$ independently as follows:
\begin{itemize}
  \item We sample $a$ from a uniform distribution over the domain $[a_{min},a_{max}]$ (without active sampling).
  \item We sample $z$ from a uniform distribution over the domain $[z_{min},z_{max}]$ as for the finite agent method.
  \item For the sampling of $\hat{\varphi}$, we use a mixture of two sampling schemes, (i) a variant of mixed steady state sampling and (ii) ergodic sampling based on the current approximation for $\hat{W}$. For sampling scheme (i), we use a ``degenerate'' mixture based on just a single steady-state solution (for $z=0$). Denote by $\hat{g}^{ss}$ the vector of steady-state density values on the discrete state grid. We take as sample points $\hat{\varphi}(a_i,l_i) = \frac{\omega_{i} \hat{g}^{ss}(a_i,l_i)}{\sum_{j=1}^{N_x} \omega_{j} \hat{g}^{ss}(a_j,l_j)}$, where the $\omega_{i}$ are i.i.d. with uniform distribution over an interval of the form $[1-d_g, 1+d_g]$, with $d_g \in (0,1)$. We gradually increase the proportion of the sample that is according to (ii) from 0\% to 90\% during training.\\
\end{itemize}

\noindent\emph{Projection Approximation}
We follow precisely the same sampling approach as in the discrete state space approximation, except for the distribution dimension $\hat{\varphi}$. We therefore only discuss the latter here. We sample the first component of the distribution, $\hat{\varphi}_1$, separately as this component exclusively controls the aggregate level of the capital stock (compare Appendix~\ref{asec:master_equations}). We sample aggregate capital $K$ from a uniform distribution over $[0.9K^{ss},1.1K^{ss}]$, where $K^{ss}$ is the steady-state level of capital in the model without common noise and $z=0$, and we adjust $\hat{\varphi}_1$ to match the sampled capital stock values. We sample the remaining four components $\hat{\varphi}_2,...,\hat{\varphi}_5$ by combining uniform sampling from a hypercube centered around zero and ergodic sampling. We gradually increase the proportion of ergodic sampling from 0\% to 80\% during training.

\subsubsection{Loss Function\label{asec:implementation_details:loss_function}}

For all three approximation method we impose penalties to impose the shape constraints $\partial_a W < 0$ and $\partial_z W < 0$ by choosing the shape error as follows:
$$\cE^s(\Theta^n, S^{n}) := \frac{1}{|S^n|} \sum_{(a,l,z,\hat{\varphi})\in S^n} \left(|\max\{\partial_a\bW(a,l,z,\hat{\varphi};\Theta^n),0\} |^2 + |\max\{\partial_z\bW(a,l,z,\hat{\varphi};\Theta^n),0\} |^2\right)$$
We combine this with the equation residual error $\cE^e(\Theta^n, S^{n})$ for the total error $\cE(\Theta^n, S^{n})$ as described in Algorithm~\ref{alg:generic}. We choose as weights $\kappa^e=100,\kappa^s=1$ for the finite agent method and $\kappa^e=\kappa^s=1$ for the other two methods. Ultimately, we find that the relative weights of the loss components only matter in early training when the shape constraints are occasionally violated. In late training, shape constraints are typically always satisfied, so that the weights $\kappa^e,\kappa^s$ have little relevance for training.

\subsection{Calibration} \label{subsec:calibration}

As was discussed in the main text, a potential benefit of deep learning algorithms is that we can include the parameters, $\zeta$, as additional inputs into the neural network:
\begin{align}
    \hat{V}(\hat{X}, \zeta) \approx \bV(\hat{X}, \zeta; \theta)
\end{align}
and then train the neural network using sampling from both $\hat{X}$ and $\zeta$.
We can then use $\bV$ to calculate the moments for different parameters and calibrate the model.
We illustrate this technique for the finite agent method by calibrating the \cite{Krusell1998} model to match a stochastic steady state capital-to-labor ratio of 5.0.
We do this by including the borrowing constraint $a_{lb}$ as an input into the neural network, training the model randomly sampling $a_{lb}$ from $[0, 1.5]$, and then using the trained model to solve for the $a_{lb}$ that generates a stochastic steady state capital-to-labor ratio of 5.0.
We show the results in Table~\ref{tab:ks_calibration}.

\begin{table}[H]
    \centering
    \begin{tabular}{ccc}
     & $K/L$& \\     Target&Model&$a_{lb}$\\\hline
     5.0&5.0& $1.082$ \\
    \end{tabular}
    \caption{Neural Networks' results for Calibration}
    \label{tab:ks_calibration}
\end{table}

\subsection{Transition Dynamics for MIT Shocks} \label{subsubsec:transition_dynamics_ABH}

We now consider the transition path following an unexpected shock to aggregate productivity (a so-called ``MIT'' shock).
In principle, an important advantage of having a global solution, i.e., $c^*$ as a function of the distribution, is that we can solve for the transition path without a ``shooting algorithm'', as is commonly done in the finite difference literature.
This is an advantage because shooting algorithms can be unstable, particularly for systems with a large number of prices that require complicated guesses for the price path.
However, in practice, attempting this only makes sense for the finite agent approximation since the other methods require ergodic sampling during training and so have difficulty with unanticipated shocks.
For this reason, we focus on the finite agent method in this section.

In order to compute the response to an MIT shock we need to compute the evolution of the distribution.
This is complicated for the finite agent method because the neural network policy rules are functions of the positions of the $N$ other agents rather than a continuous density.
To overcome this difficulty, we deploy the ``hybrid'' approach described in Algorithm \ref{alg:transition} that uses the neural network solution to approximate a finite difference approximation to the KFE.
Let $\underline{\boldsymbol{a}}= (a_m : m \le M)$ denote the grid in the $a$-dimension.
Let $\underline{\boldsymbol{g}}_{t,j}= (g_{m,j,t} : m \le M)$ denote the marginal density at labor $l_j$ on the $a$-grid and $\underline{\boldsymbol{g}}_{t}$ denote the density.
At each time step, our method draws $N_{sim}$ different samples of $N$ agents from from the current density $g_t$.
For each draw $k \le N_{sim}$, denoted by $\hat\varphi_t^k = ((a_i, l_i) : 1 \le i \le N)$, the KFE is replaced by the following finite difference equation:
\begin{equation}
    \label{eq:KFE-finite-g-achdouetal-transition}
    dg_{m,j,t} = \mu_{g,m,j}(\hat{\varphi}_t^k)dt, \qquad m \le M, j=1,2,
\end{equation}
where the drift at point $(m,j)$ is defined by
\begin{align}
    \mu_{g,m,j}(\hat\varphi^k ) :={}& -(\hat{\partial}_a[  \boldsymbol{s}(\hat{\varphi}^k)) \odot \underline{\boldsymbol{g}} ])_{m,j} 
    + \lambda(l_{\check{j}})g_{m,\check{j}} - \lambda(l_j) g_{m,j}.
\end{align}
Here, $\check{j} = 2$ for $j=1$ and $\check{j}=1$ for $j=2$, the vector of saving flows $\boldsymbol{s}(\hat{\varphi}^k)\in \mathbb{R}^{M\times 2}$ is
$$
    \boldsymbol{s}_{m,j}(\hat{\varphi}^k) := s((a_m,l_j),\hat{c}^\ast((a_m,l_j),\hat\varphi^k),r(\hat{\varphi}^k),w(\hat{\varphi}^k)),
$$
and the first order derivative is approximated using an upwind scheme:
\[
    (\hat{\partial}_a[\boldsymbol{s}\odot \underline{\boldsymbol{g}}])_{m,i} := \frac{\boldsymbol{s}_{m-1,j}^{+}g_{m-1,j}-\boldsymbol{s}_{m,j}^{+}g_{m,j}}{\Delta a}+\frac{\boldsymbol{s}_{m+1,j}^{-}g_{m+1,j}-\boldsymbol{s}_{m,j}^{-}g_{m,j}}{\Delta a}
\]
where $(\boldsymbol{s}_{m,j})^+$ and $(\boldsymbol{s}_{m,j})^-$ denote the positive and negative part of $\boldsymbol{s}_{m,j}$, respectively, and we use the convention that the component of a vector is zero if its index is outside the boundaries of the original vector (which can happen for $m=1$ or $m=M$).
From this approximation we can calculate the transition matrix $\cA_{t,k}$ for the finite difference approximation at the draw $\varphi^k$.

\begin{algorithm}[h]
\caption{Finding Transition Path by Neural Network: Aiyagari case} \label{alg:transition}
    \SetKwInOut{Input}{Input}
    \SetKwInOut{Output}{Output}
    \Input{Initial distribution $g_0$, neural network approximations for consumption and prices $(\hat{c}, \hat{r}, \hat{w})$, number of agents $N$, time step size $\Delta t$, number of time steps $N_{T}$, number of simulations $N_{sim}$, grid $\underline{\boldsymbol{a}}= \{a_m : m \le M\}$ for the finite difference approximation.} 
    \Output{A transition path $g = \{g_t : t = 0, \Delta t, \dots, N_T \Delta t\}$}

    \setstretch{1.05}
    \vspace{0.25cm}
    
    \For{$n = 0,\dots, N_T-1$}{
        \For{$k = 1, \dots, N_{sim}$}{
            Draw states for $N$ agents $\{\varphi_i^k : i=1,\dots,N\}$ from density $g_t$ at $t = n \Delta t$.
            
            Given state $\varphi^k$, compute equilibrium return $\hat{r}(\varphi^k)$ and wage $\hat{w}(\varphi^k)$.

            At each grid point $a_m \in \underline{\boldsymbol{a}}$, calculate the consumption $\hat{c}((a_m,l),\varphi^k)$, $\forall l$.
            
            Construct the transition matrix $\mathcal{A}_{t,k}$ using finite difference on the grid $\underline{\boldsymbol{a}}$, as described by \eqref{eq:KFE-finite-g-achdouetal-transition} and the subsequent equations.
        }
    Take the average: $\Bar{\mathcal{A}}_t=\frac{1}{N_{sim}}\sum_{k=1}^{N_{sim}}\mathcal{A}_{t,k}$.
        
    Update $g_t$ by implicit method: $g_{t+ \Delta t} = (I-\Bar{\mathcal{A}}_t^\top \Delta t)^{-1}g_{t}$.
    }
\end{algorithm}

This approach is very different to the finite difference approach in \cite{achdou2022income}, which is summarized in Algorithm \ref{alg:finite_diff}.
This is because our neural network transition paths do not require an outer ``guess-verify'' loop, as in the \cite{achdou2022income} shooting algorithm.

\vspace{6pt}

\begin{algorithm}[H]
\caption{Finding Transition Path by Finite Difference} \label{alg:finite_diff}
\begin{enumerate}
    \item Guess the path of equilibrium interest rate $r^o_t$, then solve HJB, with terminal condition: $v(a,l,T) = \Bar{v}(a,l)$
    \vspace{-0.2cm}
    \item Solve the policy function $c_t(a,l)$.
    \vspace{-0.2cm}
    \item Solve the forward equation, with initial condition $g_0(a,l)$.
    \vspace{-0.2cm}
    \item Calculate the capital held by the whole economy: $\sum_{j \in \{0,1\}}\int_{\underline{a}}^{\infty}ag_t(a,l_J) = K_t$, then calculate the implied interest rate by $r_t^n=\partial_KF(K_t,L)-\delta$, for every $t\in[0,T]$.
    \vspace{-0.2cm}
    \item Update the path to be $\lambda r^o_t + (1-\lambda) r^n_t$, repeat 1-4 until $||\mathsf{r}^o - \mathsf{r}^n||_{\infty}<\epsilon$.
\end{enumerate}
\end{algorithm}

\vspace{6pt}

We compare the neural network and finite difference transition paths in Figure \ref{fig:aiy:fa:transition} below.
In this numerical experiment, we train our neural network at $z = 0$ and we start from an economy in its steady state with productivity $z_t=-0.10$ for $t=0$. 
At $t = 0^{+}$, an unexpected positive productivity shock brings $z$ from $-0.10$ to $z_t = 0$ permanently. 
We solve for distributional dynamics using Algorithm \ref{alg:transition} and Algorithm \ref{alg:finite_diff} using the steady state at $z=-0.10$ as the initial condition. We plot the percentage change of capital, capital return and wage evolution respectively in the first row and the second row of Figure \ref{fig:aiy:fa:transition}. The difference between the neural net transition paths and finite difference transition paths are less than 0.1\%. 
The lower panels of Figure \ref{fig:aiy:fa:transition} compare the neural network and finite difference probability densities at time $t=15$ and $t= 30$, which are also very similar.

\begin{figure}[hbtp]
    \centering
    \includegraphics[width=\textwidth]{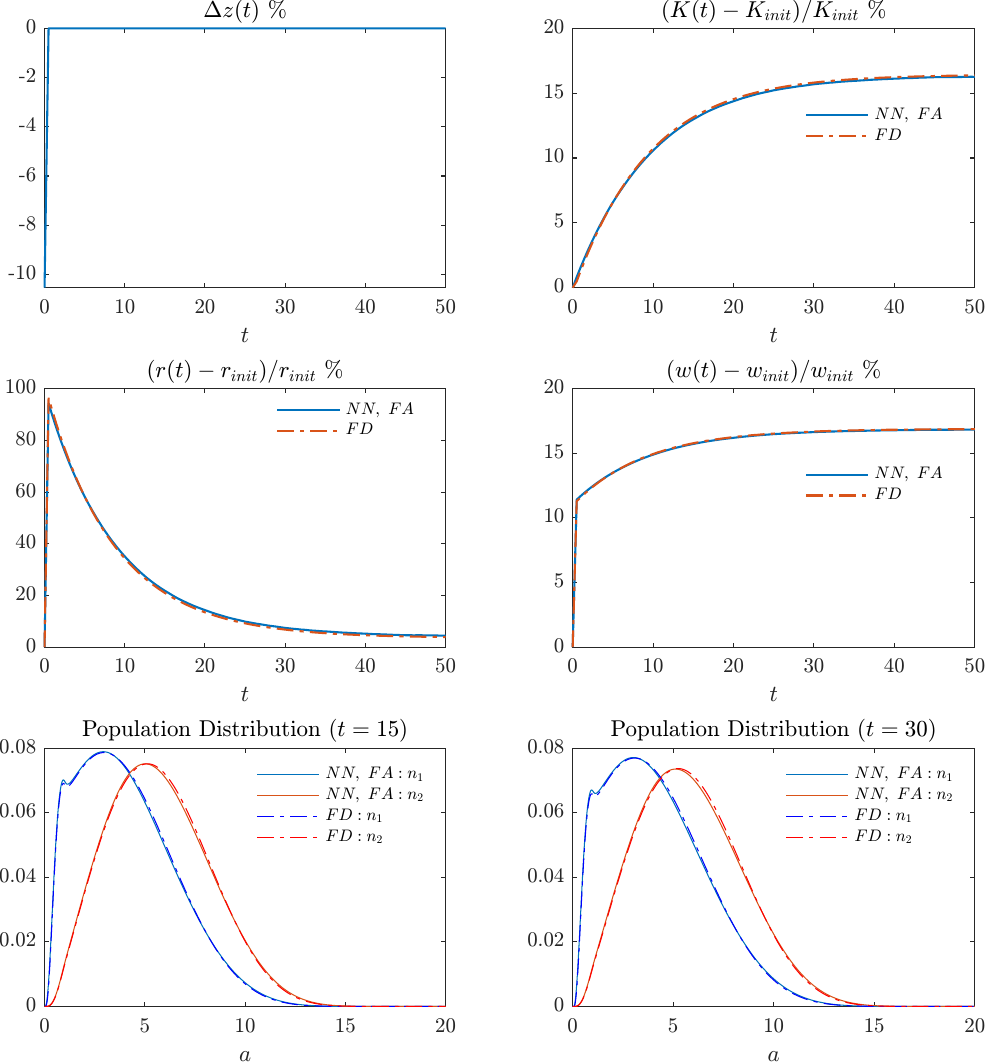}
    \caption{\small Comparison between neural network and finite impulse response for the Aiyagari model. The top left plot is the TFP shock path, the top right panel is the aggregate relative capital change, the middle left panel plots the relative capital return change, and the middle right panel plots the relative wage change. The bottom left and bottom right are snapshots of probability density at $t=15$ and $t=30$. \textit{``NN, FA''} refers to the finite agent neural network code, and the \textit{``FD''} refers to the finite difference code. Subscript \textit{init} is the initial value at the steady state $z = z(0)$.}
    \label{fig:aiy:fa:transition}
\end{figure}

\subsection{Additional Details on Simulating the KS Model} \label{subsec:KS_simulation}

We simulate from the KS model using an extended version of Algorithm \ref{alg:transition} that incorporates aggregate shocks.
We describe this in Algorithm \ref{alg:transition_KS}.

\begin{algorithm}[h]
\caption{Finding Transition Path by Neural Network: Krusell Smith case} \label{alg:transition_KS}
    \SetKwInOut{Input}{Input}
    \SetKwInOut{Output}{Output}
    \Input{Initial distribution, neural network approximations to the consumption rule $\hat{c}$, and pricing functions $(\hat{r}, \hat{w})$,  number of agents $N$, time step size $\Delta t$, number of time steps $N_{T}$, number of simulations $N_{sim}$, grid $\underline{\boldsymbol{a}}= \{a_m : m \le M\}$ for the finite difference approximation.}
    \Output{A transition path $g = \{g_t : t = 0, \Delta t, \dots, N_T \Delta t\}$}

    \setstretch{1.05}
    \vspace{0.25cm}
    
    \For{$n = 0,\dots, N_T-1$}{
        \For{$k = 1, \dots, N_{sim}$}{
        Sample $\Delta B^0_t$ from normal distribution ${N}(0, \Delta t)$, construct TFP shock path by: $z_{t+\Delta t} = z_{t} + \eta(\bar{z}-z_{t}) + \sigma  \Delta  B^0_t$.
    
        Draw states for $N$ agents $\{\varphi_i^k : i=1,\dots,N\}$ from density $g_t$ at $t = n \Delta t$.

        Given state $(z_{t+\Delta t}, \varphi_t^k)$, compute equilibrium return $\hat{r}(z_{t+\Delta t}, \varphi_t^k)$ and wage $\hat{w}(z_{t+\Delta t}, \varphi_t^k)$.

        At each grid point $a_m \in \underline{\boldsymbol{a}}$, calculate the consumption $\hat{c}((a_m,l),z_{t+\Delta t}, \varphi_t^k)$, $\forall l$.
    
        Construct the transition matrix $\mathcal{A}_{t,k}$ using finite difference on the grid $\underline{\boldsymbol{a}}$, as described \eqref{eq:KFE-finite-g-achdouetal-transition} and subsequent equations but with the $z_t$ state added.
        
    }
    Take the average: $\Bar{\mathcal{A}}_t=\frac{1}{N_{sim}}\sum_{k=1}^{N_{sim}}\mathcal{A}_{t,k}$
    
    Update $g_t$ by implicit method: $g_{t+ \Delta t} = (I-\Bar{\mathcal{A}}_t^\top  \Delta t)^{-1}g_{t}$
}
\end{algorithm}

\section{Additional Details on the Eigenfunction Basis for the Projection Technique (Online Appendix)}
\label{sec:eigenfunction-basis}

We would like to choose a basis such that just a few basis functions are enough to provide a good approximation. The key idea to achieve this is to track the slow-moving or persistent dimensions of $g_t$ while neglecting those dimensions that mean-revert fast. To understand why, recall that the only reason the distribution appears in the state space is because it helps the agent forecast future prices $q_t$. Components that mean-revert fast carry only little information about future prices beyond a short time horizon. Neglecting them induces a comparably small error into the agent's forecasts.

The persistent dimensions of the distribution are related to certain eigenfunctions of the differential operator characterizing the KFE~\eqref{eq:generic-KFE}. We can rewrite this equation as
\begin{align}%
    dg_t(x) ={}& (\cL^{KF}_t g_t)(x) dt
\end{align}
where, for generic distribution $f$ and point $x$,
\begin{align*}
\mathcal{L}^{KF}_tf(x) ={}& -\sum_{j=1}^{N_x} \partial_{x_j}\left[\mu_{x,j}(c^*(x,z_t,g_t), x, z_t, Q(z_t, g_t)) f(x)\right]\\
    {}& +  \frac{1}{2} \sum_{j,k=1}^{N_x} \partial_{x_j,x_k} \left[\Sigma_{x,jk}(c^*(x,z_t,g_t),x,z_t,Q(z_t,g_t)) f(x)\right] \\
    {}& + \lambda (f(x-\breve{\varsigma}(x,z_t,g_t))|I-D_x \breve{\varsigma}(x,z_t,g_t)| - f(x)).
\end{align*}
Notice that $\mathcal{L}^{KF}_t$ is a linear differential operator that generally depends on time and is stochastic due to the implicit dependence on $z_t$ and $g_t$. Constructing a (fixed) basis based on the eigenfunctions of $\mathcal{L}^{KF}_t$ is therefore not directly possible because the set of eigenfunctions would itself be time-dependent and stochastic.

Instead, our construction is based on a related time-invariant operator $\overline{\mathcal{L}}^{KF}$. $\overline{\mathcal{L}}^{KF}$ could take many forms. The simplest approach and the one we follow in our algorithm is to use $\overline{\mathcal{L}}^{KF} = \mathcal{L}^{KF,ss}$, where we define $\mathcal{L}^{KF,ss}$ in analogy to $\mathcal{L}^{KF}_t$ but for a simplified model with common noise set to zero, $\sigma^z\equiv 0$, and under the assumption that the aggregate states $z_t=\bar{z}$ and $g_t = g^{ss}$ have reached a steady-state. Another natural option would be to let $\overline{\mathcal{L}}^{KF}$ be the KFE operator that results from a linear perturbation of the model, but this would require solving an auxiliary problem first, in addition to the steady state problem, in order to obtain $\overline{\mathcal{L}}^{KF}$. In any case, we assume that $\overline{\mathcal{L}}^{KF}g^{ss} = 0$.\footnote{
    This is satisfied for both options for $\overline{\mathcal{L}}^{KF}$ discussed here. More generally, we can simply define $g^{ss}$ to be the steady state with respect to $\overline{\mathcal{L}}^{KF}$.
}

Irrespective of the precise choice of $\overline{\mathcal{L}}^{KF}$, there is a heuristic element in our basis construction that lies in the presumption that broadly the same dimensions of the distribution are relevant for the dynamics described by $\overline{\mathcal{L}}^{KF}$ and the true dynamics described by $\mathcal{L}^{KF}_t$. While generally plausible for the choices of $\overline{\mathcal{L}}^{KF}$ discussed above, if this requirement is not satisfied, then our proposed basis may not track the persistent dimensions of the true KFE dynamics well making the basis choice ``less efficient'', so that we may need a larger number of basis functions $N$ to achieve a good overall approximation quality.

Let us consider the eigenfunctions $\{b_i\,:\, i \ge 0\}$ of $\overline{\mathcal{L}}^{KF}$ with corresponding eigenvalues denoted by $\{\lambda_i \in \mathbb{C} \,:\, i \ge 0\}$. They satisfy:
\begin{equation}
	\overline{\mathcal{L}}^{KF} b_i = \lambda_i b_i, \qquad i \ge 0.
\end{equation}
Suppose $g^{ss}$ is the unique stationary distribution and the dynamics described by the KFE are locally stable around $g^{ss}$. Then the eigenvalue $\lambda_0=0$ exists and its associated eigenfunction is, up to scaling, $b_0 = g^{ss}$. All remaining eigenvalues satisfy $\Re \lambda_i<0$, so that these components of the distribution mean-revert to zero over time. Furthermore, a smaller $\Re\lambda_i$ is associated with a faster speed of mean-reversion. To be precise, suppose $g_t = g^{ss} + \sum_{i=1}^{\infty}{\varphi_{i,t}b_i}$ is the time-$t$ distribution expressed as a linear combination of all the eigenfunctions and suppose the time evolution of $g_t$ is described by the differential operator $\overline{\mathcal{L}}^{KF}$. Then,
$$g^{ss} + \sum_{i=1}^{\infty}{\frac{d\varphi_{i,t}}{dt}b_i} = 
\frac{dg_t}{dt} = \overline{\mathcal{L}}^{KF}g_t = \overline{\mathcal{L}}^{KF}g^{ss}  + \sum_{i=1}^{\infty}{\varphi_{i,t}\overline{\mathcal{L}}^{KF}b_i} = g^{ss}  + \sum_{i=1}^{\infty}{\varphi_{i,t}\lambda_ib_i}$$
Since the $b_i$ form a basis, the previous equation implies a system of ordinary differential equations for the coefficient functions $\varphi_{i,t}$, $i \ge 1$:
$$\frac{d\varphi_{i,t}}{dt} = \lambda_i\varphi_{i,t}.$$
The solution is given by $\varphi_{i,t} = \varphi_{i,0} e^{\lambda_i t}$, and therefore
\begin{align}
g_t = g^{ss} + \sum_{i=1}^{\infty}{\varphi_{i,0} e^{\lambda_i t}b_i}, \qquad t \ge 0.\label{eq:distribution_dynamics_eigenfunctions_and_linear_kfe}
\end{align}
Because $\Re \lambda_i<0$ for $i\geq 1$, all terms in the series on the right decay to zero as $t\rightarrow \infty$. Components corresponding to eigenvalues $\lambda_i$ with very negative real parts decay at a faster rate. As such, deviations from the stationary distribution $g^{ss}$ have a greater persistence if they are in the direction of eigenfunctions corresponding to eigenvalues that are (negative but) close to $0$. We therefore expect that a finite basis of eigenfunctions will provide a good approximation of the infinite sum if there are only a few ``significant'' eigenvalues that are close to $0$.

The previous considerations motivate our basis choice in the main text. To be precise, that basis choice means the following: we order with descending real parts, $\lambda_0=0>\Re\lambda_1>\Re\lambda_2>\dots$, and choose $b_0,b_1,...,b_N$ as the eigenfunctions corresponding to the first $N+1$ eigenvalues in this ordering as our basis. We remark here also that this choice of $b_0,b_1,...,b_N$ satisfies the required properties stated in equation~\eqref{eq:basisProperties} of Section~\ref{sec:finite_master:projection}. This is clear for $n=0$. For $n>0$, it follows form the fact that the operator $\overline{\mathcal{L}}^{KF}$ describes a mass-preserving evolution, which is only consistent with asymptotic decay over time if the integral is zero.

\begin{proof}[Proof of Proposition~\ref{prop:linear_kfe_eigenfunction_basis}]
The ``only if'' direction of the proposition essentially follows from equation~\eqref{eq:distribution_dynamics_eigenfunctions_and_linear_kfe} derived in this appendix if we set $\overline{\mathcal{L}}^{KF} := \mathcal{L}^{KF}$. Note that the additional assumptions made in the text, specifically that (i) $g^{ss}$ is the unique steady state of the operator and (ii) eigenvalues have negative real part, have been mentioned in the text only to make statements about the decay of certain components in the formula but are not necessary for the derivation of the formula itself. Therefore, the formula holds also under the assumptions of Proposition~\ref{prop:linear_kfe_eigenfunction_basis} if $g_t$ satisfies the KFE for all $t\geq0$. We now use this fact to prove the first direction in the equivalence statement. If $g_t$ satisfies the KFE, then, as just observed, it must also satisfy equation~\eqref{eq:distribution_dynamics_eigenfunctions_and_linear_kfe}. Furthermore, because $g_t-g^{ss}\in \operatorname{span}\{b_1,...,b_N\}$, it must be that $\varphi_{i,0}=0$ for all $i>N$. Hence, the infinite sum in  equation~\eqref{eq:distribution_dynamics_eigenfunctions_and_linear_kfe} collapses to the finite sum in equation~\eqref{eq:linear_kfe_eigenfunction_representation} stated in the proposition.

Conversely, if $g_t$ satisfied equation~\eqref{eq:linear_kfe_eigenfunction_representation} for all $t$, then taking the time derivative yields
$$\frac{dg_t}{dt} = \sum_{i=1}^{N}{\varphi_{i,0} e^{\lambda_i t}\lambda_i b_i} = \sum_{i=1}^{N}{\varphi_{i,0} e^{\lambda_i t} \mathcal{L}^{KF} b_i} = \mathcal{L}^{KF}\underbrace{\left(\sum_{i=1}^{N}{\varphi_{i,0} e^{\lambda_i t} b_i}\right)}_{=g_t - g^{ss}\text{ by \eqref{eq:linear_kfe_eigenfunction_representation}}} = \mathcal{L}^{KF}g_t.$$
Consequently, $g_t$ also satisfied the KFE for all $t\geq 0$.
\end{proof}

\section{Robustness (Online Appendix)}\label{sec:robustness}

\begin{figure}[h]
    \centering
    \includegraphics[width = .5\textwidth]{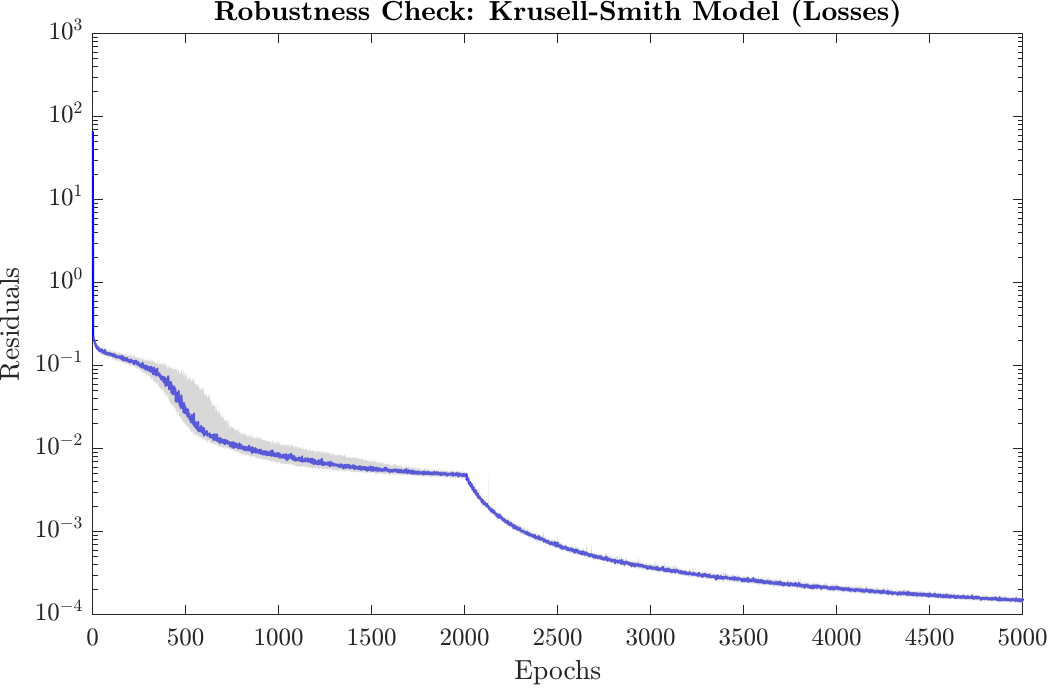}
    \caption{Training Loss vs Epochs Plots (Finite Agent Method). Shaded areas are 80\% confidence interval.}
    \label{fig:loss_path_FA_Robustness}
\end{figure}

\begin{figure}[h]
    \centering
    \includegraphics[width = .5\textwidth]{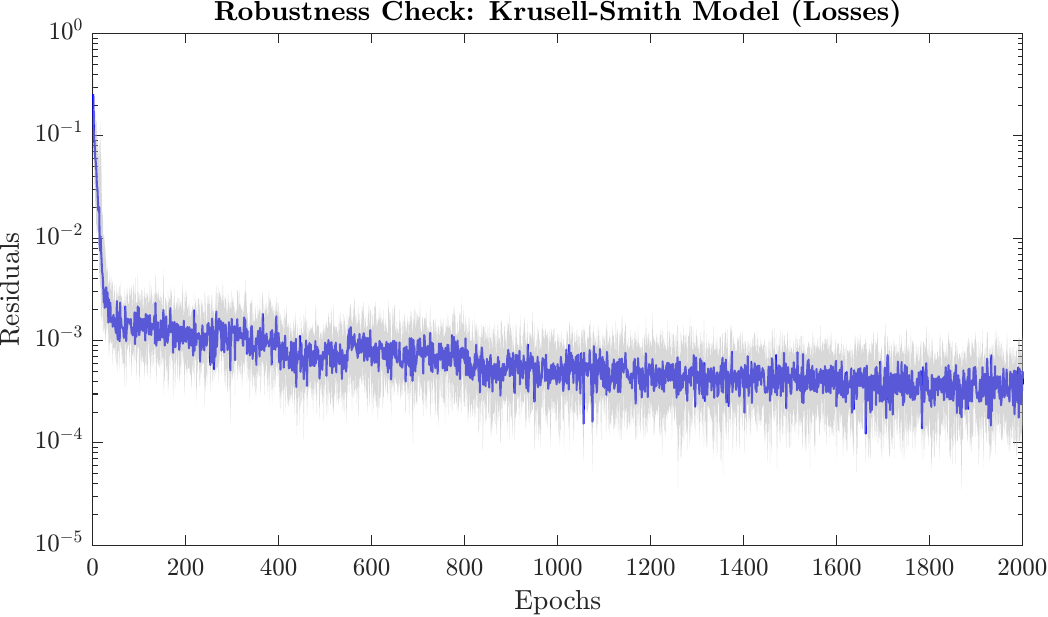}
    \caption{Training Loss vs Epochs Plots (Finite State Method). Shaded areas are 80\% confidence interval.}
    \label{fig:loss_path_FS_Robustness}
\end{figure}

\begin{figure}
    \centering
    \includegraphics[width = .5\textwidth]{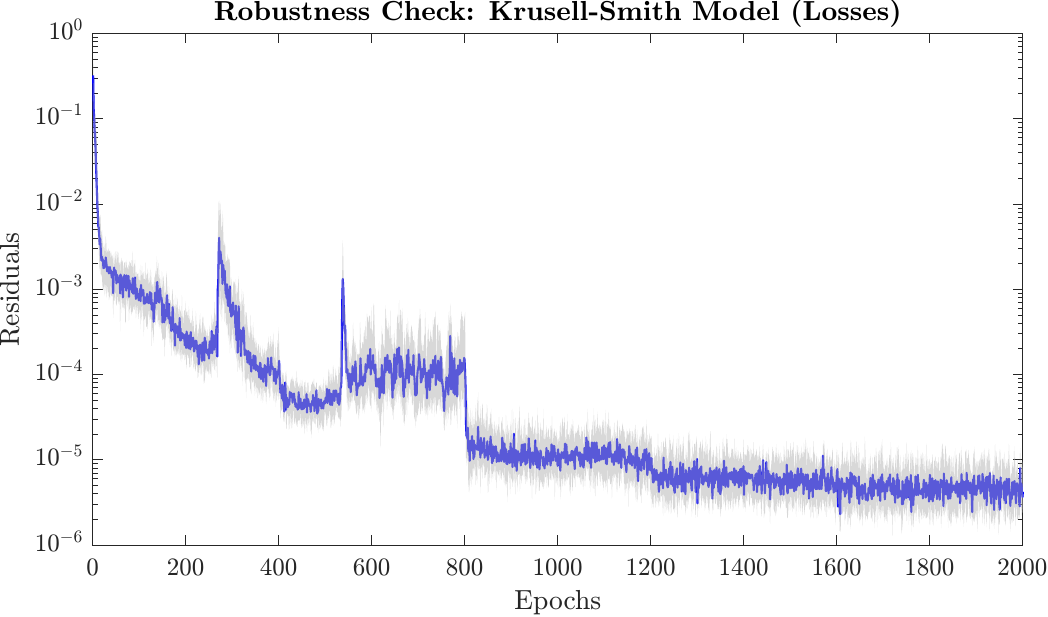}
    \caption{Training Loss vs Epochs Plots (Projection Method). Shaded areas are 80\% confidence interval.}
    \label{fig:loss_path_PROJ_Robustness}
\end{figure}

This section provides additional robustness checks for our solution techniques. 
For each method, we run the code 20 times with different random seeds and compare similarity of the neural nets at the end of the training. \\

\noindent\textit{Finite Agent Approach}. Figure \ref{fig:loss_path_FA_Robustness} shows the training losses (i.e., the residual of the PDE over distributions sampled during training) as a function of the training epoch. Each epoch corresponds to 16 steps of stochastic gradient descent. Figure \ref{fig:rstd_policy_FA_Robustness} shows the relative standard deviation for $dV$ and consumption as a function of $a$, at each of the two possible values of labor $l$.
We generate the ergodic distribution of $(z,g)$ by simulation.
We then draw 1000 samples of 40 agents from the ergodic distribution conditioned on $z = 0$.  
We use each trained neural network to evaluate $dV$ and $C$ (on a grid of 21 points between $a = 0$ and $a = 10$) for each 40 agent sample and then take the average. \\

\noindent\textit{Discrete State Approach}. Figure~\ref{fig:loss_path_FS_Robustness} shows the training loss (i.e., the residual of the PDE over distributions sampled during training) as a function of the training epoch. Here, one epoch corresponds to 20 steps of SGD. Figure~\ref{fig:rstd_policy_FS_Robustness} shows the relative standard deviation for $dV$ and the consumption as functions of $a$, in each of the two possible values of $y$. Again, we used a grid of $21$ points between $a=0$ and $a=10$. We evaluate $dV$ and consumption on the stochastic steady state for $g$ based on the finite difference code of \cite{Fernandez-Villaverde2018}.\\

\noindent\textit{Projection Approach}.
Figure~\ref{fig:loss_path_PROJ_Robustness} shows the training loss as a function of the training epoch. Here, one epoch corresponds to 20 steps of SGD. Figure~\ref{fig:rstd_policy_PROJ_Robustness} shows the relative standard deviation for $dV$ and the consumption as functions of $a$, in each of the two possible values of $y$. We used a grid of $21$ points between $a=0$ and $a=10$. We evaluate $dV$ and consumption on the same distribution $g$ as for the discrete state approach, except that we first project it onto the basis functions before we feed the coefficients into the neural network.

\begin{figure}
    \centering
    \includegraphics[width =\textwidth]{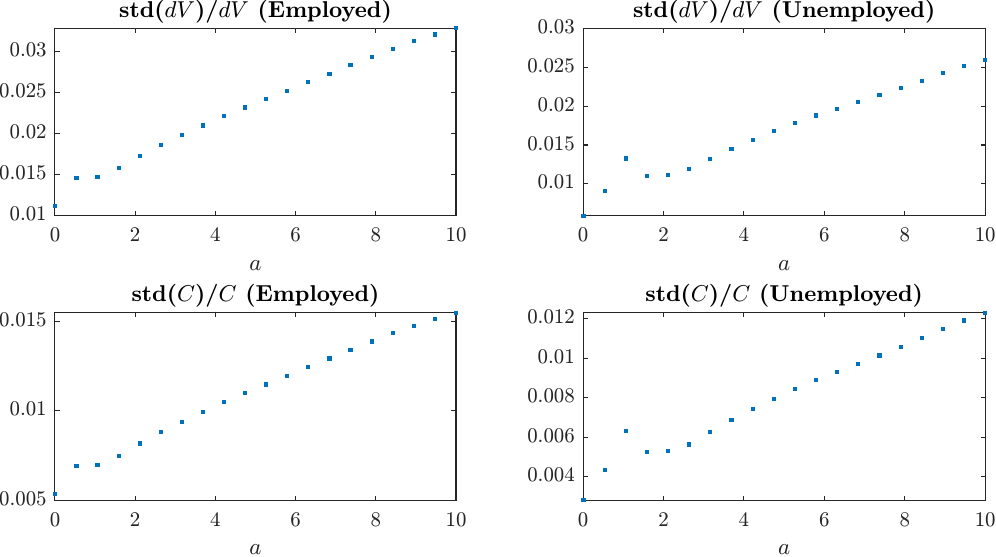}
    \caption{%
    Relative error (one standard deviation divided by mean) for value function and policy function for the finite agent approach. }
    \label{fig:rstd_policy_FA_Robustness}
\end{figure}

\begin{figure}
    \centering
    \includegraphics[width =\textwidth]{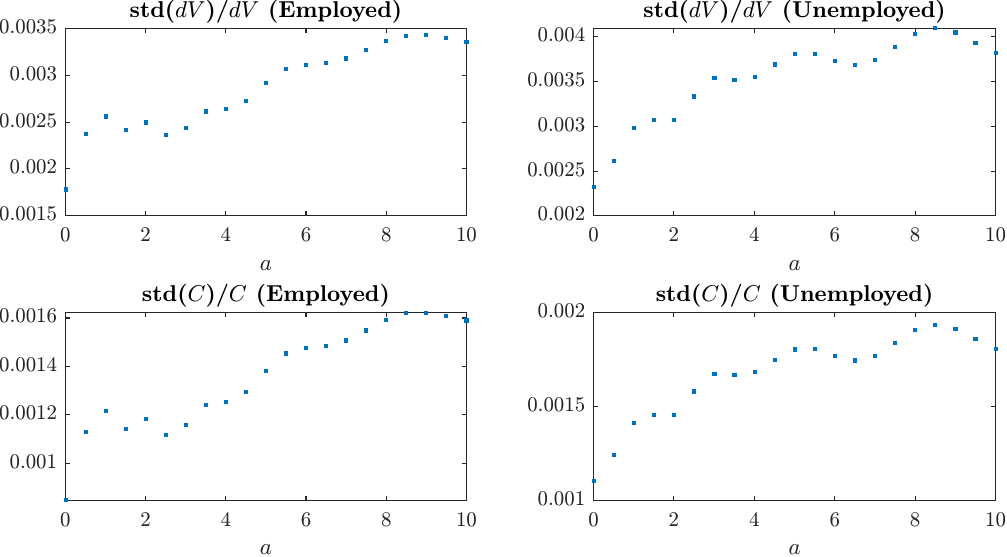}
    \caption{%
    Relative error (one standard deviation divided by mean) for value function and policy function for the finite state approach. }
    \label{fig:rstd_policy_FS_Robustness}
\end{figure}

\begin{figure}
    \centering
    \includegraphics[width =\textwidth]{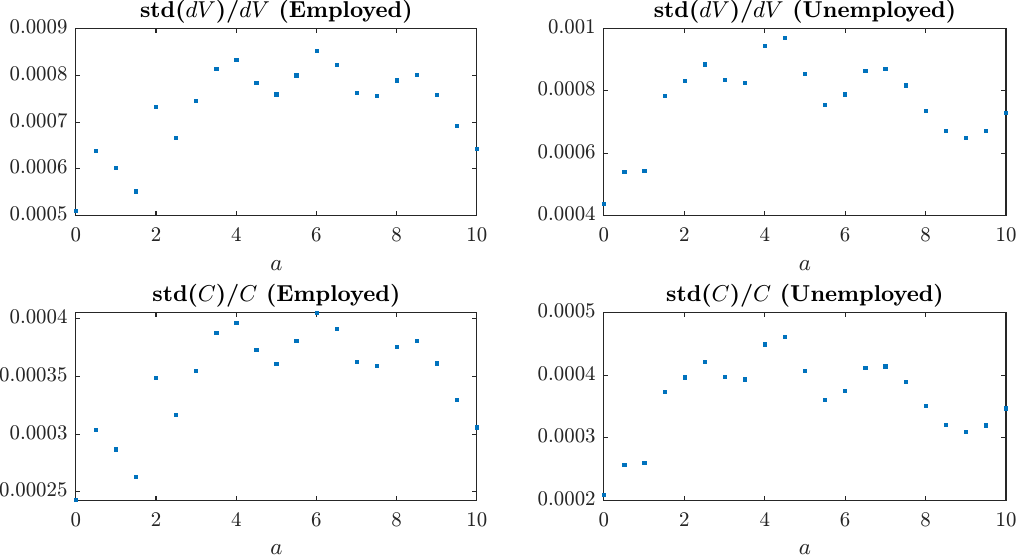}
    \caption{%
    Relative error (one standard deviation divided by mean) for value function and policy function for the projection approach.
    }
    \label{fig:rstd_policy_PROJ_Robustness}
\end{figure}

\section{Additional Working For Section \ref{sec:add_examples} (Online Appendix)}

\subsection{Firm Heterogeneity and Capital Adjustment Costs} \label{asubsec:firms}

\begin{proof}[Proof of Proposition \ref{prop:master_eqn:firms}]
    \noindent\underline{Household optimization:} The representative household chooses consumption, $C_t$, labour supply, $L_t$, and the wealth invested in firm $i$ equity, $E_t^i$, to solve:
    \begin{align}
        \max_{C_t, L_t, \{E^i\}} \mathbb E \int_0^\infty e^{-\rho t} U(C_t, L_t) dt, \quad \text{where} \quad U(C_t, L_t) = \frac{1}{1-\gamma} \left(C_t- \chi \frac{L_t^{1+\varphi}}{1+\varphi}\right)^{1-\gamma}
    \end{align}
    and where $\rho$ is the continuous time discount factor, $\sigma$ is the coefficient of relative risk aversion, $\chi$ determines the disutility of labor, $\varphi$ is the Frisch elasticity of labor supply, and the household is subject to the budget constraint:
    \begin{align}
        {dA_t} = - C_t dt + w_tL_t dt + \left[\int_{i}E_{t}^i\left(\frac{\pi_{t}^idt + dp_{t}^i}{p_{t}^i}\right)di\right], \quad \text{where }\int_iE_t^i di =A_t.
    \end{align}
    The representative household's HJB equation can be written as:
    \begin{align}
        \rho V(A,z,g) =&\max_{C,L,E_i}\Big\{U(C,L) + \partial_AV(A,z,g)\mu_A + \frac{1}{2}\partial_{AA}V(A,z,g)(\sigma_A)^2\\
        &+\partial_zV(A,z,g)\mu_z + \frac{1}{2}\partial_{zz}V(A,z,g)(\sigma_z)^2 + \partial_{Az}V(A,z,g)\sigma_z\sigma_A\\
        &+\sum_{\epsilon=\epsilon_L,\epsilon_H}\int \partial_g V(A,z,g) \mu_g(k,\epsilon)dk\Big\},
    \end{align}
    where:
    \begin{align}
        \mu_A = &- C + wL +  \left[\int_{i} E^i\left(\frac{\pi^i+ \mu_{p_i}}{p^i}\right)di\right],\\
        \sigma_A = & \left[\int_{i}E^i\left(\frac{ \sigma_{p_i}}{p^i}\right)di\right].
    \end{align}
    Household optimal consumption decision, labor supply decision and portfolio choice give:
    \begin{align}
        \partial_C U(C,L) & = \partial_A V(A,z,g) = \Lambda,\\
        \partial_L U(C,L)& = - \partial_A V(A,z,g) w,\\
        \partial_A V(A,z,g) \left[(\pi^i+\mu_{p_i})-(\pi^j+\mu_{p_j})\right] &= -(\sigma_{p_i}-\sigma_{p_j})\left(\partial_{AA} V(A,z,g)\sigma_A + \partial_{Az}V(A,z,g) \sigma_z\right).
    \end{align}
    
    The third equation implies the standard asset evaluation condition:
    \begin{align}
        p^{i}_{t} = \int_t^{\infty}e^{-\rho (s-t)}\frac{\Lambda_s}{\Lambda_t}\pi^{i}_{s} ds \equiv V^i_t.
    \end{align}
    where $V^i_t$ is the value function of the firm at time $t$.\\
    
    \noindent\underline{Firm optimization:} Each firm takes the household stochastic discount factor $\Lambda_t$ and wage as given and chooses labor hiring and investment to maximize its market value $V^i_t$. 
    Let $x_t:= [k_t,\epsilon_t]$ be the idiosyncratic state variables, $c_t$ be an arbitrary control $\{l_t, n_t\}$ and assume that optimal control $\hat{c}$ exists. Suppose the firm chooses:
    \begin{align}
        c_s^* = \left\{
        \begin{aligned}
            c_s,&\ s\in[0,t+h]\\
            \hat{c}_s,&\ s\in(t+h,\infty)
        \end{aligned}
        \right.
    \end{align}
    Then we have:
    \begin{align}
        V_t(x) \geq& \mathbb{E}_t \left[\int_t^{\infty} e^{-\rho(s-t)}\frac{\Lambda_s}{\Lambda_t}\pi(x,c_s^*,z,g)ds\right]\\
        =&\mathbb{E}_{t}\left[\int_t^{t+h}e^{-\rho (s-t)}\frac{\Lambda_s}{\Lambda_t}\pi(x,c_s,z,g)ds + \frac{\Lambda_{t+h}}{\Lambda_t}e^{-\rho h}V_{t+h}(X_{t+h})\right].
    \end{align}
    where $\pi(x,c_s^*,z,g) := e^{z_t} \epsilon_t(k_t)^\theta (l_t)^\nu - w_t l_t - \left(n_t + \psi(n_t,k_t) \right)$ is firm profit.
    Rearranging gives:
    \begin{align}\label{eq:HJB_Lam_deriv}
        0\geq \mathbb{E}_t\left[\int_t^{t+h}e^{-\rho (s-t)}\frac{\Lambda_s}{\Lambda_t}\pi(c_s)ds\right] + \mathbb{E}_t\left[\frac{\Lambda_{t+h}}{\Lambda_t}e^{-\rho h}V_{t+h}(X_{t+h})-V_t(x)\right].
    \end{align}
    By It\^o's Lemma, we have:
    \begin{align}
        d\Lambda_t &= \left[\partial_z\Lambda\mu_z(z) + \frac{1}{2}\partial_{zz}\Lambda \sigma_z(z)^2+\sum_{\epsilon=\epsilon_L,\epsilon_H}\int_{k}\frac{\partial \Lambda}{\partial g}\mu_g(k,\epsilon)dk \right]dt + \partial_z\Lambda\sigma_z(z) dB_t^0\\
        &\equiv \mu_{\Lambda}(z,g)dt + \sigma_{\Lambda}(z,g)dB_t^0\\
        dV_{t} &= \left[\partial_kV\mu_k(i,k) + \frac{1}{2}\partial_{zz}V \sigma_z(z)^2+\sum_{\epsilon=\epsilon_L,\epsilon_H}\int_{k}\frac{\partial \Lambda}{\partial g}\mu_g(k,\epsilon)dk\right]dt + \partial_z V\sigma_z(z)dB_t^0\\
        & \quad + (V(k,\tilde{\epsilon},z,g) - V(k,\tilde{\epsilon},z,g)dJ_t^i\\
        & \equiv \mu_{V}(k,\epsilon,z,g)dt + \sigma_{V}(k,\epsilon,z,g)dB_t^0 + j_V(k,\epsilon,z,g)dJ_t^i
    \end{align}
    and so:
    \begin{align}
        d(e^{-\rho t}\Lambda_t V_t) = e^{-\rho t}[(-\rho \Lambda_tV_t+\mu_\Lambda V_t + \mu_V \Lambda_t + \sigma_\Lambda \sigma_V) dt + (\sigma_\Lambda V_t +\sigma_V\Lambda_t)dB_t^0 +  j_V\Lambda_tdJ_t^i].
    \end{align}
    So the expected difference is:
    \begin{align}
        \mathbb{E}_t{}&\left[\frac{\Lambda_{t+h}}{\Lambda_t}e^{-\rho h}V_{t+h}(X_{t+h})-V_t(x)\right] \\
        =& \frac{1}{\Lambda_t}\int_t^{t+h} e^{-\rho (s-t)}\Big[-\rho \Lambda_s V_s + \mu_\Lambda V +  \partial_z V\partial_z \Lambda(\sigma_z(z))^2 + \lambda_i (V(k,\Tilde{\epsilon},z,g)-V(k,\epsilon,z,g))\\
        &+\left(\partial_kV\mu_k(k,\epsilon,z,g) +\partial_zV\mu_z(z)+ \frac{1}{2}\partial_{zz}V \sigma_z(z)^2+\sum_{\epsilon=\epsilon_L,\epsilon_H}\int_{k}\frac{\partial \Lambda}{\partial g}\mu_g(k,\epsilon)dk\right)\Lambda_s\Big]ds
    \end{align}
    Substitute this expression into \eqref{eq:HJB_Lam_deriv}, and devide by $h$ to get:
    \begin{align}
        0\geq \mathbb{E}_t{}&\Big[\frac{1}{h}\int_{t}^{t+h}e^{-\rho(s-t)} \frac{\Lambda_s}{\Lambda_t}\pi(x,c_s,z,g)ds \\
        {}&+ \frac{1}{h\Lambda_t}\int_t^{t+h} e^{-\rho (s-t)}\Big[-\rho \Lambda_s V_s + \mu_\Lambda V +  \partial_z V\partial_z \Lambda(\sigma_z(z))^2\\
        &+\lambda_i (V(k,\Tilde{\epsilon},z,g)-V(k,\epsilon,z,g))\\
        & + \left(\partial_kV\mu_k(k,\epsilon,z,g) +\partial_zV\mu_z(z)+ \frac{1}{2}\partial_{zz}V \sigma_z(z)^2+\sum_{\epsilon=\epsilon_L,\epsilon_H}\int_{k}\frac{\partial \Lambda}{\partial g}\mu_g(k,\epsilon)dk\right)\Lambda_s\Big]ds\Big]
    \end{align}
    Taking the limit $h\rightarrow0$ and noting that the inequality becomes equality for optimal control $c=(n,l)$ gives the following:
    \begin{align}
        0 = &\max_{n,l}\Big\{- \left(\rho -{\mu_{\Lambda}(z,g)}/{\Lambda(z,g)}\right)V(k,\epsilon,z,g)+ \pi(n,k,\epsilon,l
        ) \\
        &+ \partial_kV\mu_k(n,k)+\partial_zV\mu_z(z)+ \frac{1}{2}\partial_{zz}V \sigma_z(z)^2 
        +\lambda_i (V(k,\Tilde{\epsilon},z,g)-V(k,\epsilon,z,g))\\
        &
        + \frac{1}{\Lambda} \partial_z \Lambda\partial_z V(\sigma_z(z))^2
        +\sum_{\epsilon=\epsilon_L,\epsilon_H}\int\frac{\partial V}{\partial g}\mu_g(k,\epsilon)dk\Big\}.
    \end{align}
    Taking first order conditions we have:
    \begin{align}
        n^* ={}& \frac{k}{\chi_1}\left(\frac{\partial V}{\partial k} - 1\right) &
        l^* ={}& \left(\frac{w}{\nu z \epsilon k^{\theta}}\right)^{\frac{1}{\nu -1}}
    \end{align}
    Substituting in the optimal labor hiring decision $l^*$ and investment $n^*$, and all equilibrium conditions will give the stated master equation.\\
    
    \noindent\underline{Derivation for $\mu_{\Lambda}(z,g)$}. 
    In order to take the master equation to the computer, we need to derive an expression for $\mu_{\Lambda}$ through repeated application of It\^o's Lemma.
    We start by defining: 
    \begin{align}
        f(z,g) := C(z,g) - \chi \frac{L(z,g)^{1+\varphi}}{1+\varphi},
    \end{align}
    so that $\Lambda = f^{-\gamma}$.
    Applying It\^o's Lemma to $\Lambda = f^{-\gamma}$ gives:
    \begin{align}
        \mu_{\Lambda}(z,g) &=-\gamma f(z,g)^{-\gamma-1}\mu_f(z,g)+ \frac{1}{2}\gamma(\gamma+1)f(z,g)^{-\gamma-2}(\sigma_f(z,g))^2,\\
        \sigma_\Lambda(z,g) &= -\gamma f(z,g)^{-\gamma-1}\sigma_f(z,g)
    \end{align}
    Now, we can derive the expression for $\mu_f$. (For notational convenience, we drop the explicit dependence on $(z,g)$ for the remainder of the section.)
    From the goods market clearing condition, we know the consumption level is:
    \begin{align}
        C = \int_i \left[ z\epsilon^i (k^i)^\theta (l^i)^\nu - \left(n^i + \frac{\chi_1}{2}\frac{(n^i)^2}{k^i}\right) \right]di := \int_i c^i di.
    \end{align}
    So, by It\^o's lemma we have:
    \begin{align}
        \mu_f = \int_i \mu_{c_i} di- \chi L^{\varphi} \mu_{L} - \frac{\varphi\chi}{2} L^{\varphi-1}(\sigma_L)^2= \int_i\left(\mu_{c_i} - \chi L^{\varphi} \mu_{l_i}\right)di- \frac{\varphi\chi}{2} L^{\varphi-1}(\sigma_L)^2,
    \end{align}
    where we used the labor market clearing condition that $L = \int_i l^idi, \text{ and } \mu_L = \int_i \mu_{l_i}di$.
    The drift $\mu_{c_i}$ is given by:
    \begin{align}
        \mu_{c_i} =
        {}& \epsilon^i (k^i)^\theta (l^i)^\nu\mu_z +  \theta z\epsilon^i (k^i)^{\theta-1} (l^i)^\nu\mu_{k_i} + \nu z\epsilon^i (k^i)^\theta (l^i)^{\nu-1} \mu_{l_i} \\
        {}& +\frac{\nu(\nu-1)}{2}z\epsilon^i (k^i)^\theta (l^i)^{\nu-2}(\sigma_{l_i})^2 - \left[\mu_{n_i}(1+\frac{\chi_1 n^i}{k^i}) + \frac{\chi_1}{2k^i}(\sigma_{n_i})^2- \frac{\chi_1}{2}\frac{(n^i)^2}{(k^i)^2}\mu_{k_i}\right].
    \end{align}
    We can simplify this expression by applying the labor market clearing condition $w = \chi L^\varphi = \nu z\epsilon^i (k^i)^\theta (l^i)^{\nu-1}$
    to cancel out terms $\nu z\epsilon^i(k^i)^\theta (l^i)^{\nu-1}  \mu_{l_i}\text{ and } \chi L^\varphi\mu_{l_i}$, for every $i$. Additionally, we can merge terms $\frac{\nu(\nu-1)}{2}(\sigma_{n_i})^2$ with $\frac{\varphi\chi}{2}N^{\varphi+1}(\sigma_L)^2$. We can also apply It\^o's lemma to the labor market clearing condition to get $\frac{\sigma_{l_i}}{l^i} = \frac{\sigma_{L}}{L} = \frac{\sigma_z/z}{\varphi+1-\nu}$,
    which means that:
    \begin{align}
        &\int_i z\epsilon^i(k^i)^\theta (l^i)^\nu\frac{\nu(\nu-1)}{2}(\sigma_{l_i}/l^i)^2di - \frac{\varphi\chi}{2}L^{\varphi+1}(\sigma_L)^2=- \frac{Lw (\sigma_z)^2}{2(1+\varphi-\nu)z^2}.
    \end{align}
    Combining these simplifications gives that:
    \begin{align}
        \mu_f =& \int_i\left((\epsilon^i (k^i)^\theta (l^i)^\nu\mu_{z}+ z\epsilon^i (k^i)^{\theta-1} (l^i)^\nu \mu_{k_i}) - \left[\mu_{n_i}\frac{\partial V^i}{\partial k^i}+ \frac{\chi_1}{2k^i}(\sigma_{n_i})^2- \frac{\chi_1}{2}\frac{(n^i)^2}{(k^i)^2}\mu_{k}\right]\right)di\\
        &- \frac{Lw (\sigma_z)^2}{2(1+\varphi-\nu)z^2}.
    \end{align}
    The remaining terms to be specified are $\mu_{n_i}$ and $\sigma_{n_i}$. 
    To get these expression we apply It\^o's lemma to the optimal investment decision $n^i= \frac{k_i}{\chi_1}(\partial_kV^i -1)$, which implies:
    \begin{align}
        {\mu}_{n_i} &= \mu_{k_i}\left[\frac{1}{\chi_1} \left(\frac{\partial V^i}{\partial k^i}-1\right) + \frac{k_i}{\chi_1}\frac{\partial^2 V^i}{\partial (k^i)^2}\right] + \frac{k^i}{\chi_1}\frac{\partial^2 V^i}{\partial k^i\partial z}\mu_z + \frac{1}{2}\frac{k^i}{\chi_1}\frac{\partial^3V^i}{\partial k^i \partial z^2}(\sigma_z)^2\\
        \sigma_{n_i} &= \frac{k^i}{\chi_1}\frac{\partial^2V^i}{\partial k^i\partial z}\sigma_z
    \end{align}
    Finally, the expression for $\sigma_f$ is:
    \begin{align}
        \sigma_f =  \epsilon^i (k^i)^\theta (l^i)^\nu \sigma_z - \sigma_{n_i}\frac{\partial V^i}{\partial k^i}.
    \end{align}
\end{proof}

\renewcommand{\arraystretch}{.9}
\begin{table}[H]
\centering
\begin{tabular}{lcc}
\hline
Parameter & Symbol & Value  \\
\hline
Capital share & $\theta$ & $0.21$  \\
Labor share & $\nu$ & $0.64$  \\
Quarterly depreciation & $\delta$ & $0.025$ \\
Risk aversion & $\gamma$ & $1.0$ \\
Quarterly discount rate & $\rho$ & $1.25\%$ \\
Mean TFP & $\overline{Z}$ & $0.0$ \\
Reversion rate & $\eta$ & $1.0$ \\
Adjustment cost & $\chi_1$ & $2.0$ \\
Labor disutility & $\chi$ & $2.21$ \\
Frischer elasticity & $\varphi$ & $0.5$ \\
Volatility of TFP & $\sigma_z$ & $0.007$ \\
Transition rate (L to H) & $\lambda_L$ & $0.25$ \\
Transition rate (H to L) & $\lambda_H$ & $0.25$ \\
Low labor productivity & $\epsilon_L$ & $0.9$ \\
High labor productivity & $\epsilon_H$ & $1.1$ \\
Minimum of capital & $k_{min}$ & $1.0$ \\
Maximum of capital & $k_{max}$ & $6.0$ \\
Maximum TFP  & $z_{max}$ & $0.04$ \\
Minimum TFP  & $z_{min}$ & $-0.04$ \\
\hline
\end{tabular}
\caption{Parameters for Firm Dynamics Model} \label{tab:firm_dynamics}
\end{table}

\subsubsection{Implementation Details of Firm Dynamics Model}

Our economic parameters are shown in Table \ref{tab:firm_dynamics}. \\

\noindent\textit{Parameterization and Master Equation:} We parameterize the marginal value of capital $\partial_kV^i$ by a neural network with 5 layers and 64 neurons per layers, a tanh activation function between layers and no activation at the output level. We take derivatives w.r.t $k$ to the master equation stated in Proposition \ref{prop:master_eqn:firms} to reduce the number of auto-derivatives to be taken as the marginal value of capital $\partial_kV^i$ and its derivatives are the only terms that appear in the dynamics of $\Lambda$.\\

\noindent\textit{Sampling and Training:} We sample points of firm $i$'s individual capital level uniformly from $[k_{min},k_{max}]$. We sample points of the distribution $\{k^j\}$ in two steps: first we sample each $k^j$ uniformly from $[k_{min},k_{max}]$; then we sample a equilibrium wage shifter $\Delta w$ to move the distribution such that the equilibrium wage is updated as $w = (1+\Delta_w)w$, where $\Delta_w$ is drawn uniformly from $[-0.1, 0.1]$. This step is similar as the moment sampling described in section \ref{sec:solution:algorithm:sampling}. The loss function is constructed as the master equation loss plus the ``shape constraint'' $\partial_{kk}V^i>0$ with equal weights. The training part contains two steps:
\begin{enumerate}
    \item Training the neural network without $\mu_\Lambda$ in the master equation for 2000 epochs. The purpose of this step is to omit additional feedback loops between consumption and investment and get a downward sloping firm's marginal value on capital $\partial_kV$.
    \item Training the neural network with the full master equation.
\end{enumerate}
The training loss decay plot is available in Figure \ref{fig:loss_path_FAKH}.

\begin{figure}[H]
    \centering
    \includegraphics[width = .6\textwidth]{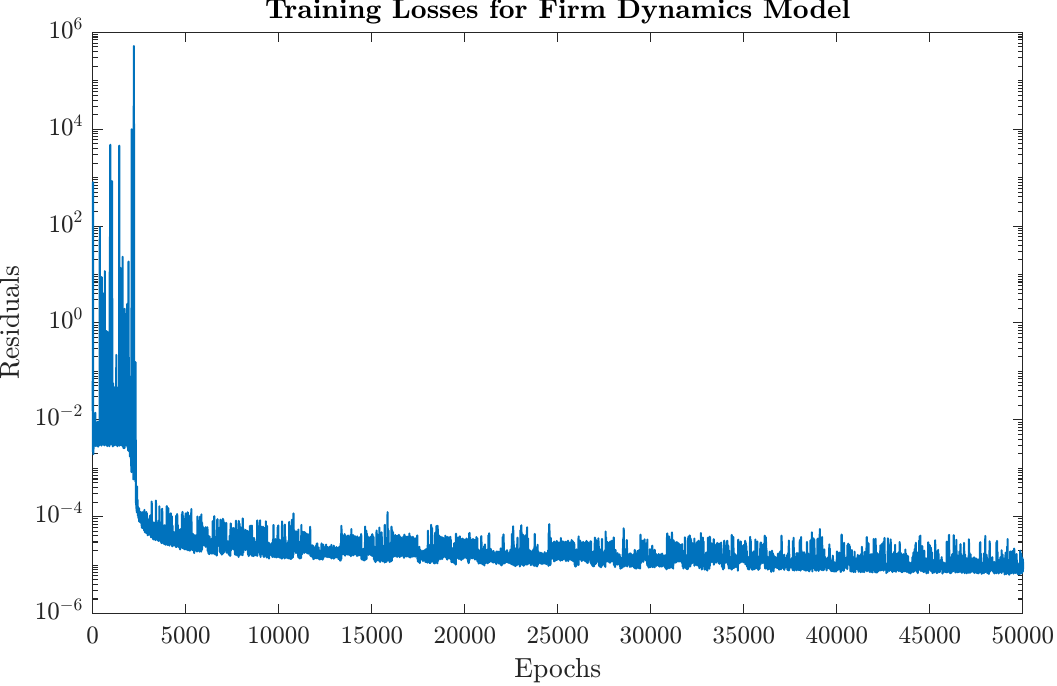}
    \caption{\small Training losses plot for Firm Dynamics model.}
    \label{fig:loss_path_FAKH}
\end{figure}

\subsection{Spatial Model} \label{asubsec:spatial}

\subsubsection{Details on the Model Solution} \label{asubsec:spatial:model_solution}

\noindent\emph{Worker decision problem:}
Formally, the decision problem of worker $i$ is to choose processes $\mathsf{c}^i = \{ c_t^i : t \ge 0 \}$ and $\mathsf{\varsigma}^i = \{ \varsigma_t^i : t \ge 0 \}$ to maximize
$$\mathbb{E}_0 \left[ \int_0^{\infty} e^{-\rho t} \left( u(c_t^i) dt + v_t^i(j_t^i + \varsigma^i_t) dJ_t^i\right)\right], \quad \text{where} \quad u(c_t^i) = \frac{(c_t^i)^{1-\gamma}}{1-\gamma} $$
subject to the set of static budget constraints
$$\forall t\geq 0 : \quad c_t^i = w_{j_t^i,t}$$
and the location evolution equation
$$dj_t^i = \varsigma^i_t dJ_t^i.$$
Here, $w_{j,t}$ denotes the market wage in location $j$ at time $t$, $dJ_t^i$ is an idiosyncratic Poisson shock with constant arrival rate $\mu$ that represents the arrival of moving opportunity shocks, $\varsigma^i_t\in \{1-j_t^i,...,J-j_t^i\}$ is the moving decision conditional on a moving opportunity shock arriving at time $t$, and $v_t^i(j^\prime)$ is an additive utility shifter conditional on a moving opportunity shock and the subsequent choice to move to new location $j^\prime$. This shifter captures location preference shocks and moving costs as follows:
$$v_t^i(j^\prime) := \xi^i_{j^\prime,t}-\tau_{j^i_t,j^\prime},$$
where $\{\xi^i_{j^\prime,t}:j\in\{1,...,J\},t\geq 0\}$ denotes a collection of independent idiosyncratic Gumbel-distributed random variables with mean zero and inverse scale parameter $\nu$ and $(\tau_{j,j^\prime}:j,j^\prime \in \{1,...,J\})$ is the matrix of moving disutility.\footnote{The $\xi_{j^\prime,t}^i$ should be interpreted as preference shocks only at jump times of $dJ_t^i$. At all other times, the $\xi_{j^\prime,t}^i$ play no role for the decision problem because they only enter utility conditional on $dJ_t^i\neq 0$. %
}

Let $V(j,z,g)$ denote the value function of a worker in location $j$ at aggregate state $(z,g)$ in a period without a moving opportunity shock or right after having moved. $V$ and the optimal consumption choice $c$ solve the HJBE
\begin{align}
 0 = \max_{c\; :\; c=w_j(z,g)}\Big\{{}&- \rho V(j,z,g) + u(c) + (\cL_x + \cL_z + \cL_g^{\tilde{\mu}_g}) V(j,z,g)\Big\},\label{eq:spatial_model:hjb}
\end{align}
where the operators $\cL_x$, $\cL_z$, and $\cL_g^{\tilde{\mu}_g}$ are defined as follows:
\begin{align}
    \cL_x V(j,z,g) &= \mu\left(\mathbb{E}^{\xi_1,...,\xi_J}\left[\max_{j^\prime \in \{1,...,J\}}\{V(j^\prime,z,g)  - \tau_{j,j^\prime} + \xi_{j^\prime}\}\right]-V(j,z,g)\right)\\
    \cL_z V(j,z,g) &= \partial_z V(j,z,g)\eta(\overline{z} - z) + \frac{1}{2} \sigma^2 \partial_{zz} V(j,z,g)\label{eq:spatial_model:operator_L_z}\\
    \cL_g^{\tilde{\mu}_g} V(j,z,g) &= \sum_{j^{\prime}=1}^J\partial_{g(j^{\prime})}V(j,z,g)\tilde{\mu}_{g}(j^\prime,z,g).\label{eq:spatial_model:operator_L_g}
\end{align}
Here, the operator $\cL_x$ contains the location choice problem conditional on receiving a moving opportunity and $\mathbb{E}^{\xi_1,...,\xi_J}$ denotes the expectation over the $J$ independent location preference shock draws $\xi_1,...,\xi_J$.\footnote{
        We use the conditional destination choice $j^\prime$ in place of $\varsigma$ for the control in the location choice problem.The two are related via $j^\prime = j + \varsigma$.}
We compute this expectation and the conditional choice probabilities by using several well-known properties of the Gumbel distribution. First, adding a non-random number to a Gumbel distribution shifts the mean (location parameter) but results again in a Gumbel distribution with the same scale parameter. Consequently, the $j^\prime$-th argument in the maximum operator is Gumbel-distributed with inverse scale parameter $\nu$ and mean $V(j^\prime,z,g) - \tau_{j,j^\prime}$. Second, the maximum of these finitely many independent Gumbel-distributed random variables (with identical scale parameter) is again Gumbel-distributed with mean
$$\mathbb{E}^{\xi_1,...,\xi_J}\left[\max_{j^\prime \in \{1,...,J\}}\{V(j^\prime,z,g)  - \tau_{j,j^\prime} + \xi_{j^\prime}\}\right] = \frac{1}{\nu}\log \left(\sum_{j^\prime=1}^{J}e^{\nu(V(j^\prime,z,g)-\tau_{j,j^\prime})}\right).$$
Plugging this result into the expression for the operator $\cL_x$ allows us to eliminate the location choice problem from the HJBE by writing the operator as
\begin{equation}\label{eq:spatial_model:operator_L_x}
\cL_x V(j,z,g) = \mu\left(\frac{1}{\nu}\log \left(\sum_{j^\prime=1}^{J}e^{\nu(V(j^\prime,z,g)-\tau_{j,j^\prime})}\right)-V(j,z,g)\right).
\end{equation}
Third, we apply the property of the Gumbel distribution that the ex-ante probability of the $j^\prime$-th draw being the largest is given by
$$\pi_{j,j^\prime}(V(\cdot,z,g)) = \mathbb{P}^{\xi_1,...,\xi_J}\left(\arg\max_{k \in\{1,...,J\}} \{V(k,z,g)  - \tau_{j,k} + \xi_{k}\} = j^\prime \right) = \frac{e^{\nu(V(j^\prime,z,g)  - \tau_{j,j^\prime})}}{\sum_{k=1}^{J}e^{\nu(V(k,z,g)  - \tau_{j,k})}}.$$
This formula tells us the choice probabilities before preference shocks are drawn and coincides with the conditional choice probabilities stated in the main text.

We remark that the structure of the HJBE in this model does not fit into our generic model setup from Section~\ref{sec:genericModel} before the location choice has been made. However, it does fit into that framework after the location choice is substituted in up to a straightforward extension: instead of a single idiosyncratic Poisson jump process we need $J$ Poisson processes for destinations $j^\prime \in \{1,...,J\}$ with arrival rates $\lambda \pi_{j,j^\prime}(V(\cdot,z,g))$. \\

\noindent\emph{Firm optimization and labor market clearing:}
The labor demand of the representative firm in each location $j$ implies that, in equilibrium,
$$w_{j,t} = (1-\alpha) \exp(\beta_j  + \chi_j z_t) L_{j,t}^{-\alpha},$$
where $L_{j,t}$ is total labor input at location $j$. Each worker inelastically supplies one unit of labor, so that the total labor supply at $j$ equals the mass of workers, $L_{j,t} = g_t(j)$. Hence,
$$Q(g_t,z_t) = \begin{bmatrix} w_{1,t} \\ w_{2,t} \\ \vdots \\ w_{J,t} \end{bmatrix} = \begin{bmatrix} (1-\alpha) \exp(\beta_1  + \chi_1 z_t) (g_t(1))^{-\alpha} \\ (1-\alpha) \exp(\beta_2  + \chi_2 z_t) (g_t(2))^{-\alpha} \\ \vdots \\ (1-\alpha) \exp(\beta_J  + \chi_J z_t) (g_t(J))^{-\alpha} \end{bmatrix}.$$

\noindent\emph{Kolmogorov Forward Equation:}
The KFE can be derived using the remark at the end of the discussion of the worker problem. Mathematically, after the optimal location choice, the evolution of the idiosyncratic state $j_t^i$ can be described by
$$dj_t^i = \sum_{j^\prime=1}^{J}(j^\prime-j_t^i) d\grave{J}_{j^\prime,t}^i,$$
where $\grave{J}_{1,t}^i$, $...$, $\grave{J}_{J,t}^i$ are independent Poisson processes with arrival rates $\mu \pi_{j_t^i,1}(V(\cdot,z_t,g_t))$, $...$, $\mu \pi_{j_t^i,J}(V(\cdot,z_t,g_t))$. Using this representation of the idiosyncratic state evolution and adapting the derivation of the KFE in Appendix~\ref{asec:genericModel} for multiple Poisson processes, we obtain
\begin{equation}\label{eq:spatial_model:kfe}
\mu_g(j, z, g) = \mu\left(\sum_{k=1}^J\pi_{k,j}(V(\cdot,z,g))g(k)-g(j)\right).
\end{equation}

\begin{proof}[Proof of Proposition~\ref{prop:spatial_model:master_equation}]

\noindent\underline{Master Equation:}
We obtain equation~\eqref{eq:ME:spatial_model} immediately from plugging in the optimal consumption choice $c^\ast(j,z,g)=w_j(z,g)$ into the HJBE~\eqref{eq:spatial_model:hjb} and showing that, once we impose belief consistency, $\tilde{\mu}_g = \mu_g$, the operators $\cL_x$, $\cL_z$, and $\cL_g^{\tilde{\mu}_g}$ in the latter equation are identical to the operators $\cL_x$, $\cL_z$, and $\cL_g$ stated in the proposition. One immediately verifies that this is the case from equations~\eqref{eq:spatial_model:operator_L_x}, \eqref{eq:spatial_model:operator_L_z}, and \eqref{eq:spatial_model:operator_L_g} derived previously.
Finally, the KFE is given by equation~\eqref{eq:spatial_model:kfe}, which is identical to the one stated in the proposition.

\noindent\underline{Relationship to \cite{bilal2021solving}:}
\citet[equation~(19)]{bilal2021solving} states the master equation (using our notation whenever concepts are identically defined)
\begin{align*}
\rho V(j,z,g) = U(j,z,g) &+ \mu\left(\frac{1}{\nu}\log \left(\sum_{j^\prime=1}^{J}e^{\nu(V(j^\prime,z,g)-\tau_{j,j^\prime})}\right)-V(j,z,g)\right)\\
&+ \partial_z V(j,z,g)\eta(\overline{z} - z) + \frac{1}{2} \sigma^2 \partial_{zz} V(j,z,g)\\
&+ \sum_{j^\prime=1}^{J} \partial_{g(j^\prime)}V(j,z,g)\mu\left(\sum_{k=1}^J\pi_{k,j^\prime}(V(\cdot,z,g))g(k)-g(j^\prime)\right),
\end{align*}
where $U(j,z,g)$ is a flow utility term. With the operators defined above, this equation can be written as follows:
$$\rho V(j,z,g) = U(i,z,g) + (\cL_x + \cL_z + \cL_g)V(j,z,g).$$
This is the same equation as equation~\eqref{prop:spatial_model:master_equation} if and only if
$$U(i,z,g) = u(w_j(z,g)) = u((1-\alpha)\exp(\beta_j  + \chi_j z_t)(g_t(j))^{-\alpha}).$$
\cite[equation~(45)]{bilal2021solving} defines the flow utility $U(i,z,g)$ as
$$U(i,z,g) = u(C_{0,j}e^{\zeta\chi_{j}^B z}(g(j))^{-\xi}),$$
where $\zeta,\xi \in(0,1)$ are (derived) parameters, the notation $\chi_j^B$ emphasizes that these are the $\chi_j$-parameters in Bilal's variant of the model (as opposed to ours) and $C_{0,j}$ is a time-invariant, location-specific consumption shifter that depends on both location-specific productivity and housing supply. Evidently, the utility flow terms are identical state by state if we choose parameters such that
$$\alpha = \xi,\qquad \chi_j = \zeta \chi_j^B,\qquad (1-\alpha)e^{\beta_j} = C_{0,j}. $$
Clearly, for any set of (derived) parameters $\zeta$, $\xi$, $\chi_j^B$, $C_{0,j}$ in Bilal's variant of the model, it is possible to find parameters $\alpha$, $\chi_j$, and $\beta_j$ in our variant to make these conditions hold. Note, furthermore, that none of the parameters $\alpha$, $\chi_j$, or $\beta_j$ appear elsewhere in the master equations, so that assigning them in this way does not prevent us from matching the remaining aspects of \cite{bilal2021solving}'s master equation.
\end{proof}

\subsubsection{Parameters}\label{asubsec:spatial:parameters}

All parameters that are not location-specific are summarized in Table~\ref{tab:spatial_model:parameters}. We choose $\gamma$, $\rho$, $\overline{Z}$, $\eta$, and $\sigma$ consistent with our parameterization of the KS model (compare Appendix~\ref{asec:implementation_details:parameters}). We choose $\mu$ and $\nu$ as in \cite{bilal2021solving} and $\alpha$ to achieve the same elasticity of the local wage to population changes as in \cite{bilal2021solving} utilizing the isomorphism between our model specification and theirs from Proposition~\ref{prop:spatial_model:master_equation}.\footnote{\cite{bilal2021solving} does not state all parameters but refers for the calibration to \cite{bilal2023anticipating}, which has a slightly different model. To the extent that parameter values cannot be inferred from \cite{bilal2021solving} directly, we try to match the corresponding parameter in \cite{bilal2023anticipating} instead.}

\begin{table}[h]
\centering
\begin{tabular}{lcc}
\hline
Parameter & Symbol & Value  \\
\hline
Number of locations & $J$ & $50$  \\
Profit share & $\alpha$ & $0.55$  \\
Risk aversion & $\gamma$ & $2.1$ \\
Discount rate & $\rho$ & $0.05\%$ \\
Mean TFP & $\overline{Z}$ & $0.00$ \\
Reversion rate & $\eta$ & $0.50$ \\
Volatility of TFP & $\sigma$ & $0.01$ \\
Moving rate & $\mu$ & $2.3$ \\
Preference shock parameter & $\nu$ & $0.48$ \\
Minimum TFP (for sampling)  & $z_{min}$ & $-0.04$ \\
Maximum TFP (for sampling)  & $z_{max}$ & $0.04$ \\
\hline
\end{tabular}
\caption{Location-independent parameters for spatial model\label{tab:spatial_model:parameters}}
\end{table}

We sample the location-specific parameters $\beta_j$, $\chi_j$, and $\tau_{j,j^\prime}$ randomly using the following sampling strategies:
\begin{itemize}
    \item We draw $\beta_j$ randomly from a distribution such that the equilibrium population distribution in the simplified steady-state model without common noise and with preference shock parameter $\nu\rightarrow \infty$ is truncated Pareto distributed over $[1,50]$ with shape parameter $1$. The choice of the Pareto distribution with shape parameter $1$ is motivated by the observation that city sizes appear to approximately satisfy ``Zipf's law'', but we use a truncation for numerical stability. We use the simplified model with $\nu\rightarrow \infty$ for this exercise because this special case has a closed-form solution and so does not require us to solve the model repeatedly in order to calibrate the $\beta_j$-distribution.\footnote{However, this does mean that the ultimate population distribution in our model solution will be less dispersed than a Pareto with shape parameter $1$ because preference shocks generate smoothing when $\nu<\infty$.}
    \item We draw the sensitivity parameters $\chi_j$ of locations to aggregate productivity by sampling uniformly from an exponential distribution with parameter $1$ (so that the mean is $1$). We then multiply the resulting $\chi_j$-draws by 0.7 reflecting the fact that the sensitivity of local consumption to productivity shocks is scaled down in \citet{bilal2021solving}'s model with housing (see proof of Proposition~\ref{prop:spatial_model:master_equation} for the correspondence between the models).
    \item We group the $50$ locations into four ``clusters'', one ``central'' cluster with $20$ locations and three ``periphery'' clusters with $10$ locations each. We set a baseline value for $\tau_{j,j^\prime}$ based on cluster membership and then randomly perturb these baseline values to generate a less regular pattern in moving costs. Regarding the baseline values ($\tau_{j,j^\prime}^0$), we make the following choices:
    \begin{itemize}
        \item Movements to the same location are free $\tau_{j,j}^0 = 0$ for all $j$
        \item Movements within a cluster have low baseline cost, $\tau_{j,j^\prime}^0 = 8.882\times 10^{-3}$. To provide a sense for the magnitude, in the steady-state model with $\nu\rightarrow \infty$ and no moving costs (which has a closed-form solution in which all workers' expected utility is identical), if one agent had to pay this moving cost after every shock arrival $dJ_t^i$ (but without changing the rest of the model, including prices), this agent's welfare would fall by the equivalent of a reduction of 2\% of consumption each period.
        \item Movements between the central and a periphery cluster have intermediate baseline cost, $\tau_{j,j^\prime}^0 = 4.8569\times 10^{-2}$. Again, to provide interpretation, the corresponding number of consumption-equivalent utility loss in the thought experiment described previously is here 10\%.
        \item Movements between two distinct periphery clusters have high baseline cost, $\tau_{j,j^\prime}^0 =  2.98030\times 10^{-1}$. The consumption-equivalent utility loss in the previous thought experiment is for this moving cost 40\%.
    \end{itemize}
    For the perturbation of baseline costs, we use in all cases
    $$\tau_{j,j^\prime} = \varepsilon_{j,j^\prime} \tau_{j,j^\prime}^0,\qquad 1\leq j\leq j^\prime \leq J,$$
    where $\varepsilon_{j,j^\prime}$ is a scaling factor that is uniformly distributed over $[0.5,1.5]$ (independent across $j,j^\prime$). The remaining entries of $\tau$ are defined by requiring that the $\tau$-matrix is symmetric.
\end{itemize}
The resulting vectors of sampled parameter values $\boldsymbol{\beta}=(\beta_1,...,\beta_{50})$ and $\boldsymbol{\chi}=(\chi_1,...,\chi_{50})$ are:
\begin{small}
\begin{align*}
 \boldsymbol{\beta} = - {}& (2.60, 2.23, 2.88, 2.69, 2.80,
       2.83, 2.77, 2.66, 2.62, 2.48,\\
       {}& \;2.60, 2.29, 2.76, 1.83, 2.86,
       2.31, 2.60, 2.46, 2.80, 2.76,\\
       {}& \;2.06, 1.29, 2.68, 2.27, 1.83,
       1.76, 2.83, 2.86 , 2.78, 1.83,\\
       {}& \;2.82, 2.60, 1.39, 2.48, 2.27,
       2.68, 2.28, 1.97, 2.87, 2.17,\\
       {}& \;1.02, 2.17, 2.71, 2.09, 2.82,
       2.57, 1.70, 2.70, 2.70, 2.81)\\
 \boldsymbol{\chi} = \phantom{+} {}& (2.76, 0.27, 1.09 , 0.93 , 0.50,
       2.05, 0.39, 1.34, 0.37, 0.25,\\
       {}& \;1.60, 0.62, 0.26, 0.62, 2.10,
       0.44, 0.29, 0.46, 0.04, 0.37,\\
       {}& \;0.07, 1.39, 1.38, 0.15, 0.65,
       1.26, 0.05, 0.74, 0.20 , 0.22,\\
       {}& \;0.09, 0.33, 0.20, 0.74, 0.92,
       0.08, 0.59, 0.03, 0.29, 0.33,\\
       {}& \;1.52, 0.04, 0.56, 0.38, 0.63,
       1.01, 0.07, 0.39 , 4.10, 0.34)
\end{align*}
\end{small}
In the interest of space, the 2500 sampled parameter values for $\tau_{jj^\prime}$ are omitted here.

\subsubsection{Implementation Details for Spatial Model}\label{asubsec:spatial:implementation}

\emph{Distribution Representation and Master Equation:} 
As argued in the main text, the discrete state space method is most appropriate here because the model's idiosyncratic state space is naturally discrete and because all the finite mass points $g(j)$ matter directly for some of the entries of $Q(z,g)$. Furthermore, because the number of locations is already finite, we do not approximate the distribution but simply choose $\hat{\varphi}=g\in [0,\infty)^J$, so that $\hat{V}=V$. The master equation for $\hat{V}$ used in the solution algorithm therefore corresponds exactly to the model's true master equation for $V$ stated in Proposition~\ref{prop:spatial_model:master_equation}. We note that this master equation also makes sense (mathematically and economically) when $g$ has total mass different from $1$ (i.e. is not a probability distribution). We exploit this fact in our sampling approach as described below.\\

\noindent\emph{Network Structure:} We parameterize the value function $V(j,z,g)$ with a neural network. We use a fully connected feed-forward neural network with $5$ layers and $64$ neurons per layer, a $\tanh$ activation function between layers, and no activation at the output level. We transform locations $j\in\{1,...,J\}$ into vectors in $\{0,1\}^J$ using one-hot encoding before feeding them into the neural network. We initialize the neural network so that $V$ matches the value function in the steady-state model with aggregate productivity fixed at $z=0$. This is done through a pre-training phase.\footnote{In the steady-state model, the master equation reduces to a finite-dimensional system of non-linear algebraic equations. We solve this system numerically with a Newton method. In the pre-training phase, we minimize a standard least-squares loss function between the neural network output and the pre-computed steady-state value function.}\\

\noindent\emph{Sampling:} 
We sample points of the form $(j,\hat{\varphi},z)$ by sampling the three dimensions independently as follows. We sample $j\in \{1,...,J\}$ and $z\in [z_{min},z_{max}]$ from a uniform distribution over their respective domains. For the distribution $\hat{\varphi}=g$, we combine two sampling schemes, (i) mixed steady state sampling and (ii) ergodic sampling, and we gradually increase the proportion of the sample from scheme (ii) from 0\% to 95\% during training. Sampling scheme (i) computes sample distributions $\hat{\varphi}$ in four steps:
\begin{enumerate}
 \item Draw a distribution $\hat{\varphi}_1$ as a random mixture from three baseline distributions, $g^{ss,z=0}$, $g^{ss,z=0.03}$, and $g^{ss,z=-0.03}$, where $g^{ss,z=\bar{z}}$ denotes the stationary distribution in the steady-state model with aggregate productivity fixed at $z=\bar{z}$. The weights on each of the three distributions are drawn uniformly from the $3$-simplex.
 \item Draw a distribution $\hat{\varphi}_2$ uniformly from the $J$-simplex.\footnote{To draw uniformly from the $J$-simplex, we use the well-known connection to a special case of the Dirichlet distribution and a common approach to draw from that distribution. Specifically, we draw uniformly from the $J$-simplex by drawing $J$ independent exponentially distributed random variables with parameter $1$ and then renormalize so that the resulting vector sums to $1$. } $\hat{\varphi}_2$ serves the purpose of introducing additional (random) noise into the distribution samples.
 \item Combine $\hat{\varphi}_1$ and $\hat{\varphi}_2$ by forming a convex combination with weight on the first distribution drawn uniformly from the interval $[0.95,1]$.
 \item Scale the resulting distribution by a uniform random number drawn from $[0.98,1.02]$. This last step perturbs the total mass of the distribution, which helps the neural network to learn the sign of the derivatives $\partial_{g(j)} V$.
\end{enumerate}
For the ergodic sampling (scheme (ii)), we start the simulation from an initial sample drawn from scheme (i). Because of the mass perturbation in step 4 and mass preservation in the KFE, also the total masses in our ergodic sample end up being uniformly distributed in $[0.98,1.02]$.\\

\noindent\emph{Loss Function:} We impose penalties to impose the shape constraints $\partial_z V > 0$ and $\partial_{g(j)} V < 0$ for all $j=1,...,J$ by choosing the shape error as follows:
\begin{align*}
\cE^s(\theta^n, S^{n}) :=& \frac{1}{|S^n|} \sum_{(j,z,\hat{\varphi})\in S^n} |\max\{-\partial_z\hat{V}(j,z,\hat{\varphi};\theta^n),0\} |^2\\
&+ \frac{1}{J|S^n|}\sum_{(j,z,\hat{\varphi})\in S^n}\sum_{j^\prime=1}^{J}|\max\{\partial_{g(j^\prime)}\hat{V}(j,z,\hat{\varphi};\theta^n),0\} |^2
\end{align*}
We combine this the equation residual error $\cE^e(\theta^n, S^{n})$ for the total error $\cE(\theta^n, S^{n})$ as described in Algorithm~\ref{alg:generic} using the weights $\kappa^e=1$ and $\kappa^s=0.2$.\\

\subsubsection{Training Losses}\label{asubsec:spatial:loss_decay}

The training loss decay plot is available in Figure \ref{fig:spatial:training_loss}.

\begin{figure}[hbtp]
    \centering
    \includegraphics[width=0.6\textwidth]{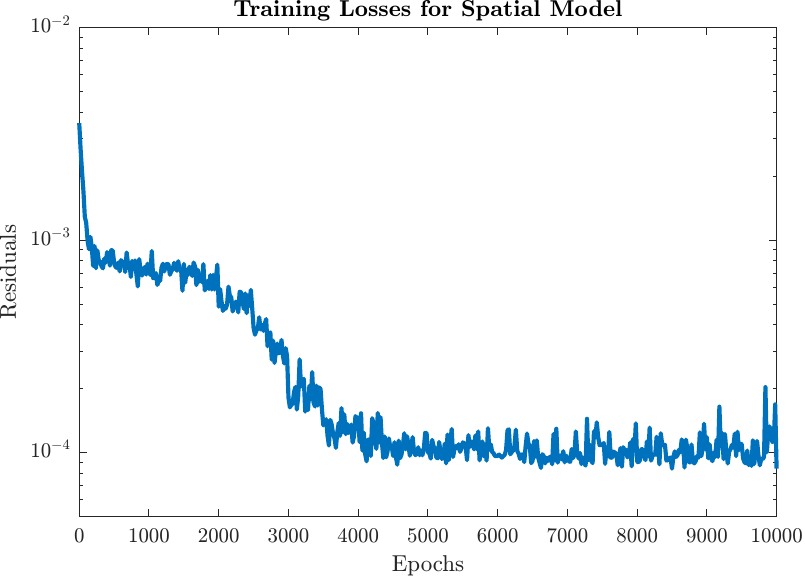}
    \caption{\small Training losses plot for the spatial model (pre-training not shown)}
    \label{fig:spatial:training_loss}
\end{figure}

\end{document}